\documentclass[sn-mathphys]{sn-jnl}

\jyear{2022}%

\theoremstyle{thmstyleone}%
\newtheorem{theorem}{Theorem}
\theoremstyle{thmstyletwo}%

\theoremstyle{thmstylethree}%

\renewcommand{\vec}[1] {\ensuremath{\boldsymbol{#1}}}

\usepackage{amsopn}

\def\R{{\mathbb R}}
\def\U{{\mathcal U}}
\def\Z{{\mathcal Z}}
\def\t{{\vec{\theta}}}
\def\u{{\vec{u}}}
\def\c{{\vec{c}}}
\def\z{{\vec{z}}}
\def\x{{\vec{x}}}
\def\e{{\vec{e}}}
\def\v{{\vec{v}}}
\def\w{{\vec{w}}}
\def\y{{\vec{y}}}

\def\b{{\vec{b}}}
\def\J{{\vec{J}}}
\def\c{{\vec{c}}}
\def\d{{\vec{\delta}}}

\def\B{{\vec{B}}}
\def\E{{\vec{E}}}
\def\M{{\vec{M}}}
\def\K{{\vec{K}}}
\def\G{{\vec{G}}}
\def\S{{\vec{S}}}
\def\N{{\vec{N}}}
\def\L{{\vec{L}}}
\def\C{{\vec{C}}}

\def\H{{\vec{H}}}

\def\Q{{\vec{Q}}}

\def\V{{\vec{V}}}

\def\I{{\vec{I}}}

\def\W{{\vec{W}}}
\def\T{{\vec{T}}}
\def\Y{{\vec{Y}}}
\def\F{{\vec{F}}}
\DeclareMathAlphabet\mathbfcal{OMS}{cmsy}{b}{n}

\usepackage{array}
\newcolumntype{P}[1]{>{\centering\arraybackslash}p{#1}}
\newcolumntype{M}[1]{>{\centering\arraybackslash}m{#1}}

\begin{document}

\title[HDSA for model discrepancy]{Hyper-differential sensitivity analysis with respect to model discrepancy: Mathematics and computation}

\author*[1]{\fnm{Joseph} \sur{Hart}}\email{joshart@sandia.gov}

\author[1]{\fnm{Bart} \sur{van Bloemen Waanders}}\email{bartv@sandia.gov}

\affil*[1]{\orgdiv{Scientific Machine Learning}, \orgname{Sandia National Laboratories}, \orgaddress{\street{P.O. Box 5800}, \city{Albuquerque}, \postcode{87123}, \state{NM}, \country{USA}}}

\abstract{
Model discrepancy, defined as the difference between model predictions and reality, is ubiquitous in computational models for physical systems. It is common to derive partial differential equations (PDEs) from first principles physics, but make simplifying assumptions to produce tractable expressions for the governing equations or closure models. These PDEs are then used for analysis and design to achieve desirable performance. For instance, the end goal may be to solve a PDE-constrained optimization (PDECO) problem. This article considers the sensitivity of PDECO problems with respect to model discrepancy. We introduce a general representation of the discrepancy and apply post-optimality sensitivity analysis to derive an expression for the sensitivity of the optimal solution with respect to the discrepancy. An efficient algorithm is presented which combines the PDE discretization, post-optimality sensitivity operator, adjoint-based derivatives, and a randomized generalized singular value decomposition to enable scalable computation. Kronecker product structure in the underlying linear algebra and corresponding infrastructure in PDECO is exploited to yield a general purpose algorithm which is computationally efficient and portable across  a range of applications. Known physics and problem specific characteristics of discrepancy are imposed through user specified weighting matrices. We demonstrate our proposed framework on two nonlinear PDECO problems to highlight its computational efficiency and rich insight. 
}


\maketitle

\section{Introduction}
Computational models provide a wealth of opportunity to understand fundamental properties of physical systems.  In many cases, these models take the form of partial differential equations (PDEs) which are derived from first principles based on laws of physics. However, as famously said by Box~\cite{box_1979}, ``All models are wrong but some are useful." To determine usefulness, it is critical to understand two sources of error: (i) uncertainty in model parameters, and (ii) error in the form of the model itself. Uncertainty in model parameters has been studied extensively in the field of uncertainty quantification \cite{uq_handbook}. Model form error, also referred to as model inadequacy or structural error, has also received attention, although its analysis is less mature due to the myriad of challenges and physics specific considerations. Error in the structure of a model (i.e. a simplification or omission for an operator in a PDE) produces a difference between the model prediction and reality, which we herein refer to as the model discrepancy.

The study of model discrepancy has been prominent in the model calibration literature since the seminal work of Kennedy and O'Hagan \cite{ohagan2001}. Their original framework considered calibration of model parameters and the hyper-parameters defining a Gaussian process discrepancy function to accommodate modeling error in calibration and prediction. Their framework has been extended by various authors \cite{Ling_2014,Maupin,Arendt_2012,Higdon_2008}. Several authors \cite{bayes_approx_error_petra,kopke_2018,Kaipio_2008} have considered the incorporation of model error in Bayesian inverse problems by modeling discrepancy in the likelihood function to mitigate bias on the solution of the inverse problem. Recent trends have turned the focus toward developing representations of model form error which respect laws of physics such as conservation properties. In \cite{Sargsyan_2019,sargsyan_2018,Sargsyan_2015}, an embedded approach is considered that places the representation of error inside the PDE in an attempt to satisfy physical properties. Other work have posed model form error as an unknown operator which is constrained by laws of physics in its representation and then inferred from data \cite{morrison_2018,portone}.

A common characteristic shared across these approaches is the trade-off between computational cost, model intrusiveness, and physics specific developments. Generally speaking, highly intrusive methods require considerable effort on a problem-to-problem basis and provide efficient analysis at the cost of development time. On the other hand, non-intrusive methods facilitate more rapid deployment but at the expense of higher computational cost and/or data requirements. This article focuses on developing a framework that is portable, nonintrusive, and computationally efficient.

We introduce a new approach to analyze the effect of model discrepancy in PDE-constrained optimization (PDECO) problems. Building on post-optimality sensitivity analysis \cite{post_opt_tutorial,shapiro_SIAM_review,Griesse_part_1,Griesse_part_2,griesse2} and its recent advances with hyper-differential sensitivity analysis (HDSA) \cite{HDSA,sunseri_hdsa,saibaba_gsvd,hart_2021_bayes}, the sensitivity of optimal solutions with respect to model discrepancy is considered. Because of the complexity of PDECO, we focus on computationally scalable methods and seek to leverage tools such as parallel and matrix-free linear algebra, adjoint-based derivative computations (which provides efficient derivative computation in high dimensions), and low rank approximations.

Our contributions include: 
\begin{enumerate}
\item[$\bullet$] an infinite dimensional formulation of post-optimality sensitivity analysis with respect to model discrepancy,
\item[$\bullet$] expressions for the discretized model discrepancy which are consistent with the infinite dimensional formulation,
\item[$\bullet$] exploitation of the Kronecker product structure to produce computationally scalable algorithms that utilize PDECO,
\item[$\bullet$] an algorithmic framework for the computation of the model discrepancy sensitivities and the associated computational complexity analysis,
\item[$\bullet$] demonstration of the proposed approach on nonlinear PDECO problems.
\end{enumerate}
These contributions are interrelated arising from a holistic perspective that spans the problem formulation, computational implementation, interpretation, and use of the sensitivities. This article focuses on the mathematical formulation, discretization, computation, and effect of model discrepancy on the optimization solution, while a companion article~\cite{model_discrepancy_2} demonstrates how the sensitivities may be coupled with high-fidelity data to improve the optimal solution and characterize corresponding uncertainty. 

The article is organized as follows. Section~\ref{sec:opt_and_mfes} gives the infinite dimensional formulation of a general PDECO problem and our proposed model discrepancy sensitivities. The problem discretization is used to derive an expression for the model discrepancy in Section~\ref{sec:discretization}. Section~\ref{sec:computation} outlines the algorithmic framework and Kronecker product structure in the model discrepancy representation that ensures scalability. Section~\ref{sec:numerical_results} demonstrates the proposed approach on three examples: an illustrative example using Poisson's equation, a distributed source control problem constrained by the convection-diffusion-reaction equation, and a boundary thermal flux control problem constrained by Boussinesq flow equations. Section~\ref{sec:conclusion} concludes with a discussion of the proposed approach  and impact on a companion article that focuses on the use of these sensitivities to leverage high-fidelity data in support of decision-making.

\section{Optimization and model discrepancy sensitivity} \label{sec:opt_and_mfes}

Consider the PDECO problem
\begin{align}
\label{eqn:true_opt_prob}
& \min_{z \in \Z} J(S(z),z) 
\end{align}
where $z$ denotes optimization variables in the (possibly infinite dimensional) Hilbert space $\Z$, $S:\Z \to \U$ denotes the solution operator for a PDE $c(u,z)=0$ with state variable $u$ in a (infinite dimensional) Hilbert space $\U$, and $J:\U \times \Z \to \R$ is the objective function. This formulation is applicable for design, control, and inverse problems. There is a wealth of literature on both algorithms and applications of PDECO. We direct the reader to \cite{Vogel_99, Archer_01,Haber_01,Vogel_02,Biegler_03,Biros_05,Laird_05,Hintermuller_05,Hazra_06,Biegler_07,Borzi_07,Hinze_09,Biegler_11,frontier_in_pdeco} for a comprehensive review. The ingredients most relevant to this work includes adjoint-based derivative computations, Krylov and Newton iterative solves, and parallel numerical linear algebra.

 In general, the ``true" PDE governing the system may be unknown or computationally intractable, so in practice we typically solve
\begin{align}
\label{eqn:approx_opt_prob}
& \min_{z \in \Z} J(\tilde{S}(z),z) 
\end{align}
where $\tilde{S}:\Z \to \U$ is the solution operator for a simpler PDE $\tilde{c}(u,z)=0$. In many applications, the high-fidelity PDE $c$ is only known in theory but cannot be expressed or solved, whereas $\tilde{c}$ is the PDE derived from first principles physics with simplifying assumptions that make it tractable. In practice, the lower-fidelity optimization problem~\eqref{eqn:approx_opt_prob} is solved to approximate the solution of the high-fidelity optimization problem~\eqref{eqn:true_opt_prob}. 

Our goal in this article is to study how discrepancies arising from changes in the PDE constraint influence the solution of~\eqref{eqn:approx_opt_prob}. To this end, consider the parameterized optimization problem 
\begin{align}
\label{eqn:approx_opt_prob_pert_rs}
 \min_{z \in \Z} \hspace{1 mm} \hat{J}(z,\delta):=J(\tilde{S}(z)+\delta(z),z) 
\end{align}
where $\delta$ represent the model discrepancy. Then~\eqref{eqn:approx_opt_prob_pert_rs} coincides with the lower-fidelity optimization problem~\eqref{eqn:approx_opt_prob} when $\delta=0$ and the high-fidelity optimization problem~\eqref{eqn:true_opt_prob} when $\delta = S-\tilde{S}$. 

We seek to compute the sensitivity of~\eqref{eqn:approx_opt_prob_pert_rs} with respect to $\delta$. This will require defining a Hilbert space for possible discrepancies. The definition of this space serves as a mechanism to impose known physics or mathematical properties on $\delta$. To this end, let $\mathcal{Y} \subseteq \mathcal{U}$ be a subspace of the state space equipped with an inner product $(\cdot,\cdot)_\mathcal{Y}:\mathcal Y \times \mathcal Y \to \R$ (that may be different from $\mathcal U$'s inner product) to impose known physical characteristics on the discrepancy. Let $\mu$ be a measure on $\Z$ to marginalize the discrepancy's dependence on $z$ and $C^2(\Z,\U)$ denotes the set of twice continuously differentiable operators from $\Z$ to $\mathcal Y$. Then 
\begin{eqnarray*}
\Delta = \{ \delta \in C^2(\Z,\mathcal Y) \vert (\delta(z),\delta(z))_\mathcal{Y}:\mathcal Z \to \R \text{ is integrable with respect to } \mu \} ,
\end{eqnarray*}
forms a Hilbert space of possible model discrepancies with the inner product $(\delta_1,\delta_2)_{\Delta} = \int (\delta_1(z),\delta_2(z))_\mathcal{Y} d\mu(z)$.

To satisfy technical assumptions in what follows, assume that $\hat{J}$ is twice continuously differentiable with respect to ($z,\delta)$. 
Let $\delta=\delta_0=0$ and $\overline{z}$ be a local minimum for~\eqref{eqn:approx_opt_prob_pert_rs} which satisfies the first and second order optimality conditions
\begin{eqnarray*}
\hat{J}_z'(\overline{z},\delta_0)=0 \qquad \text{and} \qquad \hat{J}_{z,z}''(\overline{z},\delta_0) \hspace{4 mm} \text{is positive definite,}
\end{eqnarray*}
where $\hat{J}_z'$ and $\hat{J}_{z,z}''$ denote the first and second Fr\'echet derivative of $\hat{J}$ with respect to $z$, respectively. Then the Implicit Function Theorem implies the existence of an operator $\mathcal F: \mathcal N(\delta_0) \to \mathcal N(\overline{z})$, defined on neighborhoods of $\delta_0 \in \Delta$ and $\overline{z} \in \Z$, such that
\begin{eqnarray*}
\hat{J}_z'(\mathcal F(\delta),\delta)=0 \qquad \forall \delta \in \mathcal N(\delta_0) .
\end{eqnarray*}
Further, the Fr\'echet derivative of $\mathcal F$ with respect to $\delta$, evaluated at $\delta_0$, is 
\begin{eqnarray}
\label{eqn:sen_op}
\mathcal F_\delta'(\delta_0) = - \mathcal H^{-1} \mathcal B
\end{eqnarray}
where $\mathcal H = \hat{J}_{z,z}''(\overline{z},\delta_0)$ and $\mathcal B = \hat{J}_{z,\delta}''(\overline{z},\delta_0)$. We interpret $\mathcal F_\delta'(\delta_0)$ as the sensitivity of the solution of the optimization problem~\eqref{eqn:approx_opt_prob_pert_rs} with respect to perturbations of the model discrepancy. 

The post-optimality sensitivity operator has been shown to provide rich insights when analyzing sensitivity with respect to physical parameters appearing in the PDE~\cite{HDSA,sunseri_hdsa}. This new formulation of post-optimality sensitivites with respect to model discrepancy is a valuable tool to support model development and decision-making through a quantitive understanding of how discrepancies influence the solution of optimization problems. 

To enable analysis of $\mathcal F_\delta'$, this article introduces a scalable and efficient computational framework by:
\begin{enumerate}
\item utilizing the PDE discretization and properties of $\mathcal F_\delta'$ to define a set of discretized model discrepancy operators,
\item defining an inner product on the discretized model discrepancy operators to enable user specification of known physics and/or mathematical properties, 
\item deriving computationally efficient expressions for the discretized $\mathcal F_\delta'$ and\\ model discrepancy inner product, and
\item developing a randomized generalized Singular Value Decomposition algorithm which exploits structure in the discretization to ensure computational efficiency.
\end{enumerate}
These four items are highly interrelated. For instance, the discretization of $\mathcal F_\delta'$ dictates the linear algebra structures needed for computational efficiency. As the following sections progress through the above points the reader should note how modeling choices at earlier stages are motivated by the resulting computational expressions they produce.

\section{Discretization} \label{sec:discretization}
We begin with the PDE discretization and use the resulting linear algebra constructs to direct our developments for the discretized model discrepancy and its inner product.
\subsection{PDE discretization}
Let $\{\phi_i\}_{i=1}^m$ be a basis for a finite dimensional subspace $\U_h \subset \U$ and $\{\psi_j\}_{j=1}^n$ be a basis for a finite dimensional subspace of $\Z_h \subset \Z$. These may, for instance, be finite element basis functions. Let $\u \in \R^m$ and $\z \in \R^n$ denote coordinates, $T_u:\R^m \to \U_h$ and $T_z:\R^n \to \Z_h$ denote coordinate transformations given by
$$T_u(\u) = \sum_{i=1}^m u_i \phi_i \qquad \text{ and } \qquad T_z(\z) = \sum_{k=1}^n z_k \psi_k,$$
and $T_u^{-1}: \U_h \to \R^m$ and $T_z^{-1}:\Z_h \to \R^n$ denote their inverses.

Inner products in $\U_h$ and $\Z_h$ can be computed via multiplication with the mass matrices $\M_u \in \R^{m \times m}$ and $\M_z \in \R^{n \times n}$, defined by 
\begin{eqnarray*}
(\M_u)_{i,j} = \left( \phi_i,\phi_j \right)_\U \qquad \text{and} \qquad (\M_z)_{i,j} = \left( \psi_i,\psi_j \right)_\Z .
\end{eqnarray*} 

The objective function $J$ is discretized by $\J:\R^m \times \R^n \to \R$ where 
$$\J(\u,\z) = J(T_u(\u),T_z(\z)),$$
 and the constraint $\tilde{c}$ is discretized by the system of equations $\tilde{\c}:\R^m \times \R^n \to \R^m$ where $\tilde{c}(T_u(\u),T_z(\z))$ is enforced at $m$ elements in the dual space of $\U$ (for instance, by projection or collocation). Similarly, we denote the discretized PDE solution operator as $\tilde{\S}:\R^n \to \R^m$.

\subsection{Model discrepancy discretization}
To achieve an expression for the model discrepancy we restrict it to the finite dimensional spaces $\mathcal U_h$ and $\mathcal Z_h$. A general form for the discrepancy in the discretized spaces is
\begin{align*}
&z \mapsto \sum\limits_{i=1}^m f_i(z) \phi_i
\end{align*}
where $f_i:\Z_h \to \R$, $i=1,2,\dots,m$, denotes functionals. We realize a simplification by observing that the sensitivity operator~\eqref{eqn:sen_op} only depends on $(z,z)$ and $(z,\delta)$ derivatives of the objective $\hat{J}$ evaluated at $\delta_0=0$. Without loss of generality we assume that $f_i(z)$, $i=1,2,\dots,m$, are affine functions of $z$. Using the Riesz representation for bounded linear functionals we have a general form
\begin{eqnarray*}
 z \mapsto \sum\limits_{i=1}^m (\theta_{i,0}+(z,a_i(\t))_{\mathcal Z}) \phi_i
\end{eqnarray*}
where $\theta_{i,0} \in \R$, $i=1,2,\dots,m$ and $a_i(\t) \in Z_h$, $i=1,2,\dots,m$, are the Riesz representation elements, which we parameterize with $\t$, a vector of coefficients. A general expression for elements in $Z_h$ is given by writing them as a linear combination of basis functions as
\begin{eqnarray*}
a_i(\t) =  \sum\limits_{j=1}^n  \theta_{i,j} \psi_j,
\end{eqnarray*}
where $\theta_{i,j} \in \R$. This yields the discretized model discrepancy, defined on the space of coordinates (rather than in the function spaces), $\d:\R^m \times \R^p \to \R^m$ given by
\begin{eqnarray}
\label{eqn:delta_1}
\d(\z,\t)=\sum\limits_{i=1}^m \left( \theta_{i,0} + \sum\limits_{j=1}^n \theta_{i,j} (T_z(\z),\psi_j)_{\mathcal Z} \right)T_u^{-1}(\phi_i)
\end{eqnarray}
where the vector of coefficients is defined as $\t=(\t_0^T,\t_1^T,\dots,\t_m^T)^T \in \R^p$, $p=m(n+1)$, where $\t_0=(\theta_{0,0},\theta_{1,0},\dots,\theta_{m,0})^T \in \R^m$ corresponds to the $m$ intercept terms and $\t_i = (\theta_{i,1},\theta_{i,2},\dots,\theta_{i,n})^T \in \R^n$, $i=1,2,\dots,m$, corresponds to the $m$ linear functionals.

Observe that when $\mathcal Z$ is a function space, and hence $n=dim(\mathcal Z_h)$ is large, $p$ will be extremely large (scaling like the square of the number of nodes in a mesh) and pose computational challenges. However, recalling that $T_u^{-1}(\phi_i)=\e_i^m$ and $(T_z(\z),\psi_j)_{\mathcal Z} =\z^T \M_z \e_j^n$, where $\e_i^m$ and $\e_j^n$ denotes canonical basis vectors in $\R^m$ and $\R^n$, respectively, we rewrite~\eqref{eqn:delta_1} as a matrix-vector product with $\t$. Observe that $\d(\z,\t)$ is a bi-linear function of $(\z,\t)$ which admits a convenient Kronecker product representation 
\begin{eqnarray}
\label{eqn:delta_kron}
\d(\z,\t) =
\left( \begin{array}{cc}
\I_m & \I_m \otimes \z^T \M_z
\end{array} \right) \t
\end{eqnarray}
where $\I_m \in \R^{m \times m}$ is the identity matrix. The Kronecker product structure of~\eqref{eqn:delta_kron} will prove critical in the subsequent analysis to achieve computational scalability. 

\subsection{Optimization problem discretization}
The discretization of the parameterized optimization problem~\eqref{eqn:approx_opt_prob_pert_rs} is given by
\begin{align}
\label{eqn:dis_approx_opt_prob_pert_rs}
 \min_{\z \in \R^n} \hspace{1 mm} \vec{\hat{J}}(\z,\t):=\J(\tilde{\S}(\z)+\d(\z,\t),\z) .
\end{align}
We assume that $\overline{\z} \in \R^n$ is a local minimum of~\eqref{eqn:dis_approx_opt_prob_pert_rs} which satisfies the first and second order optimality conditions when $\t=\vec{0}$, the zero vector in $\R^p$.

The discretization of the sensitivity operator~\eqref{eqn:sen_op} is given by
\begin{eqnarray}
\label{eqn:dis_sen_op}
\nabla_{\t} \F(\vec{0}) = - \H^{-1} \B \in \R^{n \times p}
\end{eqnarray}
where $\H \in \R^{n \times n}$ is the Hessian of the reduced space objective $\nabla_{\z,\z} \vec{\hat{J}}$, evaluated at $(\overline{\z},\vec{0})$, and $\B \in \R^{n \times p}$ is the Jacobian of $\nabla_\z \hat{J}$ with respect to $\t$, i.e. $\nabla_{\z,\t} \vec{\hat{J}}$, evaluated at $(\overline{\z},\vec{0})$. Throughout the article we will use $\nabla$ to denote differentiation of functions defined in Euclidean space with the differentiation variable denoted by subscripts. 

Our goal is to determine directions $\t \in \R^p$ for which the optimization problem is most sensitive. Specifically, we find directions that maximize $\vert \vert \nabla_{\t} \F(\vec{0}) \vert \vert_{\M_z}$ by computing the truncated Singular Value Decomposition.

\subsection{Model error inner product}
Our parameterization $\d$~\eqref{eqn:delta_kron} permits great expressiveness in the model discrepancy. This is motivated by our desire to develop algorithms which are portable across applications. However, in practice there is typically some known physics which should constrain it. We impose information from the underlying physics such as smoothness, conservation properties, boundary conditions, or invariances by weighting $\t$ to favor $\d$'s which respect the users specifications. This corresponds to the discretization of the inner product $\mathcal Y$ in Section~\ref{sec:opt_and_mfes} which we define through a symmetric positive definite weighting matrix $\L \in \R^{m \times m}$. We will discuss the choice further in Section~\ref{sec:numerical_results}, but our only requirement is that $\L$ be efficiently invertible.

Leveraging the Kronecker structure of $\d$~\eqref{eqn:delta_kron}, we compute the $\L$-weighted inner product of $\d$ with itself and arrive at the convenient expression
\begin{eqnarray}
\label{eqn:delta_ip}
(\d(\z,\t),\d(\z,\t))_\L = \t^T 
\left( \begin{array}{cc}
\L & \L \otimes \z^T \M_z \\
\L \otimes \M_z \z & \L \otimes \M_z \z \z^T \M_z
\end{array} \right) 
\t .
\end{eqnarray}
Since~\eqref{eqn:delta_ip} measures the size of the discrepancy for a specific $\z$, we compute the expectation of~\eqref{eqn:delta_ip} with respect to $\z$ to determine the size of $\d$ globally over the $\z$ space. To enable interpretability and achieve computational efficiency, we compute the expectation using a Gaussian probability measure (discretizing the measure $\mu$ in Section~\ref{sec:opt_and_mfes}) whose mean is $\overline{\z}$ and covariance is $\vec{\Gamma} \in \R^{n \times n}$. We leverage known properties of Gaussians to define $\vec{\Gamma}$ in a way to enforce length scales and smoothness properties, as discussed in Section~\ref{sec:numerical_results}.

Manipulating linear algebra expressions and properties of the mean and covariance of Gaussian random vectors, we observe that
\begin{eqnarray*}
\mathbb{E}_\z \left[ (\d(\z,\t),\d(\z,\t))_\L  \right]= \t^T 
\left( \begin{array}{cc}
\L & \L \otimes \overline{\z}^T \M_z \\
\L \otimes \M_z \overline{\z} & \L \otimes \E
\end{array} \right) 
\t ,
\end{eqnarray*}
where $\E = \M_z \left( \vec{\Gamma} + \overline{\z} \hspace{.4 mm}  \overline{\z}^T \right) \M_z$. This implies the symmetric positive definite weighting matrix
\begin{eqnarray}
\label{eqn:M_theta}
\M_\t =  \left( \begin{array}{cc}
\L & \L \otimes \overline{\z}^T \M_z \\
\L \otimes \M_z \overline{\z} & \L \otimes \E
\end{array} \right)  \in \R^{p \times p}
\end{eqnarray} 
to define the inner product on $\t \in \R^p$. 

\section{Sensitivity computation using Kronecker structure} \label{sec:computation}
Our developments thus far have provided general expressions for the model discrepancy and its associated inner product. The goal of our analysis is to compute the leading singular values and singular vectors of the discretized sensitivity operator~\eqref{eqn:dis_sen_op} with $\M_\t$ and $\M_z$ weighted inner products on its domain and range, respectively. This is challenging because the high dimensionality of $\t \in \R^p$ makes forming and computing with dense vectors in $\R^p$ prohibitive. We leverage the randomized Generalize Singular Value Decomposition (GSVD) algorithm from~\cite{saibaba_gsvd}. To ensure scalability, the computation is performed exclusively in $\R^m$ and $\R^n$ using the Kronecker structure of $\d$, without every forming or computing with a vector in $\R^p$.

\subsection{Kronecker structure in sensitivity matrices}
The Kronecker structure of $\d$ translates to the matrices $\M_\t$, $\M_\t^{-1}$, and $\B$. Given the structure of $\M_\t$~\eqref{eqn:M_theta}, Theorem~\ref{thm:M_theta_inv} provides $\M_\t^{-1}$ in a similar structure. 

\begin{theorem}
\label{thm:M_theta_inv}
\begin{eqnarray}
\label{eqn:M_theta_inv}
\M_\t^{-1} =  (1+\beta)
\left( \begin{array}{cc}
\L^{-1} & -\L^{-1} \otimes \x^T \\
- \L^{-1} \otimes \x & \L^{-1} \otimes \N
\end{array} \right)
\end{eqnarray}
where
\begin{eqnarray*}
\beta = \overline{\z}^T \vec{\Gamma}^{-1} \overline{\z}, \qquad \x =  \M_z^{-1} \left( \vec{\Gamma}^{-1} -\G \right) \overline{\z},
\end{eqnarray*}
\begin{eqnarray*}
 \N = \frac{1}{1+\beta} \M_z^{-1} \vec{\Gamma}^{-1} \M_z^{-1}, \qquad \text{and} \qquad \G = \frac{1}{1+\beta}\vec{\Gamma}^{-1} \overline{\z}} \hspace{.4 mm} {\overline{\z}^T \vec{\Gamma}^{-1} .
\end{eqnarray*}
\end{theorem}
\begin{proof}
A proof by multiplying $\M_\t \M_\t^{-1}$ is given in the Appendix.
\end{proof}

The matrices $\M_z$, $\vec{\Gamma}$, and $\L$ are typically sparse (or admit matrix-vector products using sparse multiplies/solves). Hence we efficiently compute matrix-vector products with $\M_\t$ and $\M_\t^{-1}$ without explicitly forming them.

To determine the form of $\B = \nabla_{\z, \t} \vec{\hat{J}}(\overline{\z},\vec{0})$ we apply the Chain rule to $\vec{\hat{J}}(\z,\t) = \J(\tilde{\S}(\z)+\d(\z,\t),\z)$ yielding the expression
\begin{eqnarray}
\label{eqn:J_grad}
\nabla_\z \vec{\hat{J}}= \nabla_\u \J \nabla_\z \tilde{\S} + \nabla_\u \J \nabla_\z \d + \nabla_\z \J 
\end{eqnarray}
where $\nabla_{\z}\tilde{\S}$ is the Jacobian of $\tilde{\S}$ with respect to $\z$. 

To facilitate subsequent linear algebra derivations we adopt the convention that gradients are row vectors. For notational simplicity, we omit the input arguments of functions whenever they are not needed. All functions will be evaluated at the nominal solution $\z=\overline{\z}$, $\t=\vec{0}$, and $\u=\tilde{\S}(\overline{\z})$.

 Differentiating~\eqref{eqn:J_grad} with respect to $\t$ yields
\begin{align}
\label{eqn:B_abstract}
\nabla_{\z,\t} \vec{\hat{J}}  = & \nabla_\z\tilde{\S}^T \nabla_{\u,\u} \J \nabla_\t\d + \nabla_\z\d^T \nabla_{\u,\u} J \nabla_\t\d \\
& + \nabla_\u \J \nabla_{\z,\t} \d + \nabla_{\z,\u} \J \nabla_\t\d \nonumber.
\end{align}
To simplify our subsequent analysis we assume that $ \nabla_{\z,\u} \J =0$. This is common on a wide range of optimization problems \footnote{This occurs when the objective function admits the structure $J(u,z) = J_{mis}(u) + J_{reg}(z)$ where $J_{mis}$ is a state misfit or design criteria and $J_{reg}$ is an optimization variable regularization.}. Our subsequent analysis can be done in the more general case but is beyond the scope of the article.

In order to write $\B$ in a Kronecker structure we need expressions for $\nabla_\t \d$, $\nabla_z\d$, and $\nabla_\u \J \nabla_{\z,\t}\d$ \footnote{The notation $\nabla_\u \J \nabla_{\z,\t}\d$ refers to the $(\z,\t)$ Hessian of the scalar valued function $\vec{v} \d$, where $\vec{v}=\nabla_\u \J$.}. Note that for $\nabla_\u \J \in \R^{1 \times m}$ we have
\begin{eqnarray*}
\nabla_\u \J \d(\z,\t) =
\left( \begin{array}{cc}
 \nabla_\u \J & \z^T \left( \nabla_\u \J \otimes \M_z \right)
 \end{array} \right)
  \t ,
\end{eqnarray*}
a bi-linear form which can be easily differentiated. Our previous expression~\eqref{eqn:delta_kron} for $\d$ is linear in $\t$ which permits a simple computation of $\nabla_\t \d$. We also write $\d$ as a linear function of $\z$,
\begin{eqnarray*}
\d(\z,\t) = \mathcal I(\t) + \mathcal K(\t) \M_z \z ,
\end{eqnarray*}
where $\mathcal I: \R^p \to \R^m$ and  $\mathcal K: \R^p \to \R^{m \times n}$ are linear operator defined as
\begin{align*}
\mathcal I 
\left( \begin{array}{c}
\t_0 \\
\t_1 \\
\vdots \\
\t_m
\end{array} \right) 
=
\t_0
\quad \text{and} \quad
\mathcal K 
\left( \begin{array}{c}
\t_0 \\
\t_1 \\
\vdots \\
\t_m
\end{array} \right) 
=
\left( \begin{array}{c}
\t_1^T \\
\t_2^T \\
\vdots \\
\t_m^T
\end{array} \right)
=
\left( \begin{array}{cccc}
\theta_{1,1} & \theta_{1,2} & \cdots & \theta_{1,n} \\
\theta_{2,1} & \theta_{2,2} & \cdots & \theta_{2,n} \\
\vdots & \vdots & \ddots & \vdots \\
\theta_{m,1} & \theta_{m,2} & \cdots & \theta_{m,n} \\
\end{array} \right) .
\end{align*}

Using these representations of $\d$ we express the derivatives as,
\begin{eqnarray*}
\nabla_\t \d(\z,\t) =
\left( \begin{array}{cc}
 \I_m &   \I_m \otimes \z^T \M_z 
 \end{array} \right)
\in \R^{m \times p},
 \qquad \nabla_\z \d(\z,\t) = \mathcal K(\t) \M_z \in \R^{m \times n}, 
\end{eqnarray*}
and
\begin{eqnarray*}
\nabla_\u \J \nabla_{\z,\t}\d(\z,\t) =
\left( \begin{array}{cc}
\vec{0} & \nabla_\u \J  \otimes \M_z 
\end{array} \right)
\in \R^{n \times p}.
\end{eqnarray*}

With our assumption that $ \nabla_{\z,\u} \J =0$, the derivatives of $\d$ above, and the observation that $\nabla_\z\d(\overline{\z},\vec{0})=0$, we have
\begin{eqnarray}
\label{eqn:B}
\B = 
\nabla_\z \tilde{\S}^T \nabla_{\u, \u} \J
\left( \begin{array}{cc}
 \I_m & \I_m \otimes \overline{\z}^T \M_z 
 \end{array} \right)
 +
\left( \begin{array}{cc}
\vec{0} & \nabla_\u \J \otimes \M_z
 \end{array} \right)
 \in \R^{n \times p} 
\end{eqnarray}
where $\nabla_\z \tilde{\S}$, $\nabla_\u \J$, and $\nabla_{\u,\u} \J$ are evaluated at $(\tilde{\S}(\overline{\z}),\overline{\z})$. Properties of Kronecker product transposes and symmetry of $\M_z$ gives
\begin{eqnarray}
\label{eqn:Bt}
\B^T =
\left( \begin{array}{c}
 \I_m \\
  \I_m \otimes \M_z \overline{\z} 
 \end{array} \right)
  \nabla_{\u, \u} \J  \nabla_\z \tilde{\S}
 +
\left( \begin{array}{c}
\vec{0} \\
 \nabla_\u \J^T \otimes \M_z
 \end{array} \right)
 \in \R^{p \times n} .
\end{eqnarray}

\subsection{Randomized GSVD algorithm}
Since the sensitivity operator $\nabla_{\t} \F(\vec{0})$ is a large dense matrix which is only accessible via matrix-vector products, we analyze it by computing its truncated Generalized Singular Value Decomposition (GSVD), where the inner products are defined by $\M_z$ and $\M_\t$. Algorithm~\ref{alg:GSVD} summarizes the randomized GSVD algorithm from~\cite{saibaba_gsvd} to highlight the matrix-vector products required to compute the truncated GSVD.  Algorithm~\ref{alg:GSVD} repeatedly calls CholQR~\footnote{To highlight which outputs of CholQR are needed we use Matlab notation with $\sim$ denoting an empty output argument for a function call.}, the Cholesky QR algorithm to orthogonalize in weighted inner products, which we present in Algorithm~\ref{alg:cholQR}.

\begin{algorithm}
\caption{Randomized GSVD for model discrepancy sensitivities}
\label{alg:GSVD}
\begin{algorithmic}[1]
\State \textbf{Input: } target rank $k$, oversampling factor $\ell$, subspace iterations $q$ 
\State Generate a random matrix $\vec{\Omega}$ \hspace{46.7 mm} $\vec{\Omega} \in \R^{\mathbf{p} \times (k+\ell)}$ 
\State Compute $\Y = \H^{-1} \B \vec{\Omega}$ \hspace{53.7 mm} $\Y \in \R^{n \times (k+\ell)}$
\State $[\sim,\M_z \Q,\sim]= \text{CholQR}(\Y,\M_z)$  \hspace{33.8 mm} $\M_z \Q \in \R^{n \times (k+\ell)}$
\For{$N$ = 1 to $q$}
\State Compute $\Y = \B^T \H^{-1} \M_z \Q$ \hspace{40 mm} $\Y \in \R^{\mathbf{p} \times (k+\ell)}$
\State $[\sim,\M_\t^{-1}\Q,\sim]= \text{CholQR}(\Y,\M_\t^{-1})$ \hspace{22.1 mm} $\M_\t^{-1} \Q \in \R^{\mathbf{p} \times (k+\ell)}$ 
\State Compute $\Y = \H^{-1} \B \M_\t^{-1} \Q $  \hspace{40.1 mm} $\Y \in \R^{n \times (k+\ell)}$
\State $[\Q,\M_z \Q,\sim]= \text{CholQR}(\Y,\M_z)$ \hspace{34.3 mm} $\Q \in \R^{n \times (k+\ell)}$
\EndFor
\State Compute $\W = \B^T \H^{-1} \M_z \Q$ \hspace{43 mm} $\W \in \R^{\mathbf{p} \times (k+\ell)}$
\State $[\Q_\W,\sim,\vec{R}_\W]= \text{CholQR}(\M_\t^{-1} \W,\M_\t)$ \hspace{25.1 mm} $\Q_\W \in \R^{\mathbf{p} \times (k+\ell)}$
\State Compute the SVD of $\vec{R}_\W^T=\vec{U}_\W \vec{\Sigma} \vec{V}_\W^T$ \hspace{8 mm} $\vec{U}_\W, \vec{\Sigma},\vec{V}_\W \in \R^{(k+\ell) \times (k+\ell)}$
\State Compute $\vec{U} = \Q \vec{U}_\W$ and $\vec{V} = \Q_\W \vec{V}_W$ \hspace{6.1 mm} $\vec{U} \in \R^{n \times (k+\ell)}$, $\vec{V} \in \R^{\mathbf{p} \times (k+\ell)}$
\State \textbf{Return: } Estimated truncated GSVD $\H^{-1} \B \approx \vec{U} \vec{\Sigma} \vec{V}^T \M_\t$
\end{algorithmic}
\end{algorithm}

\begin{algorithm}
\caption{CholQR}
\label{alg:cholQR}
\begin{algorithmic}[1]
\State \textbf{Input: } Rectangular matrix $\Y$, symmetric positive definite matrix $\M$
\State $\vec{Z} = \M \Y$
\State $\vec{C} = \Y^T \vec{Z}$
\State Compute the Cholesky factorization $\vec{C} = \vec{R}^T \vec{R}$
\State $\Q = \Y \vec{R}^{-1}$
\State $\M \Q = \vec{Z} \vec{R}^{-1}$
\State \textbf{Return: } $\Q $, $\M \Q$, $\vec{R}$ such that $\Y = \Q \vec{R}$ and $\Q^T \M \Q = \I$
\end{algorithmic}
\end{algorithm}

Notice that the majority of steps in Algorithm~\ref{alg:GSVD} involve computation with matrices containing $p$ rows or columns. Storing, communicating, and computing with vectors in $\R^p$ is intractable for many applications since $p=m(n+1)$, or as in the numerical results, $p=\mathcal(10^8)$. We demonstrate how the computation with vectors in $\R^p$ can be compressed to computation in $\R^m$ and $\R^n$ thus enabling efficient estimation of the truncated GSVD.

\subsection{Computation with Kronecker structure}
Each line of Algorithm~\ref{alg:GSVD} can be executed efficiently leveraging the Kronecker structure observed in $\M_\t^{-1}$, $\B$, and $\B^T$. We represent each matrix of size $p \times (k+\ell)$ in Algorithm~\ref{alg:GSVD} using a family of $d=k+\ell$ vectors in the form
\begin{eqnarray}
\label{eqn:kron_vec}
\left( \begin{array}{c}
a \u_N \\
\u_N \otimes \z_0
\end{array} \right)
+
\left( \begin{array}{c}
b_N \u_0 \\
\u_0 \otimes \z_N
\end{array} \right) \in \R^p
\end{eqnarray}
where $a \in \R$, $\b=(b_1,b_2,\dots,b_{d})^T \in \R^{d}$, $\u_N \in \R^m$ and $\z_N \in \R^n$, $N=0,1,\dots,d$. Hence the storage and communication requirement for a $p \times d$ matrix is $(m+n+1)(d+1)$ floating point numbers. We will refer to~\eqref{eqn:kron_vec} as a rank-2 Kronecker form. The subsequent developments show how each line of Algorithm~\ref{alg:GSVD} can be executed using matrices represented in rank-2 Kronecker form.
\subsubsection*{Lines 2-3}
A naive implementation of these lines requires that we generate the random matrix $\vec{\Omega} \in \R^{p \times d}$ and compute $d$ matrix-vector products $\B \vec{\omega}_N$, $N=1,2,\dots,d$, where $\vec{\omega_N}$ is the $N^{th}$ column of $\vec{\Omega}$. There are a number of different sampling strategies available for the generation of $\vec{\Omega}$; sampling each entry independently from a standard normal distribution is prominent in practice. We will adopt this sampling strategy to facilitate our subsequent derivations. Theorem~\ref{thm:B_omega} shows that we can compute $\B \vec{\Omega}$ using random vectors and computation in $\R^m$ and $\R^n$ rather than $\R^p$.
\begin{theorem}
\label{thm:B_omega}
Let $\vec{\omega} \sim N(\vec{0},\I_p)$ and $\vec{\omega}_{\u_0} \sim N(\vec{0},\I_m)$ be Gaussian random vectors in $\R^p$ and $\R^m$, respectively, and let
\begin{align*}
\left(
\begin{array}{c}
\vec{\omega}_{\u} \\
\vec{\omega}_{\z}
\end{array}
\right)
\sim
N
\left( 
\vec{0} ,
\left(
\begin{array}{cc}
\vert \vert \M_z \overline{\z} \vert \vert_2^2 \I_m & \nabla_\u \J^T (\M_z \overline{\z})^T \\
(\M_z \overline{\z}) \nabla_\u \J & \vert \vert \nabla_\u \J \vert \vert_2^2 \I_n
\end{array}
\right)
\right)
\end{align*}
be a  Gaussian random vector in $\R^{m + n}$. Then 
\begin{eqnarray*}
\B \vec{\omega} \qquad \text{and} \qquad \nabla_\z \tilde{\S}^T \nabla_{\u,\u}\J (\vec{\omega}_{\u_0} + \vec{\omega}_\u ) + \M_z \vec{\omega}_\z
\end{eqnarray*}
have the same distribution.
\end{theorem}
\begin{proof}
Notice that
\begin{align*}
\B \vec{\omega} &= \nabla_\z \tilde{\S}^T \nabla_{\u, \u} \J
\left( \begin{array}{cc}
 \I_m & \I_m \otimes \left( \overline{\z}^T \M_z \right) 
 \end{array} \right) \vec{\omega}
 +
\left( \begin{array}{cc}
\vec{0} & \nabla_\u \J \otimes \M_z
 \end{array} \right) \vec{\omega} \\
& = \nabla_\z \tilde{\S}^T \nabla_{\u, \u} \J \mathcal I(\vec{\omega}) +  \nabla_\z \tilde{\S}^T \nabla_{\u, \u} \J  \mathcal K(\vec{\omega}) \M_z \overline{\z}  +  \M_z \mathcal K(\vec{\omega})^T \nabla_\u \J^T \\
& = \nabla_\z \tilde{\S}^T \nabla_{\u, \u} \J \left( \mathcal I(\vec{\omega}) + \mathcal K(\vec{\omega}) \M_z \overline{\z} \right) + \M_z \mathcal K(\vec{\omega})^T \nabla_\u \J^T .
\end{align*}
Observe that the random vectors $\mathcal K(\vec{\omega}) \M_z \overline{\z} \in \R^m$ and $\mathcal K(\vec{\omega})^T \nabla_\u \J^T \in \R^n$ correspond to applying the linear transformation
\begin{align*}
\vec{T} = \left( \begin{array}{cc}
\vec{0}_{m \times m} & \I_m \otimes ( \M_z \overline{\z})^T \\
\vec{0}_{n \times m} & \nabla_\u \J \otimes \I_n
 \end{array}
  \right) \in \R^{(m+n) \times p}
\end{align*}
to $\vec{\omega}$. Hence 
\begin{align*}
\left(
\begin{array}{c}
\mathcal K(\vec{\omega}) \M_z \overline{\z} \\
\mathcal K(\vec{\omega})^T \nabla_\u \J^T
\end{array}
\right)
\end{align*}
is a Gaussian random vector in $\R^{m+n}$ with mean $\vec{0}$ and covariance matrix $\T \T^T$, i.e. the same distribution as $\begin{pmatrix} \vec{\omega}_\u^T & \vec{\omega}_\z^T \end{pmatrix}^T$.
\end{proof}
Theorem~\ref{thm:B_omega} indicates that rather than forming a random matrix $\vec{\Omega} \in \R^{p \times d}$, we form two random matrices $\vec{\Omega}_{\u_0} \in \R^{m \times d}$ and $\vec{\Omega}_{\u,\z} \in \R^{(m+n) \times d}$ and perform computation in $\R^m$ and $\R^n$. This dimension reduction is enabled because taking matrix-vector products with $\mathcal K(\vec{\omega})$, a $m \times n$ matrix with independent standard normal entries, sums over the rows to produce $m$ and $n$ dimensional random vectors.

\subsubsection*{Line 6}
This line requires computation of $d$ matrix-vector products of the form $\B^T \v_N \in \R^p$ where $\v_N \in \R^n$ is a column of $\H^{-1} \M_z \Q \in \R^{n \times d}$. Recalling the Kronecker structure of $\B^T$, we have
\begin{align*}
\B^T \v_N & =
\left( \begin{array}{c}
 \I_m \\
  \I_m \otimes \left(  \M_z \overline{\z} \right) 
 \end{array} \right)
  \left(  \nabla_{\u, \u} \J  \nabla_\z \tilde{\S} \v_N \otimes 1\right)
 +
\left( \begin{array}{c}
\vec{0} \\
 \nabla_\u \J^T \otimes \M_z
 \end{array} \right) \left( 1 \otimes \v_N \right) \\
&= \left( \begin{array}{c}
   \nabla_{\u, \u} \J  \nabla_\z \tilde{\S} \v_N \\
 \nabla_{\u, \u} \J  \nabla_\z \tilde{\S} \v_N \otimes \left(  \M_z \overline{\z} \right) 
 \end{array} \right)
 +
\left( \begin{array}{c}
\vec{0} \\
 \nabla_\u \J^T \otimes \M_z \v_N
 \end{array} \right) .
\end{align*}

\subsubsection*{Line 7}
The result of line 6 is a matrix $\Y \in \R^{p \times d}$, stored in the rank-2 Kronecker form~\eqref{eqn:kron_vec} with $a=1$, $\b=\vec{0}$, whose columns we denote as
\begin{eqnarray} \label{eqn:y_N}
\y_N =
\left( \begin{array}{c}
 \u_N  \\
 \u_N \otimes \z_0 
 \end{array} \right)
 +
 \left( \begin{array}{c}
 \vec{0} \\
  \u_0 \otimes \z_N
   \end{array} \right)
   \in \R^p
\end{eqnarray}
to simplify notation for what follows. The CholQR call in line 7 requires the computation of
\begin{eqnarray*}
\vec{Z}= \M_\t^{-1} \Y \in \R^{p \times d}, \quad \C = \Y^T \vec{Z} = \Y^T \M_\t^{-1} \Y \in \R^{d \times d}, \quad \text{and} \quad \vec{Z} \vec{R}^{-1} \in \R^{p \times d} .
\end{eqnarray*}

All of these computations are performed in $\R^m$ and $\R^n$. Recalling $\M_\t^{-1}$~\eqref{eqn:M_theta_inv} and $\y_N$~\eqref{eqn:y_N}, we compute the $N^{th}$ column of $\vec{Z} \in \R^{p \times d}$ as
\begin{eqnarray*}
\z_N= (1+\beta) \left[
\left( \begin{array}{c}
(1-\x^T \z_0) \L^{-1} \u_N \\
\L^{-1} \u_N \otimes \left( \N \z_0 - \x \right)
\end{array} \right)
+
\left( \begin{array}{c}
- (\x^T \z_N) \L^{-1} \u_0 \\
\L^{-1} \u_0 \otimes \N \z_N
\end{array} \right) \right] .
\end{eqnarray*}
Hence $\vec{Z}$ is stored in the rank-2 Kronecker form~\eqref{eqn:kron_vec}.

The matrix-matrix product $\C = \Y^T \vec{Z}$ requires computing $d^2$ inner products involving columns of $\Y$ and $\vec{Z}$, which are in $\R^p$. However, recalling~\eqref{eqn:kron_vec}, we compute the inner products in $\R^p$ by computing four inner products in $\R^m$ and four inner products in $\R^n$.

Lastly, since $\vec{R} \in \R^{d \times d}$ is stored as a dense upper triangular matrix, we compute the product of $\vec{Z}$ with $\vec{R}^{-1}$ by taking linear combinations of the columns of $\vec{Z}$. Since it is stored in the rank-2 Kronecker form~\eqref{eqn:kron_vec}, we compute the linear combinations of vectors in $\R^m$ and $\R^n$. The output of line 7 is another $p \times d$ matrix represented in the form of~\eqref{eqn:kron_vec}.

\subsubsection*{Line 8}
Line 8 involves $d$ matrix-vector products involving $\B$ with incoming vectors of the form~\eqref{eqn:kron_vec}. Recalling the Kronecker structure of $\B$~\eqref{eqn:B}, we have matrix-vector products of the form
\begin{align*}
\B 
\left[
\left( \begin{array}{c}
a \v_{\u_1} \\
\v_{\u_1} \otimes \v_{\z_1} 
\end{array} \right)
+
\left( \begin{array}{c}
b \v_{\u_2} \\
\v_{\u_2} \otimes \v_{\z_2} 
\end{array} \right)
\right]
&= 
\left(a+\overline{\z}^T \M_z \v_{\z_1} \right)\nabla_\z \tilde{\S}^T \nabla_{\u, \u} \J \v_{\u_1} \\
& + 
\left( \nabla_\u \J \v_{\u_1} \right) \M_z \v_{\z_1} \\
&+ 
\left(b+\overline{\z}^T \M_z \v_{\z_2} \right)\nabla_\z \tilde{\S}^T \nabla_{\u, \u} \J \v_{\u_2} \\
& + 
\left( \nabla_\u \J \v_{\u_2} \right) \M_z \v_{\z_2} 
\end{align*}
which requires matrix-vector products and inner products in $\R^m$ and $\R^n$ to return a vector in $\R^n$.

\subsubsection*{Lines 9-14}
For conciseness, we omit a complete discussion of how each of the subsequent lines of Algorithm~\ref{alg:GSVD} are implemented. We note that the implementation of lines 11 and 12 mirrors lines 6 and 7 where the matrix-vector products and vector inner products are done in the rank-2 Kronecker form. The matrix-matrix product $\Q_\W \V$ in line 14 is also executed by taking linear combinations of columns of $\Q_\W$, which are stored in rank-2 Kronecker form.

\subsection{Computational cost, communication, and memory requirements} \label{ssec:cost_comm_memory}
In our implementation of Algorithm~\ref{alg:GSVD} we have replaced all computation, storage, and communication of vectors in $\R^p$ with the rank-2 Kronecker form~\eqref{eqn:kron_vec}. Table~\ref{tab:compute_cost} summarizes the matrix-vector products required to execute Algorithm~\ref{alg:GSVD}, omitting dense linear algebra on matrices in $\R^{d \times d}$, $\R^{m\times d}$, and $\R^{n \times d}$. The first row in Table~\ref{tab:compute_cost} is an initialization step where matrix-vector products needed in subsequent lines are computed once and stored for later use.
	\begin{table}[!ht]
	    \centering
	    \begin{tabular}{|c|c|c|c|c|c|c|c|c|c}
	    \hline
	    \hspace{1 mm} & $\M_z$ & $\M_z^{-1}$  & $\vec{\Gamma}^{-1}$ & $\L^{-1}$ & $\nabla_{\u,\u} \J$ & $\nabla_\z \tilde{\S}$ & $\nabla_\z \tilde{\S}^T$ & $\nabla_{\z,\z} \vec{\hat{J}}^{-1}$ \\ \hline
	    Initialization & 1 & 1 & 1 & 1& 0 & 0& 0 & 0 \\ \hline
	    Line 3 & $d$  & 0 & 0 & 0 & $d$ & 0 & $d$ & $d$ \\
	    Line 4 & $d$  & 0 & 0 & 0 & 0 & 0 & 0 & 0 \\ \hline
	    Line 6 & $d$  & 0 & 0 &  0 & $d$ & $d$ & 0 & $d$ \\
	    Line 7 & 0 & $2d$  & $d$ &  $d$ & 0 & 0 & 0 & 0 \\
	    Line 8 & $d+1$  & 0 & 0 &  0 & $d+1$ & 0 & $d+1$ & $d$ \\
	    Line 9 & $d$ &  0 & 0 &  0 & 0 & 0 & 0 & 0 \\ \hline
	    Line 11 & $d$ &  0 & 0 &  0 & $d$ & $d$ & 0 & $d$ \\
	    Line 12 & 0 & $2d$ & $d$ &  $d$ & 0 & 0 & 0 & 0 \\ \hline
	     \end{tabular}
	    \caption{Summary of the matrix-vector products required to execute Algorithm~\ref{alg:GSVD}. Note that $d=k+\ell$ is the target rank plus oversampling factor and that lines 6-9 are in a for loop which is executed $q$ times.}
	    \label{tab:compute_cost}
	\end{table}
	
To analyze the computational cost we will first unpack the computation required for matrix-vector products with $\nabla_\z \tilde{\S}$ and $\nabla_\z \tilde{\S}^T$. Recalling that $\tilde{\S}$ is the solution operator which satisfies $\tilde{\c}(\tilde{\S}(\z),\z)=0$ for all $z$, we apply the Chain Rule and observe that
\begin{eqnarray*}
\nabla_\z \tilde{\S} = -\nabla_{\u} \tilde{\c}^{-1} \nabla_{\z} \tilde{\c} .
\end{eqnarray*}
The Jacobians $\nabla_{\u} \tilde{\c}$ and $\nabla_{\z} \tilde{\c}$ are sparse matrices which we assemble once prior to executing Algorithm~\ref{alg:GSVD}. Hence the cost of computing a matrix-vector product with $\nabla_\z \tilde{\S}$ (or $\nabla_\z \tilde{\S}^T$) is approximately the cost of a linear system solve with $\nabla_{\u} \tilde{\c}$. We will use these solves as a measure of cost since they depend upon properties of the physics (symmetry, availability of preconditioners, etc.) and the system size (which determines whether sparse direct or iterative solvers should be used).

The Hessian $\nabla_{\z,\z} \vec{\hat{J}}$ is large, dense, and only accessible via matrix-vector products. Furthermore, such matrix-vector products are computed using incremental adjoint equations where each product requires solving a linear system with coefficient matrix $\nabla_{\u} \tilde{\c}$ and another linear system with coefficient matrix $\nabla_{\u} \tilde{\c}^T$. Because of symmetry, we compute matrix-vector products with $\nabla_{\z,\z} \vec{\hat{J}}^{-1}$ using a conjugate gradient iterative solver. Assuming that an average of $L_{CG}$ iterations are needed for the solves, we will measure the cost of $\nabla_{\z,\z} \vec{\hat{J}}^{-1}$ matrix-vector products as $L_{CG}$ solves with $\nabla_{\u} \tilde{\c}$ plus $L_{CG}$ solves with $\nabla_{\u} \tilde{\c}^T$.

Typically, we can compute matrix-vector products with $\M_z^{-1}$ and $\L^{-1}$ via sparse direct or iterative solvers. We measure the cost of inverting these matrices in our total cost analysis with a recognition that these solves are nontrivial, but are well understood. We omit the cost of matrix-vector products with $\vec{\Gamma}^{-1}$ since the covariance matrix $\vec{\Gamma}$ is typically defined as the inverse of a differential operator, hence matrix-vector products with $\vec{\Gamma}^{-1}$ do not require linear solves.

Relative to the cost of these linear system solves, the computational cost of matrix-vector product with $\M_z$ and $\nabla_{\u,\u} \J$ is negligible. We summarize the total approximate cost of the algorithm in large linear system solves in Table~\ref{tab:linear_solves}. Typically, linear system solves with $\nabla_{\u} \tilde{\c}$ and $\nabla_{\u} \tilde{\c}^T$ will have a comparable or greater cost than that of $\M_z$ and $\L$. With this in mind, the total cost of Algorithm~\ref{alg:GSVD} is approximately $2(q+1)d(1+2L_{CG})$ large linear system solves (involving $\nabla_{\u} \tilde{\c}$ or $\nabla_{\u} \tilde{\c}^T$).

	\begin{table}[!ht]
	    \centering
	    \begin{tabular}{|c|c|c|c|}
	    \hline
	     $\M_z$ & $\L$ & $\nabla_{\u} \tilde{\c}$ & $\nabla_{\u} \tilde{\c}^T$ \\ \hline
	     $2(q+1)d+1$ & $(q+1)d+1$ & $(q+1)d(1+2L_{CG})$ & $(q+1)d(1+2L_{CG})+q$ \\ \hline
	     \end{tabular}
	    \caption{Summary of large linear system solves required by Algorithm~\ref{alg:GSVD}. Note that $d=k+\ell$ is the target rank plus oversampling factor, q is the number of subspace iterations, and $L_{CG}$ is the average number of iterations needed to invert $\nabla_{\z,\z} \vec{\hat{J}}^{-1}$ using the conjugate gradient algorithm.}
	    \label{tab:linear_solves}
	\end{table}

The communication requirement depends on specifics of the hardware on which the algorithm is executed. With an eye toward high performance computing, we assume that there are computational resources available so that we compute matrix-vector products on $d$ different processors (or collection of processors). Then thanks to the independence of the random vectors initializing the algorithm, the matrix-vector products (which involve linear system solves) are  parallelized so that $d$ matrix-vector products are computed simultaneously. This reduces the wall clock time by approximately a factor of $d$ (with some overhead from synchronization and communication). In this case, the communication bottle neck is associated with sharing vectors across processes. Because of its rank-2 Kronecker form, this cost is $\mathcal O(m+n)$ rather than $\mathcal O(p)=\mathcal O(mn)$. 

\subsection{Post-processing and visualizing sensitivities}
\label{sec:post_proc}
The result of the truncated GSVD is the set 
\begin{eqnarray*}
\left\{
\sigma_N,\w_N,
\t_N = \left( \begin{array}{c}
a \u_N \\
\u_N \otimes \z_0
\end{array} \right)
+
\left( \begin{array}{c}
b_N \u_0 \\
\u_0 \otimes \z_N
\end{array} \right)
\right\}_{N=1}^k
\end{eqnarray*}
where $\sigma_N \in \R$ are the singular values, $\w_N \in \R^n$ are the left singular vectors, and $\t_N \in \R^p$ are the right singular vectors of the sensitivity matrix~\eqref{eqn:dis_sen_op}, i.e. $\nabla_{\t} \F(\vec{0}) \t_N = \sigma_N \w_N$. The right singular vectors are defined by $\u_0,\u_N \in \R^m$, $\z_0,\z_N \in \R^n$, $a,b_N \in \R$ in the rank-2 Kronecker form and their span defines the subspace of model discrepancy perturbations $\d(\z,\t_N)$ which have the greatest influence on the solution of the optimization problem (perturbing $\overline{\z}$ by $\sigma_N \w_N$). 

To process and visualize these perturbations we evaluate $\d$~\eqref{eqn:delta_kron}, centered at the nominal solution $\overline{\z}$, for an arbitrary linear combination, defined by $\c \in \R^k$, of right singular vectors. This yields the expression 
\begin{align}
\label{eqn:eval_delta}
\d \left(\z,\sum\limits_{N=1}^k c_N \t_N \right) &= (a+\overline{\z}^T \M_z \z_0) \vec{U} \c + ( \c^T \b + \overline{\z}^T \M_z \vec{Z} \c ) \u_0 \\
&+ (\z-\overline{\z})^T \M_z \z_0 \vec{U} \c + (\z-\overline{\z})^T \M_z \vec{Z} \c \u_0 \nonumber
\end{align}
where the right singular vectors are represented using $\vec{U} = \left( \u_1,\u_2,\dots,\u_k \right) \in \R^{m \times d}$, $\vec{Z}=\left( \z_1,\z_2,\dots,\z_k \right) \in \R^{n \times d}$, and $\b=(b_1,b_2,\dots,b_k)^T \in \R^k$. 

Varying $\c \in \R^k$ generates a collection of model perturbations (mapping $\Z_h$ to $\U_h$) defining a $d$ dimensional subspace of operators which causes the greatest change in the optimal solution. Visualizing these model perturbations and corresponding perturbed optimal solutions (linear combinations of $\sigma_N \w_N$'s) provides considerable insight which is used to direct model development, mitigate against worst case model discrepancies, or associate uncertainty with the optimal solution. 

\section{Numerical Results}
\label{sec:numerical_results}
We present three examples to demonstrate our proposed approach. The first, an illustrative example, provides intuition for the use of the sensitivity operator. The second, an optimal source control problem constrained by the nonlinear convection-diffusion-reaction PDE in two dimensions, aids to highlight the computation involved in the proposed algorithm. The third, an optimal boundary heat flux control problem constrained by a thermal-fluid model, demonstrates the physical insight our analysis provides on a nonlinear multi-physics application. For ease of exposition and details for reproducibility, we provide problem parameters for all examples in Table~\ref{tab:numerical_results_problem_params} (in the Appendix) and restrict the text to discussing parameters of greatest relevance to our algorithm.

\subsection{Illustrative example}
Let $S(z)$ be the solution operator for the advection-diffusion equation 
\begin{align*}
& -\frac{d^2 u}{dx^2} + \frac{du}{dx} = z, \qquad u(0) = u'(1) = 0.
\end{align*}
Given a true source $z^\star(x)$ and corresponding true state $T=S(z^\star)$, we consider the optimization problem 
\begin{align}
\label{eq:illustrative_ex}
& \min_{z} \frac{1}{2} \int_\Omega (\tilde{S}(z) - T)^2
\end{align}
where $\tilde{S}(z)$ is the solution operator for the diffusion equation
\begin{align*}
-\frac{d^2 u}{dx^2} = z, \qquad u(0) = u'(1) = 0.
\end{align*}
This emulates a common scenario in practice where an approximate model (the diffusion equation) is used to facilitate optimization because the high-fidelity model (the advection-diffusion equation) is either unknown or is computationally intractable. We analyze the sensitivity of the optimal solution of~\eqref{eq:illustrative_ex} with respect to model discrepancy, which in this illustrative example is caused by the absence of the advection term $u'$.

In practice, the discrepancy $S-\tilde{S}$ is not known. In such cases, computing the leading singular values and vectors of $\mathcal F_\delta'(\delta_0)$ show which forms of model discrepancy will cause the greatest change in the optimal solution, and how it will be changed. To illustrate the mathematical properties of the post-optimality sensitivity operator $\mathcal F_\delta'(\delta_0)$, we compute $S-\tilde{S}$ explicitly since the high-fidelity model may be executed efficiently. The left panel of Figure~\ref{fig:illustrative} displays the discrepancy in the state, $S-\tilde{S}$. Its right panel shows $\overline{z}$, the solution of~\eqref{eq:illustrative_ex}, $z^\star$, the high-fidelity solution, and $\overline{z} + \mathcal F_\delta'(\delta_0)(S-\tilde{S})$, the sensitivity prediction of optimal solution given the model discrepancy. We observe that the sensitivity prediction of the optimal solution is a good approximation of the high-fidelity solution. In general, $\overline{z} + \mathcal F_\delta'(\delta_0)(S-\tilde{S})$ will not coincide with $z^\star$ since $\mathcal F_\delta'(\delta_0)$ is a linearization; however, it provides valuable insight for the lower-fidelity problem which completely lacks the advection term.

\begin{figure}[h]
\centering
  \includegraphics[width=0.35\textwidth]{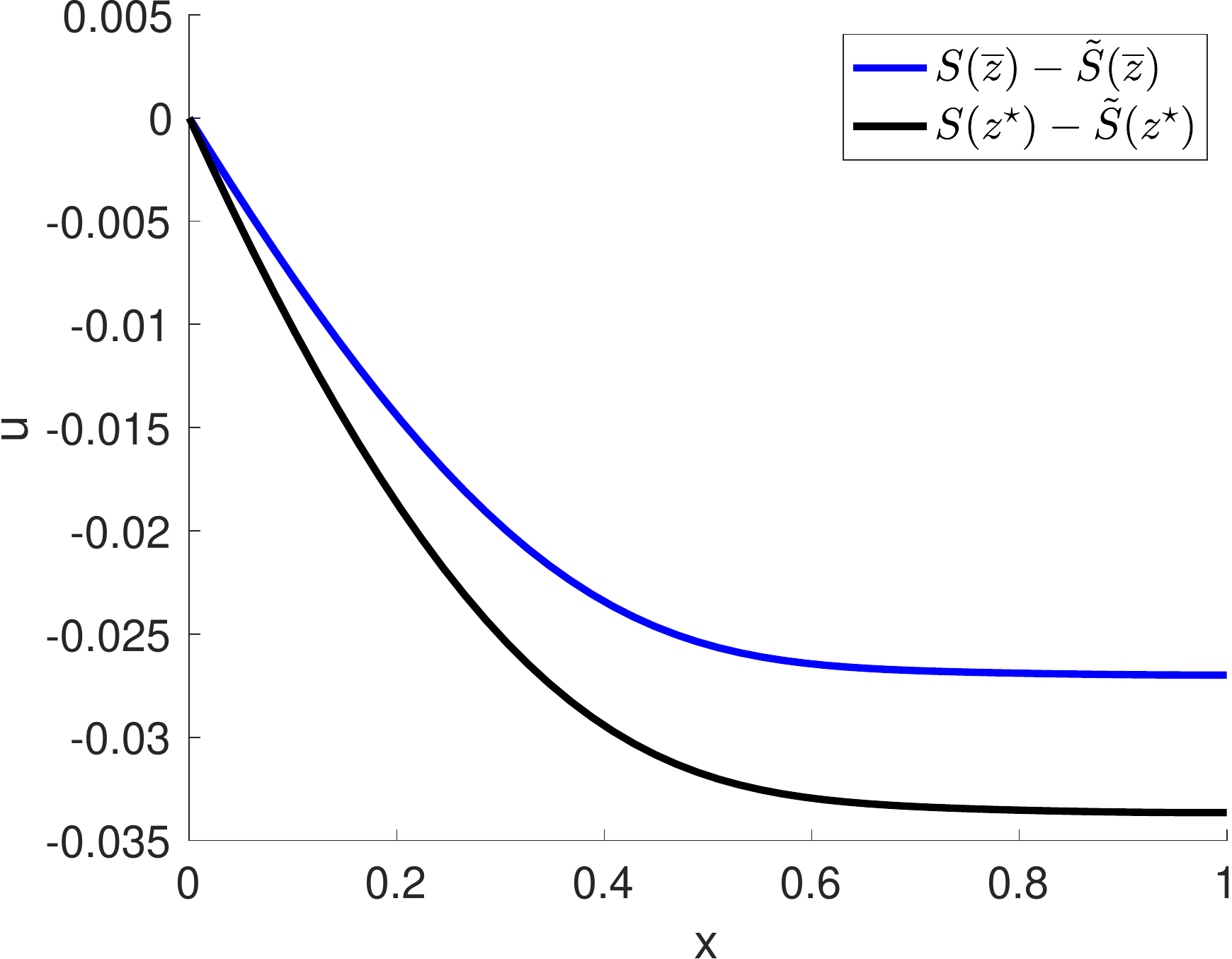}
  \hspace{1 mm}
    \includegraphics[width=0.35\textwidth]{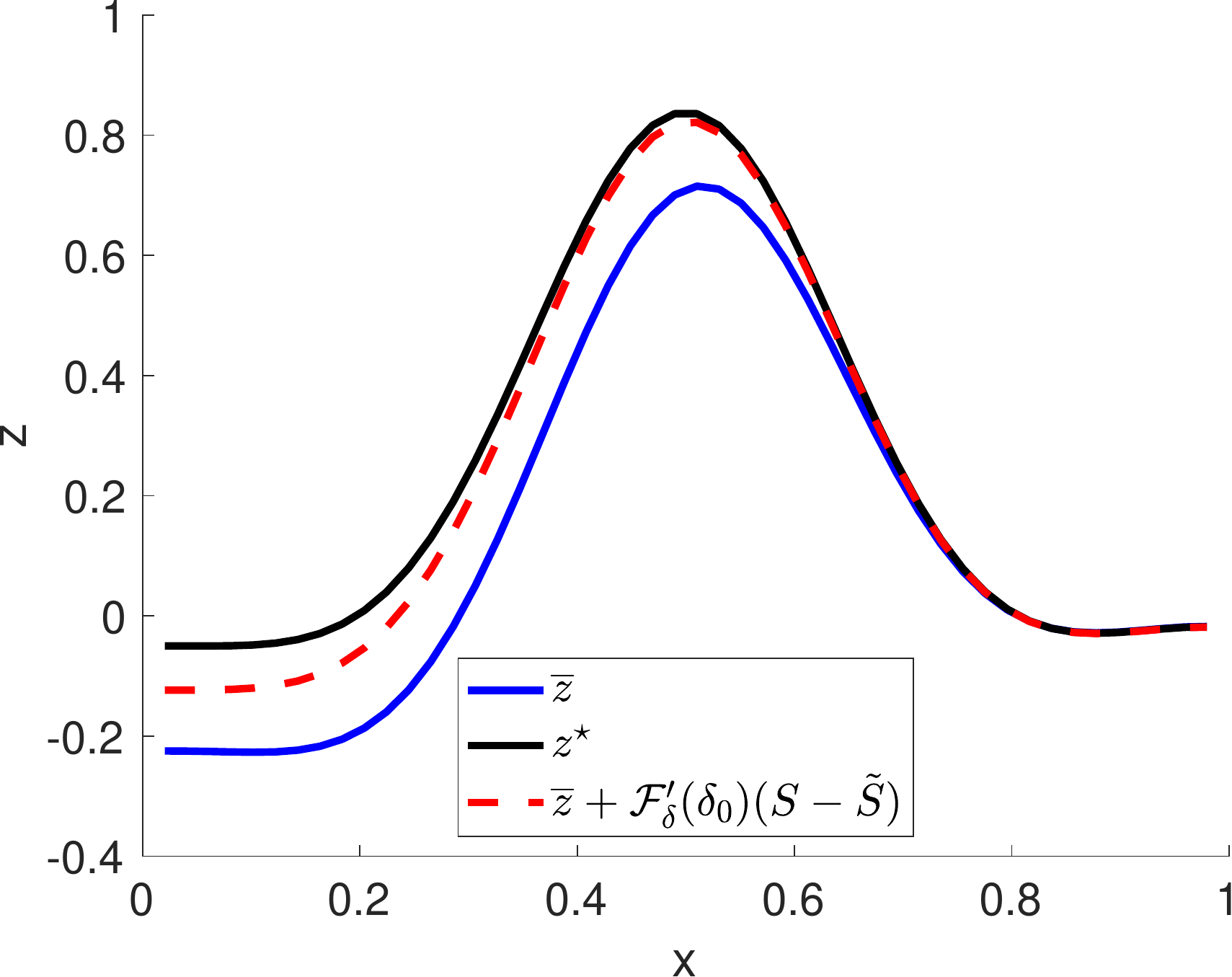}
  \caption{Left: the model discrepancy $S-\tilde{S}$ evaluated at $z^\star$ and $\overline{z}$; right: optimal solution $\overline{z}$ alongside the high-fidelity solution $z^\star$ and the sensitivity prediction $\overline{z} + \mathcal F_\delta'(\delta_0)(S-\tilde{S})$. }
  \label{fig:illustrative}
\end{figure}

\subsection{Convection-diffusion-reaction}
Consider the optimal forcing control constrained by the convection-diffusion-reaction PDE. Specifically, the optimization problem 
\begin{align*}
 & \min_{z} \frac{1}{2} \int_\Omega (\tilde{S}(z) - T)^2 +\frac{\beta_1}{2} \int_\Omega z^2+\frac{\beta_2}{2} \int_\Omega \vert \vert \nabla z \vert \vert_2^2 
\end{align*}
where $\tilde{S}(z)$ solves 
\begin{align*}
& -\nu \nabla^2 u + \v \cdot \nabla u = z + R(u) \quad &
\text{ in } \Omega \\
& u = 0 \quad & \text{on } \Gamma_d \\
& \nabla u \cdot \vec{n} = 0 \quad & \text{ on } \Gamma_n .
\end{align*}
The optimization seeks to achieve a target state $T$ constrained by a PDE which depends on the the velocity field $\v$ (depicted in the left panel of Figure~\ref{fig:cdr_opt_sol}) and reaction function $R$. Regularization by a squared Sobolev norm with coefficients $\beta_1$ and $\beta_2$ encourages a smooth optimal controller. The optimal state and controller are shown in the center and right panels of Figure~\ref{fig:cdr_opt_sol}. 

This model, which we have focused on in mathematical abstraction to illustrate properties of the proposed algorithms, is representative of many physical systems. A common source of model discrepancy is associated with the form of the diffusion $\nu$ which we have taken as a simple spatially homogenous isotropic model rather than a more complex model with spatial heterogeneity and anisotropic behavior. Similarly, the reaction function is an idealization of chemistry which is typically modeled by more complex relationships in high-fidelity models. Without knowledge of the high-fidelity system, we compute the leading singular values/vectors of the post-optimality sensitivity operator $\mathcal F_\delta'(\delta_0)$ to understand how the optimal controller will change given different forms of model discrepancy.

\begin{figure}[h]
\centering
 \includegraphics[width=0.23\textwidth]{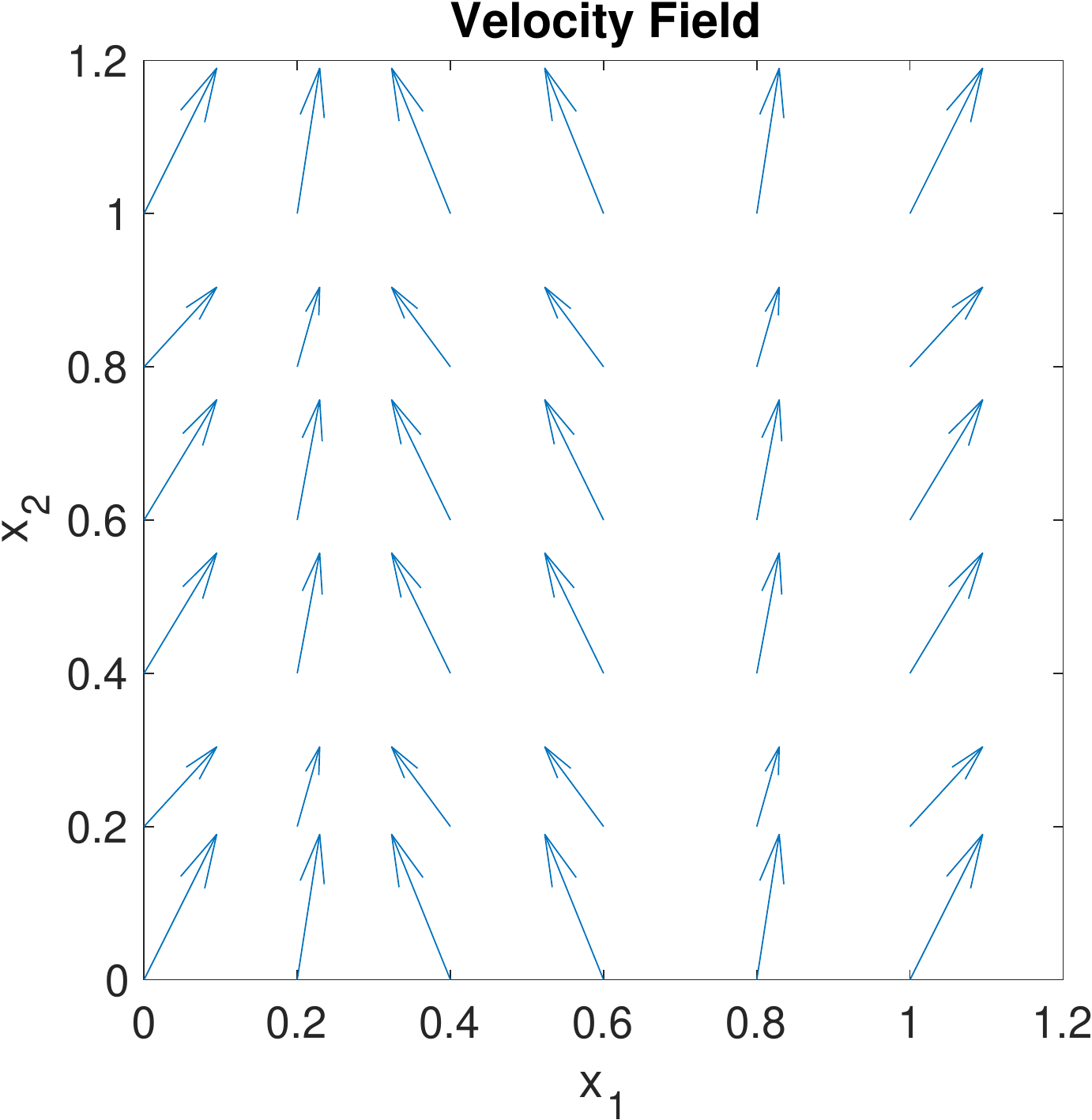}
  \includegraphics[width=0.32\textwidth]{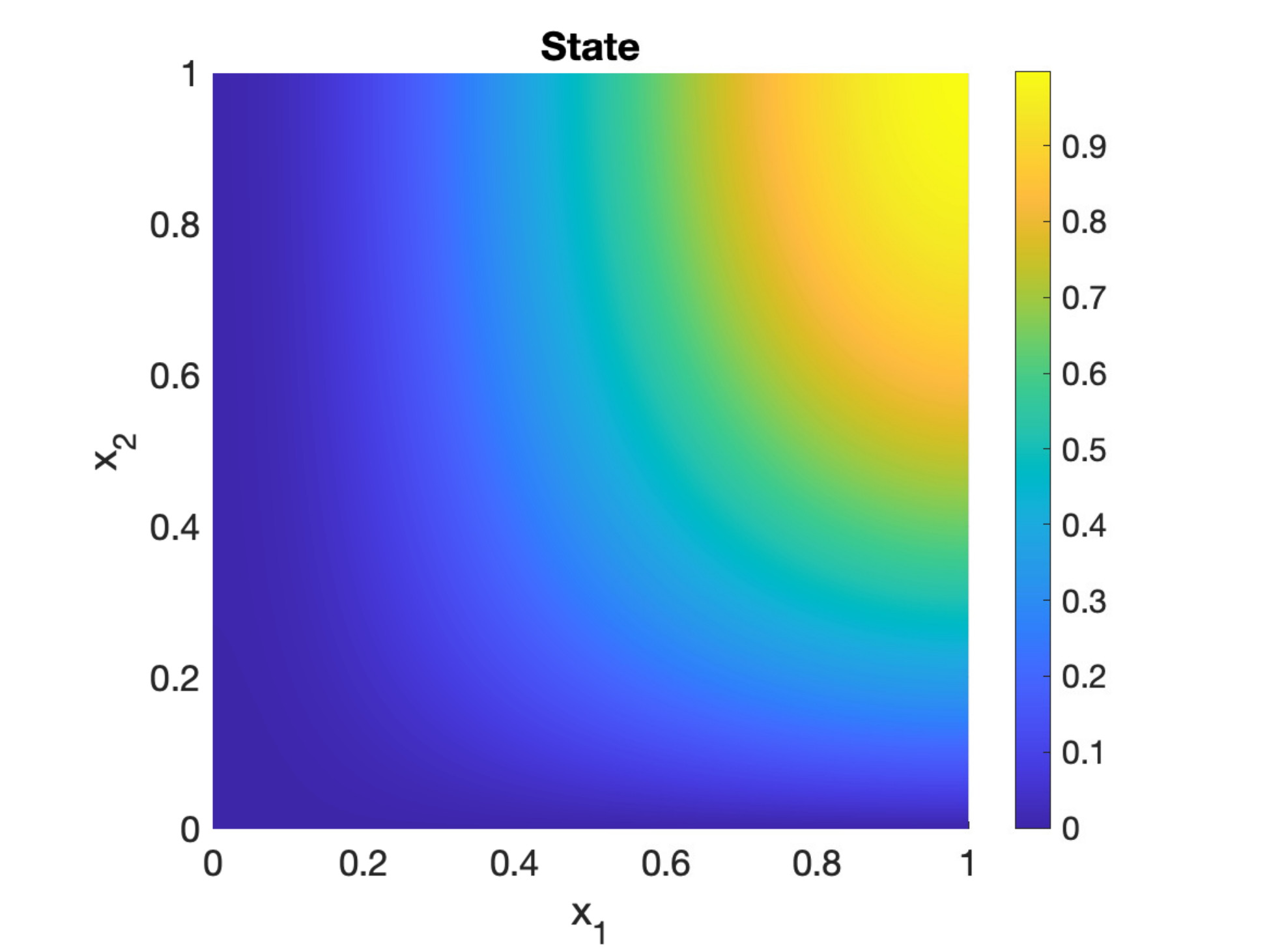}
    \includegraphics[width=0.32\textwidth]{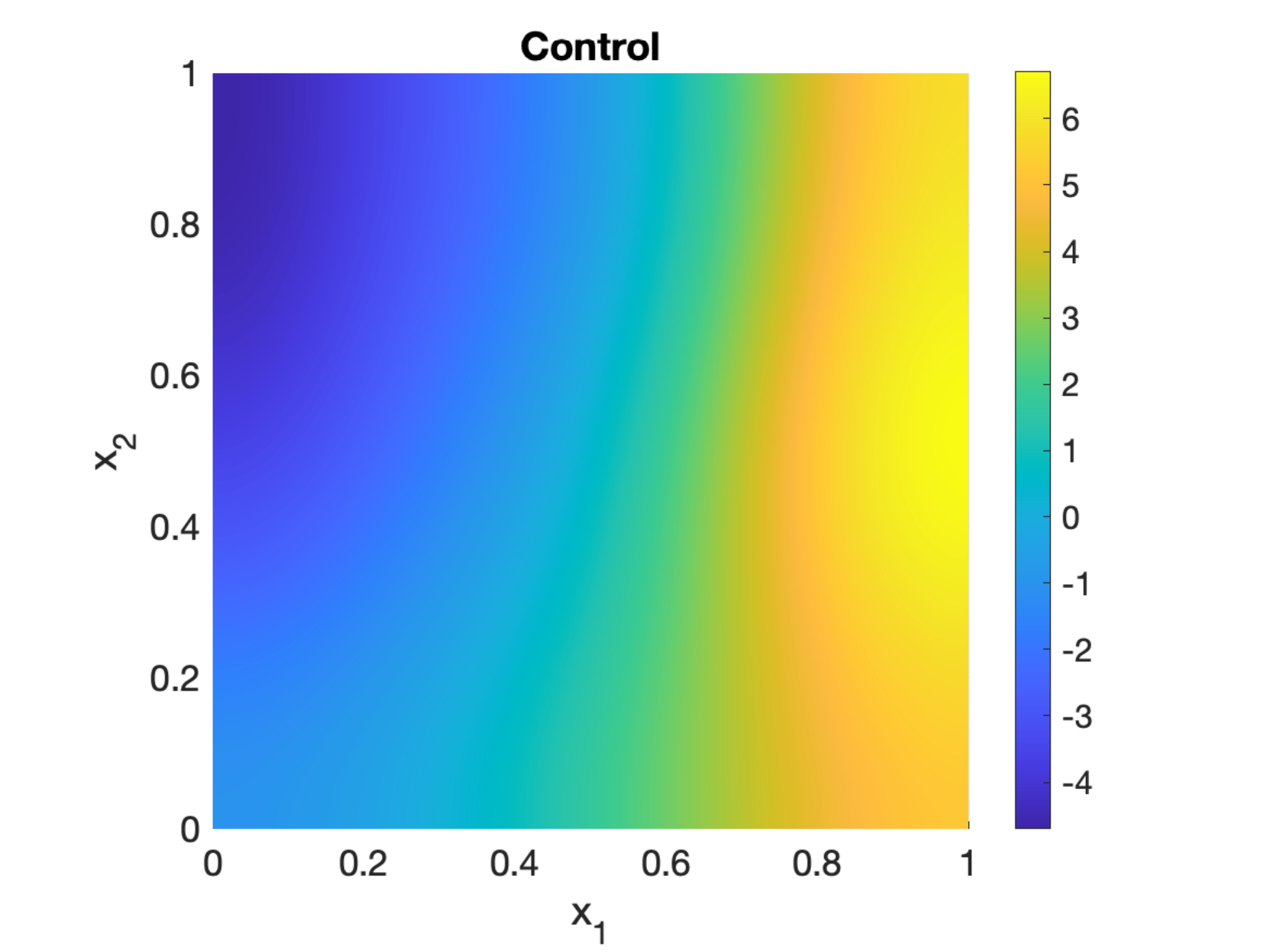}
  \caption{Optimal state solution (left) and optimal controller (right) for the convection-diffusion-reaction problem.}
  \label{fig:cdr_opt_sol}
\end{figure}

\subsection*{HDSA with respect to model discrepancy}

The state inner product weighting matrix $\L$ and controller covariance matrix $\vec{\Gamma}$ define $\M_\t$ and serve to impose physics constraints on the model discrepancy. Since the discrepancy is not constrained in any other way, the choice of inner product is important to identify plausible types of model discrepancy in the leading singular vectors of the post-optimality sensitivity operator. 

We impose that $\delta$ is smooth and that it respect the zero Dirichlet condition on $\Gamma_{d}$. To this end, define $\vec{K} \in \R^{m \times m}$ as
\begin{eqnarray*}
\vec{K}_{i,j} = \epsilon \int \nabla \phi_i \cdot \nabla \phi_j + \int \phi_i \phi_j,
\end{eqnarray*}
which induces the $H^1$ Sobolev norm, thus encouraging smoothness by penalizing the magnitude of the gradient. Our results use $\epsilon=0.001$. Additional numerical tests indicate that the results do not change significantly for other choices of $\epsilon$. The Dirichlet boundary condition is enforced via a soft penalty by defining $\L = \K + \tau \vec{P}^T \vec{P}$, where $\vec{P} \in \R^{m_b \times m}$ is the projector for the state variable coordinates onto the $\Gamma_d$ boundary coordinates. Setting the penalty coefficient $\tau=50$ was found to be sufficient to ensure that the resulting $\delta$'s respect the boundary condition. 

Recall that $\vec{\Gamma}$ is the covariance matrix for the coordinates discretizing $z \in L^2(\Omega)$. We define $\vec{\Gamma}$ by discretizing the covariance operator of a Gaussian random field on $\Omega$ and leverage the fact that the inverse of the elliptic operator $\mathcal A = \frac{1}{\alpha^2}(-\beta \Delta + \mathcal I)$, where $\Delta$ is the Laplacian and $\mathcal I$ is the identity operator, is a self-adjoint positive definite trace class operator for all $\beta>0$, and hence defines a valid covariance operator. We take $\beta=10^{-6}$ and $\alpha=1$ for the subsequent results. Numerical experiments (omitted for conciseness) indicate that taking $\alpha$ too small will result in model discrepancies which are small at $\overline{\z}$ but much larger for perturbed sources. Taking larger values of $\alpha$ will favor discrepancies $\d$ which are nearly constant functions of $\z$.

Algorithm~\ref{alg:GSVD} was executed with a target rank of $k=54$, oversampling factor $\ell=8$, and $q=1$ subspace iterations. The resulting singular values are displayed in Figure~\ref{fig:cdr_sing_vals}. We observe around one order of magnitude decrease in the singular values, which, in the context of low rank approximations is insufficient, but for our sensitivity analysis purposes is adequate for identifying greatest sources of uncertainty contributed by model discrepancy.

In this example, $m=n=10^4$ and hence the dimension of $\t \in \R^p$ is $p>10^8$. However, the leading singular pairs of sensitivity operator $\nabla_{\t} \F(\vec{0}) \in \R^{n \times p}$ are well approximated with a computational cost of $\mathcal O(100)$ Hessian inversions, many of which are executed asynchronistically. This highlights the computational advantages of our proposed approach. 
 
\begin{figure}[h]
\centering
  \includegraphics[width=0.3\textwidth]{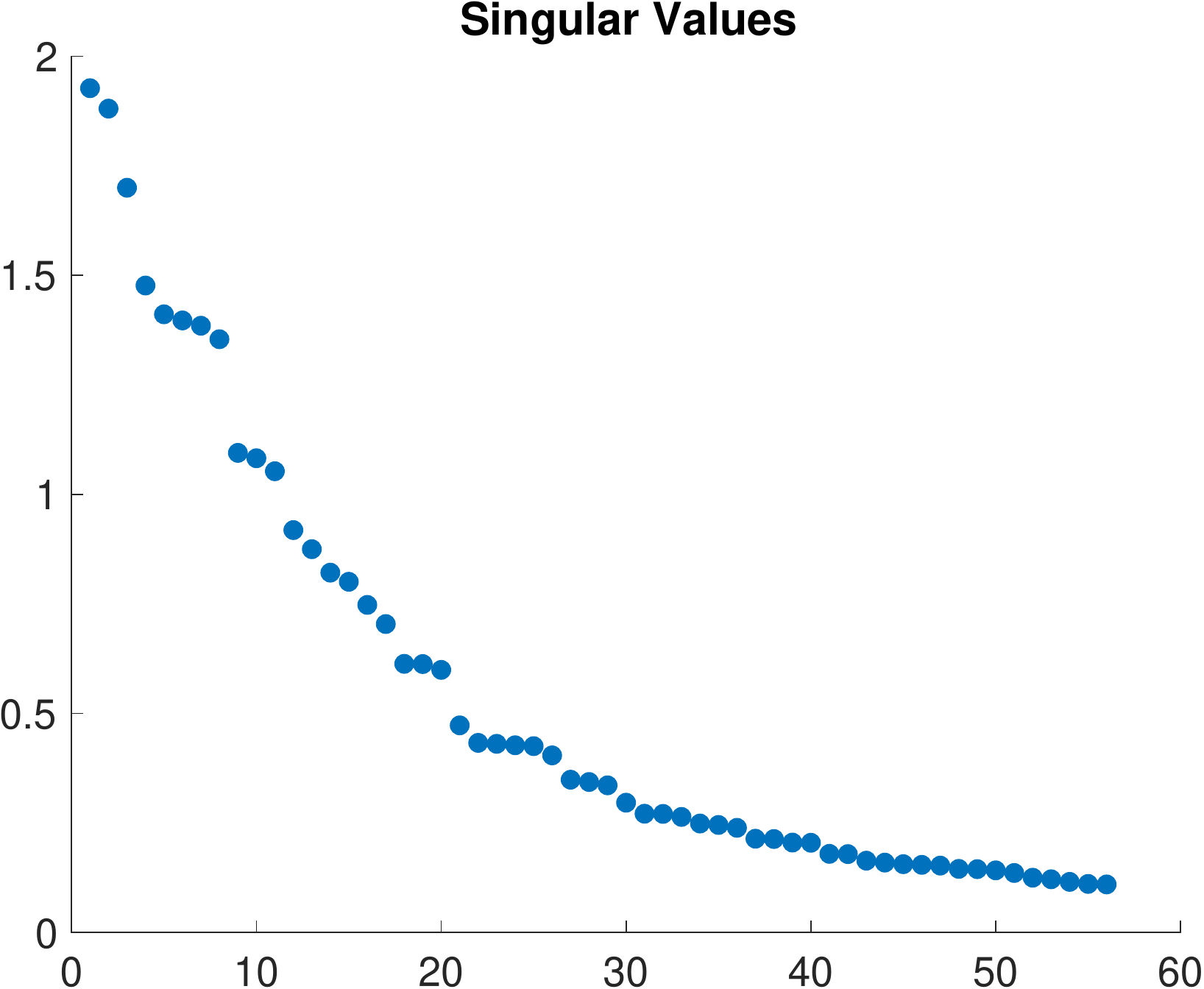}
  \caption{Singular values for the convection-diffusion-reaction problem.}
  \label{fig:cdr_sing_vals}
\end{figure}

Using~\eqref{eqn:eval_delta}, Figure~\ref{fig:cdr_sing_vectors} shows the model discrepancies and the corresponding perturbations of the optimal solutions for the two largest singular values. Comparing $\d(\overline{\z},\t_N)$ in the left column with $\d(\overline{\z}+\sigma_N \w_N,\t_N)$ in the center column highlights the dependence of the model discrepancy on $\z$. We observe that the leading model discrepancies at $\z=\overline{\z}$ are most prominent in the lower region of the domain and propagate upward when evaluated at the perturbed solution $\z=\overline{\z}+\sigma_N \w_N$. The corresponding controller perturbations are concentrated in the lower region of the domain as the velocity field disperses the forcing term upward.

\begin{figure}
  \centering
  \begin{tabular}{M{.025\textwidth}M{.25\textwidth}M{.25\textwidth}M{.25\textwidth}}
$N$ & $\delta(\overline{z},\theta_N)$ & $\delta(\overline{z}+\sigma_N w_N,\theta_N)$ & $\sigma_N w_N$ \\
 1  & 
  \includegraphics[width=0.25\textwidth]{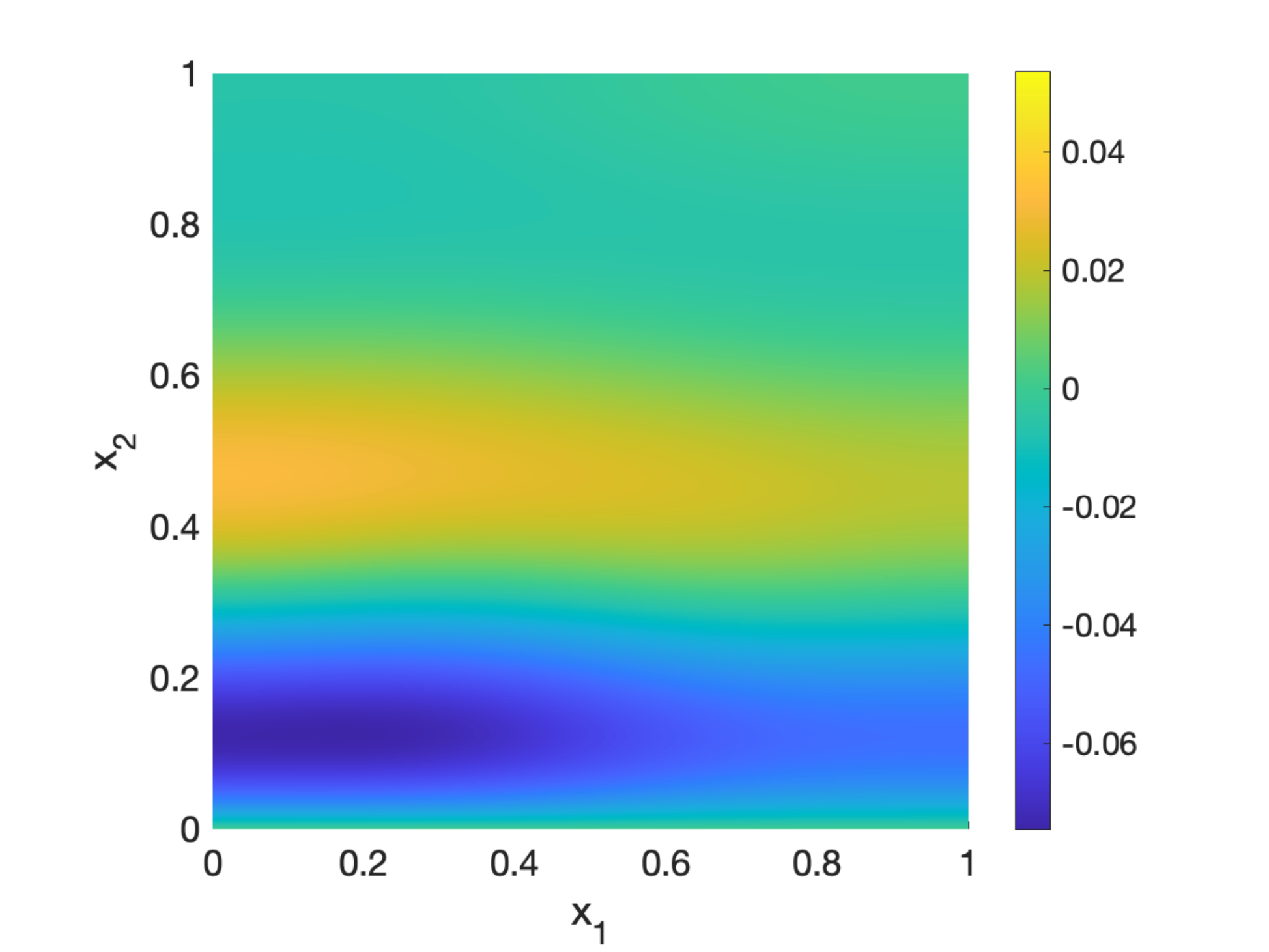} &
    \includegraphics[width=0.25\textwidth]{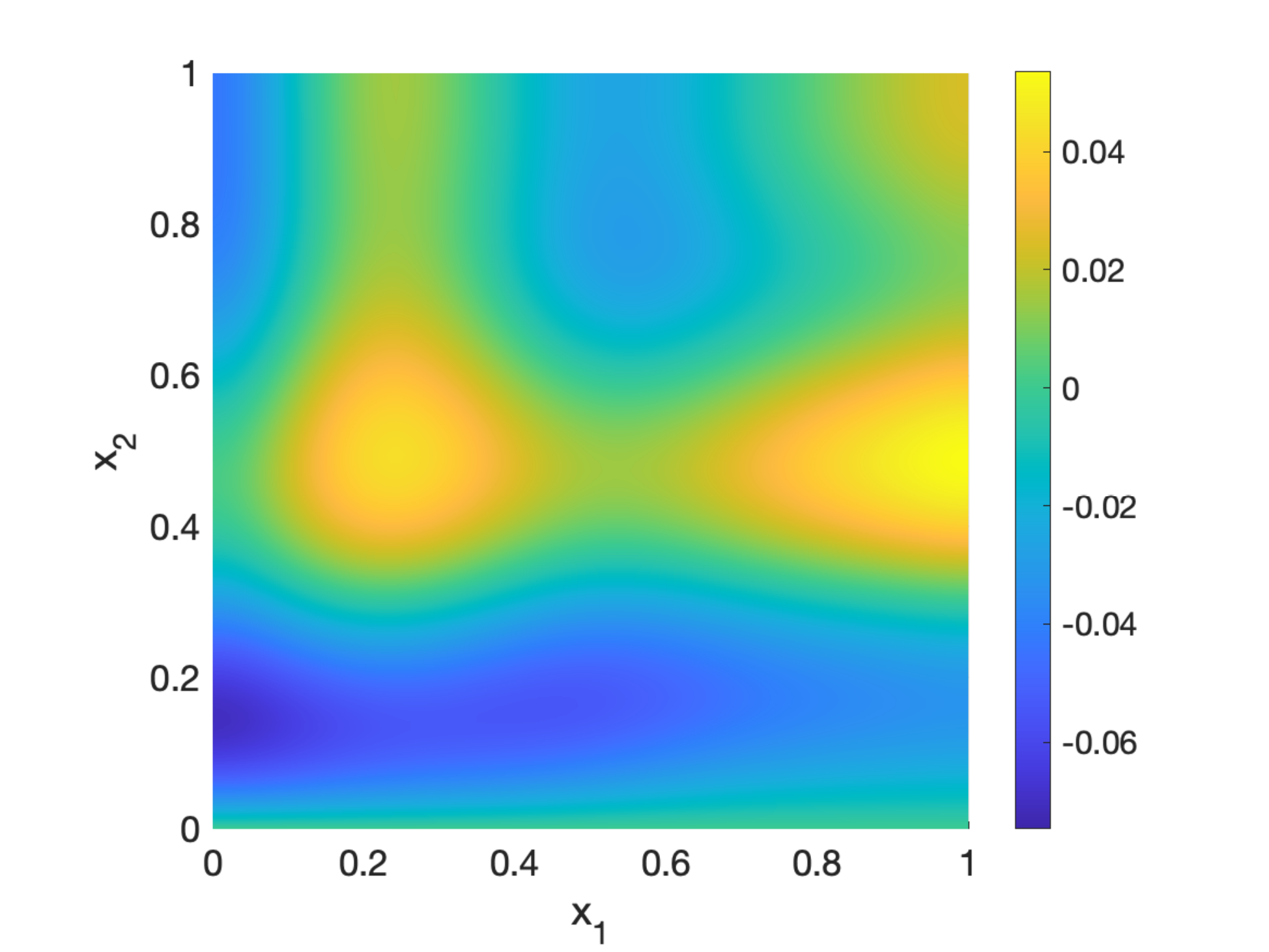} &
        \includegraphics[width=0.25\textwidth]{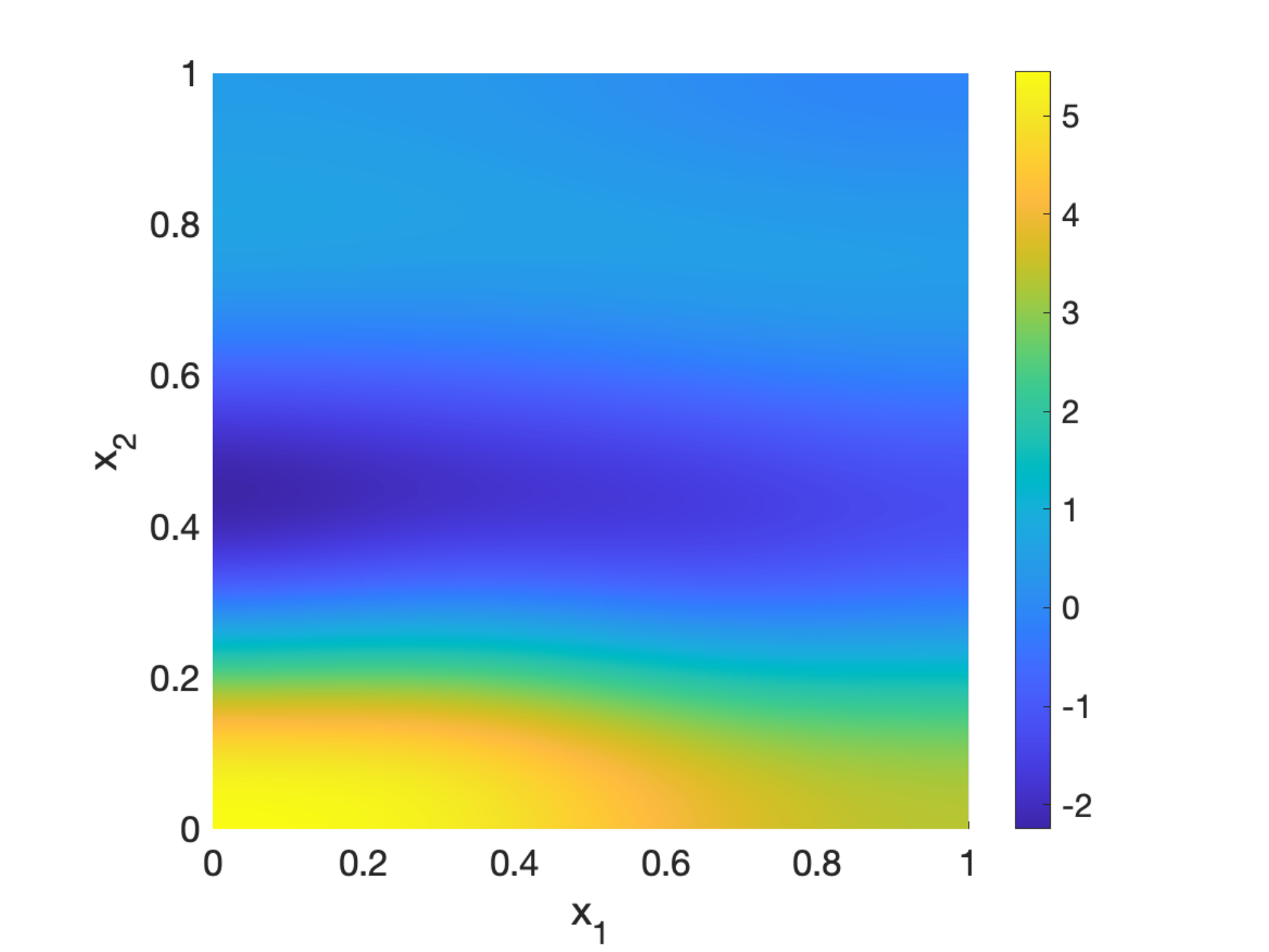} \\
         2&
          \includegraphics[width=0.25\textwidth]{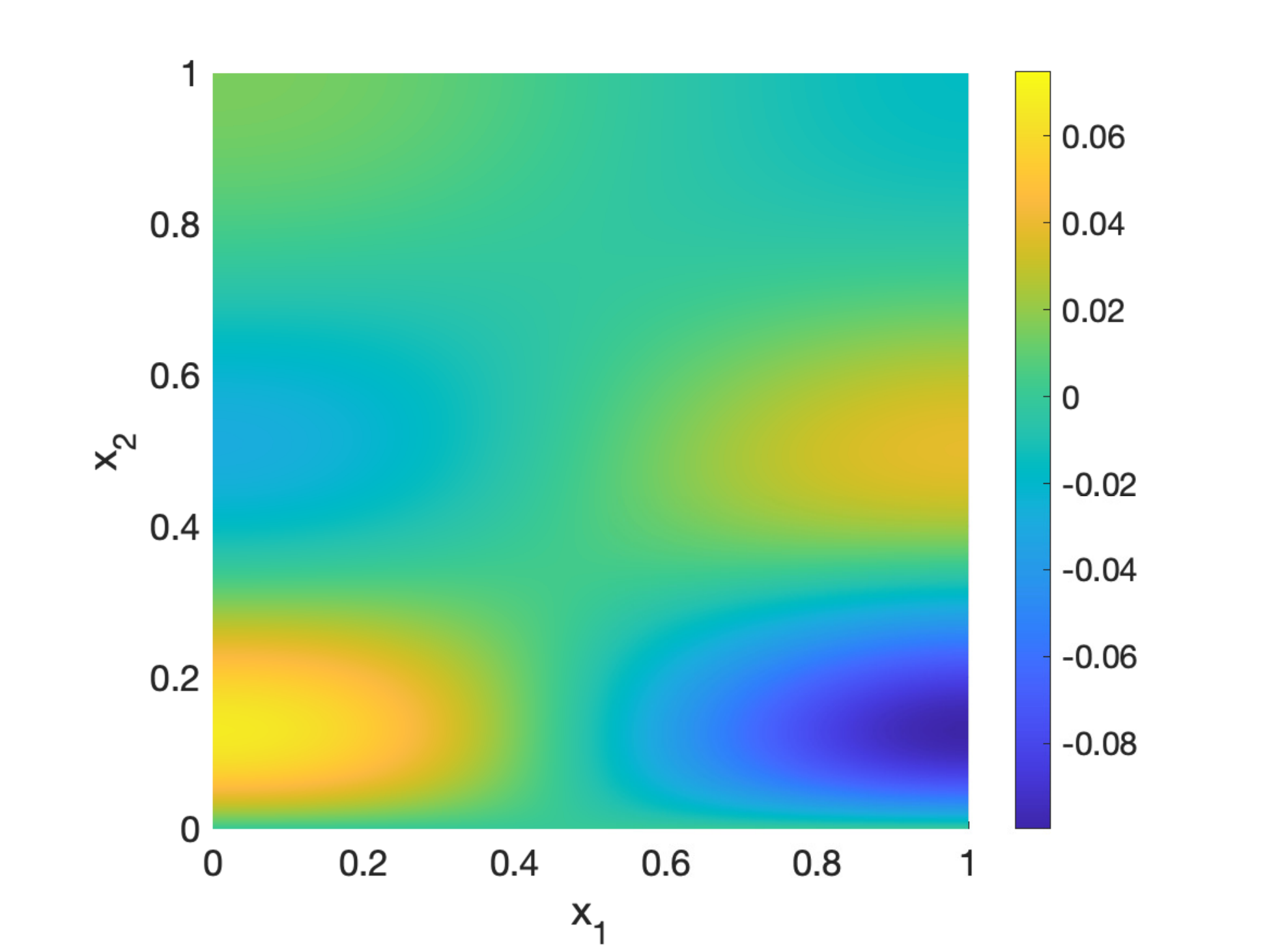} &
    \includegraphics[width=0.25\textwidth]{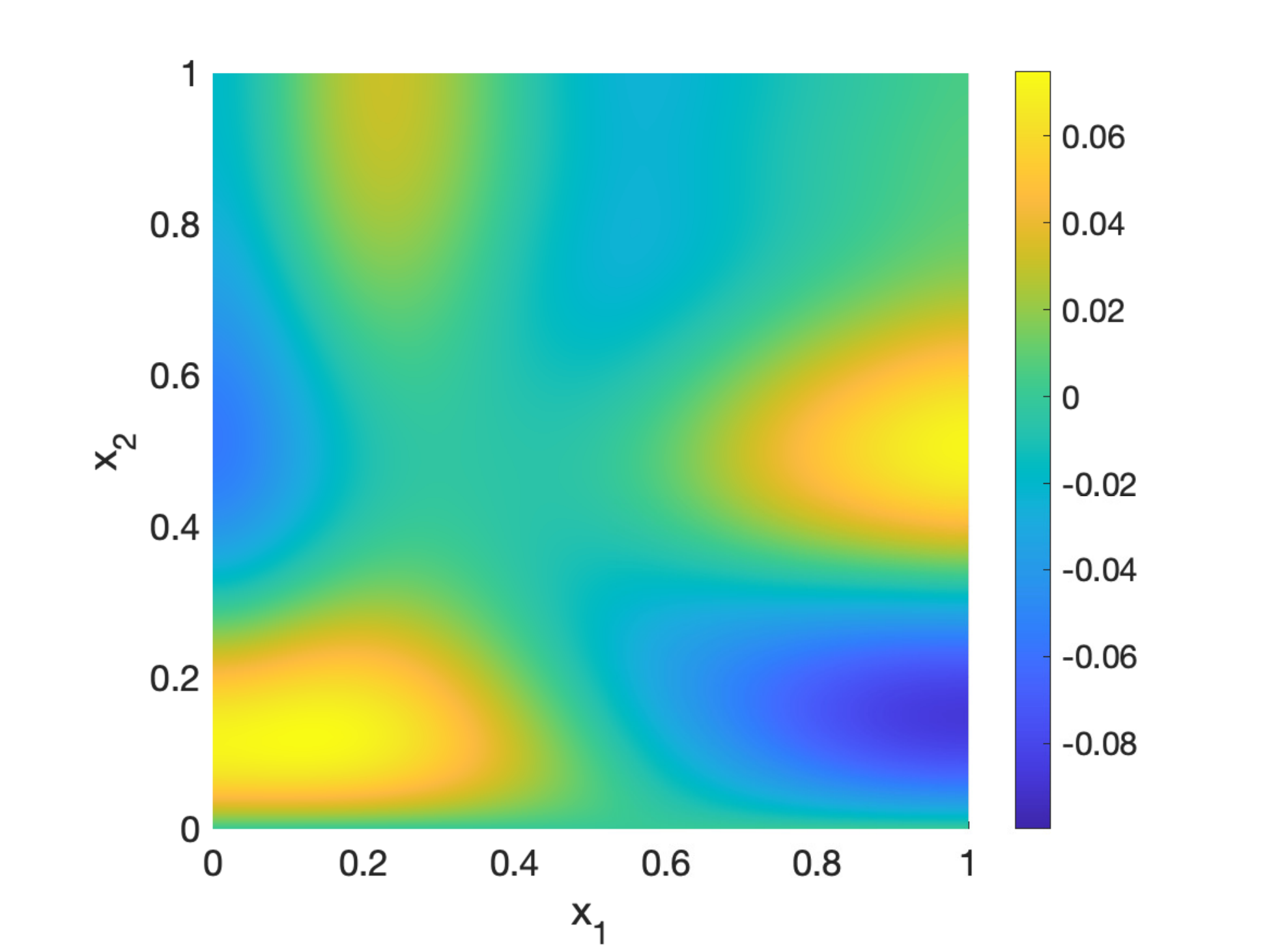} &
        \includegraphics[width=0.25\textwidth]{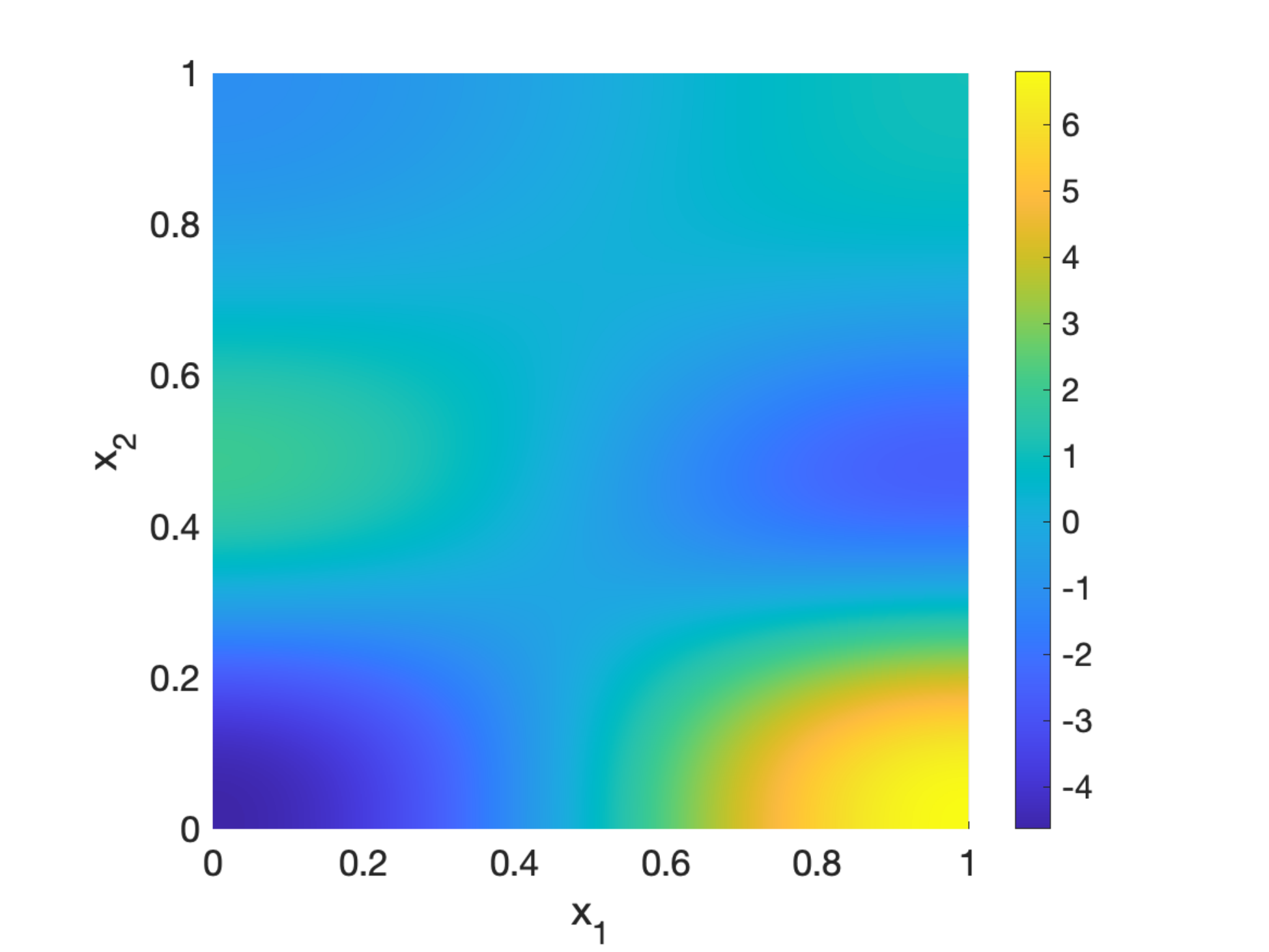} \\
  \end{tabular}
    \caption{Leading model discrepancies, $\d(\z,\t_N)$, and the corresponding perturbations of the optimal solution, $\sigma_N \w_N$, for the convection-diffusion-reaction problem. From top to bottom, each row corresponds to a singular vector, $\t_N$, $N=1,2$. Across each row, the left panel is $\d(\overline{\z},\t_N)$, the center is $\d(\overline{\z}+\sigma_N \w_N)$, and the right is $\sigma_N \w_N$.}
  \label{fig:cdr_sing_vectors}
\end{figure}

\subsection{Thermal-fluid}
In this subsection we consider control of the nonlinear multi-physics system modeled by the Boussinesq flow equations, an emulation of a chemical vapor deposition reactor. Reactant gases are injected in the top of a reactor and flow downwards to create an epitaxial film on the bottom. Vorticity created by buoyancy-driven convection inhibit some gases from reaching the bottom of the reactor. We control thermal fluxes on the side walls of the reactor to minimize the vorticity. Consider the optimization problem,

\begin{align*} 
& \min\limits_{z} \frac{1}{2} \int_{\Omega} (\nabla \times \tilde{v}(z))^2 + \frac{\gamma}{2} \int_{\Gamma_c} z^2  
\end{align*}
where $\tilde{S}(z)=(\tilde{v}(z),\tilde{p}(z),\tilde{T}(z))$ solves
\begin{align*} 
& -\frac{1}{Re} \nabla^2 v + (v \cdot \nabla ) v + \nabla p + \eta T g = 0 & \text{ in } \Omega  \nonumber \\
& \nabla \cdot v = 0 & \text{ in } \Omega  \nonumber \\
& - \kappa \Delta T + v \cdot \nabla T = 0 & \text{ in } \Omega  \nonumber \\
&T = 0  \qquad \text{and} \qquad v = v_i & \text{on } \Gamma_i  \nonumber \\
& \kappa \nabla T \cdot \vec{n} = 0 \qquad \text{and} \qquad v=v_o & \text{on } \Gamma_o  \nonumber \\
& T = 1 \qquad \text{and} \qquad v=0 & \text{on } \Gamma_b  \nonumber \\
&  \kappa \nabla T \cdot \vec{n} = T-z \qquad \text{and} \qquad v=0 & \text{on } \Gamma_c . \nonumber 
\end{align*}
The left panel of Figure~\ref{fig:tf_velocity_field} depicts the domain and boundaries. The state consists of horizontal ($x_1$) and vertical ($x_2$) velocities which we denote as $v=(v_1,v_2)$, the pressure $p$, and the temperature $T$. The controller $z$ is a function defined on the left and right boundaries. Figure~\ref{fig:tf_velocity_field} displays the uncontrolled (center) and controlled (right) velocity fields. The undesired vorticity is observed in the uncontrolled velocity field and are reduced by the control strategy. The optimal states are displayed in Figure~\ref{fig:tf_states} and the corresponding optimal controllers are shown in Figure~\ref{fig:tf_control}.

\begin{figure}
\centering
\includegraphics[width=0.35\textwidth]{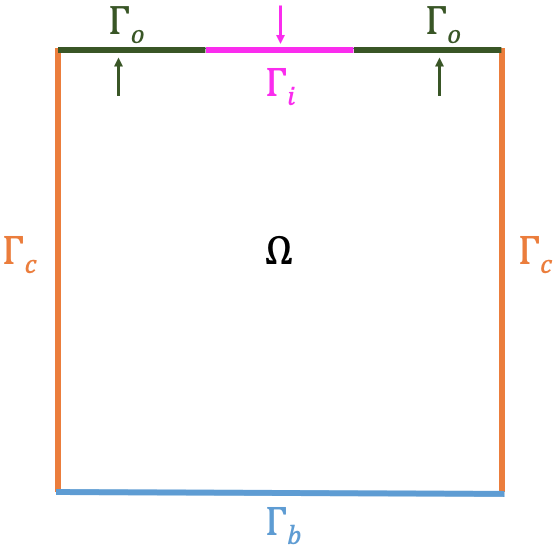}
 \includegraphics[width=0.31\textwidth]{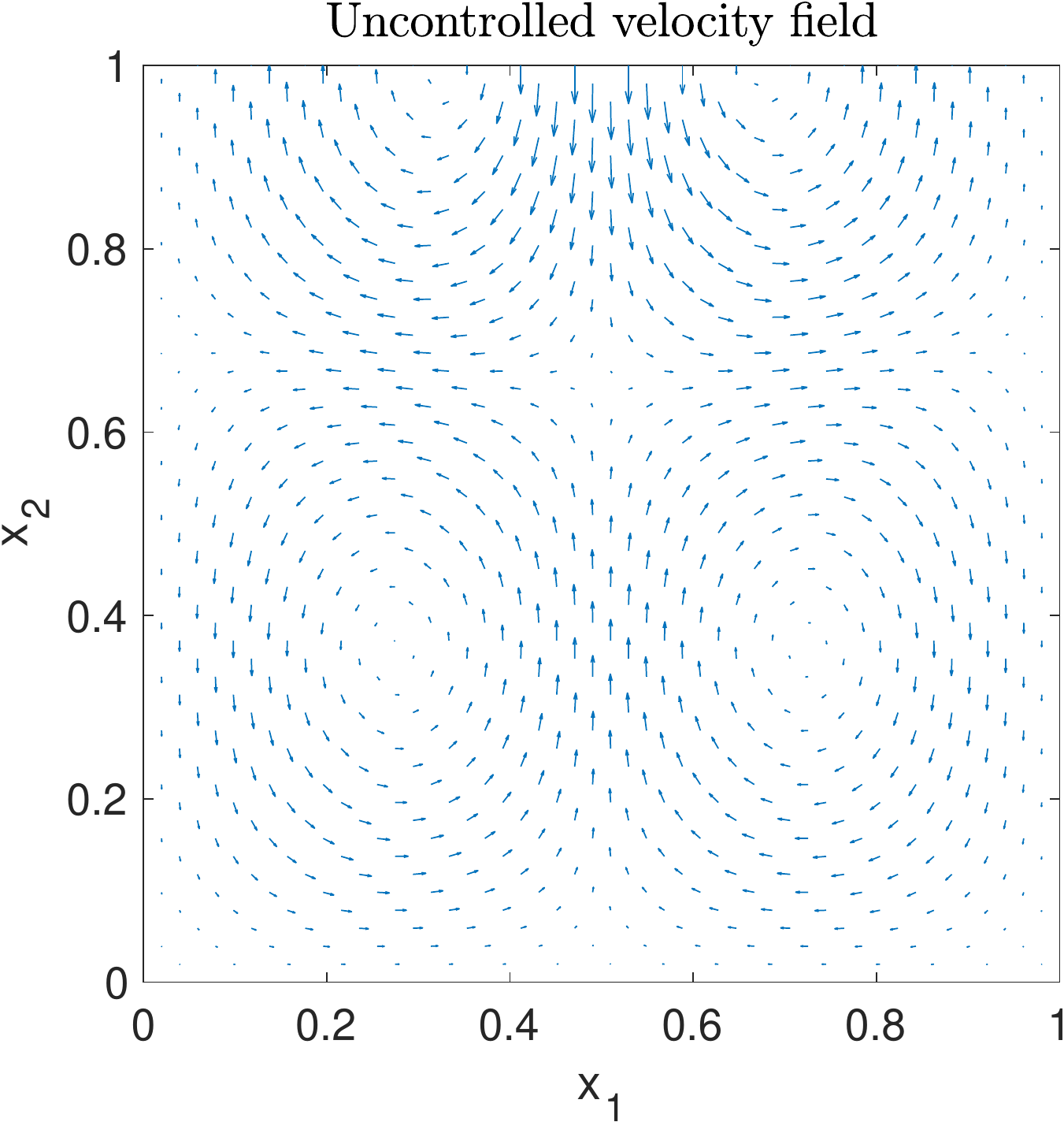} 
 \includegraphics[width=0.31\textwidth]{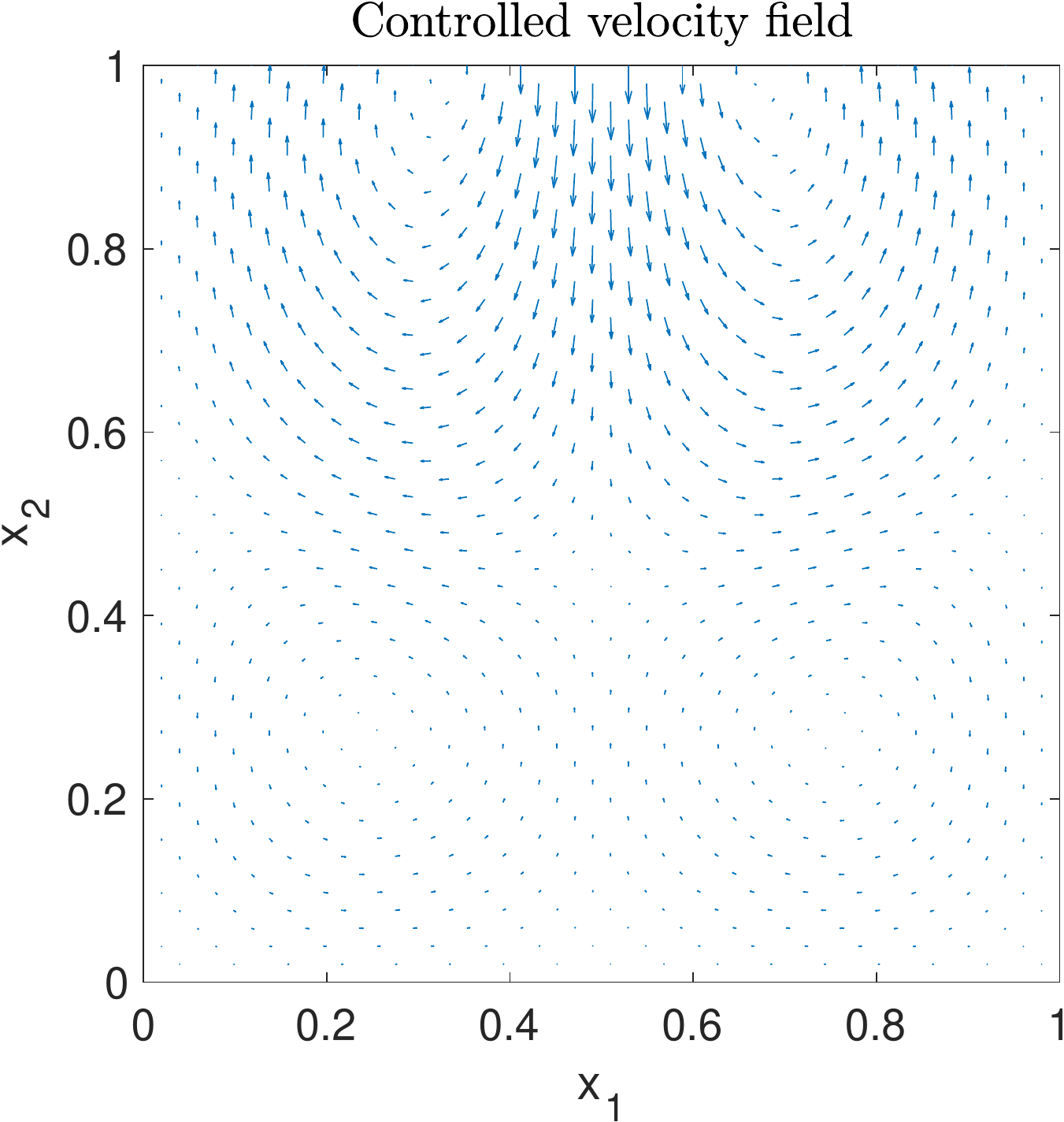}
\caption{Left: domain and boundaries; center and right: uncontrolled and controlled velocity fields, respectively, for the thermal-fluid problem.}
\label{fig:tf_velocity_field}
\end{figure}
	
\begin{figure}[h]
\centering
  \includegraphics[width=0.24\textwidth]{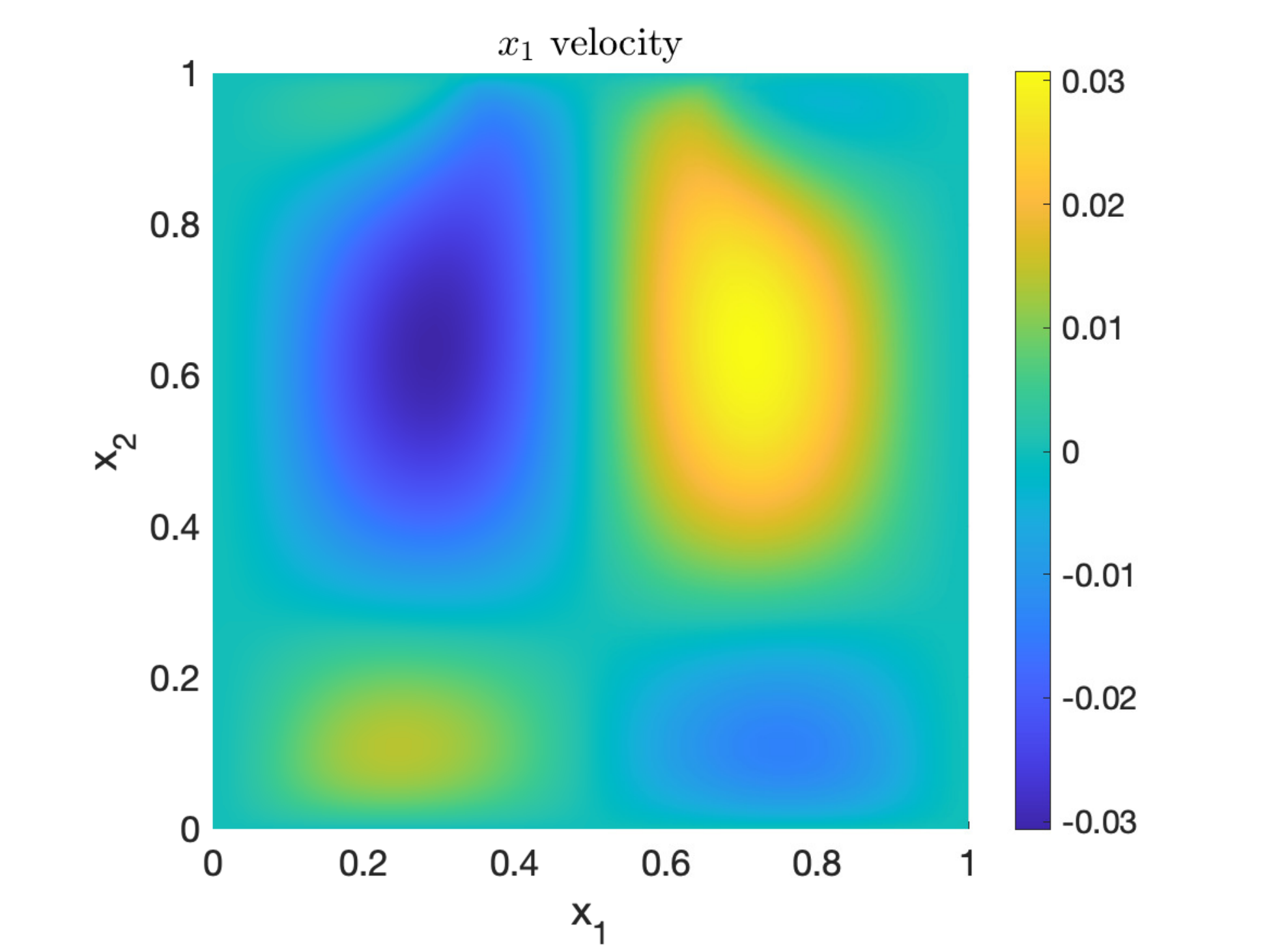}
    \includegraphics[width=0.24\textwidth]{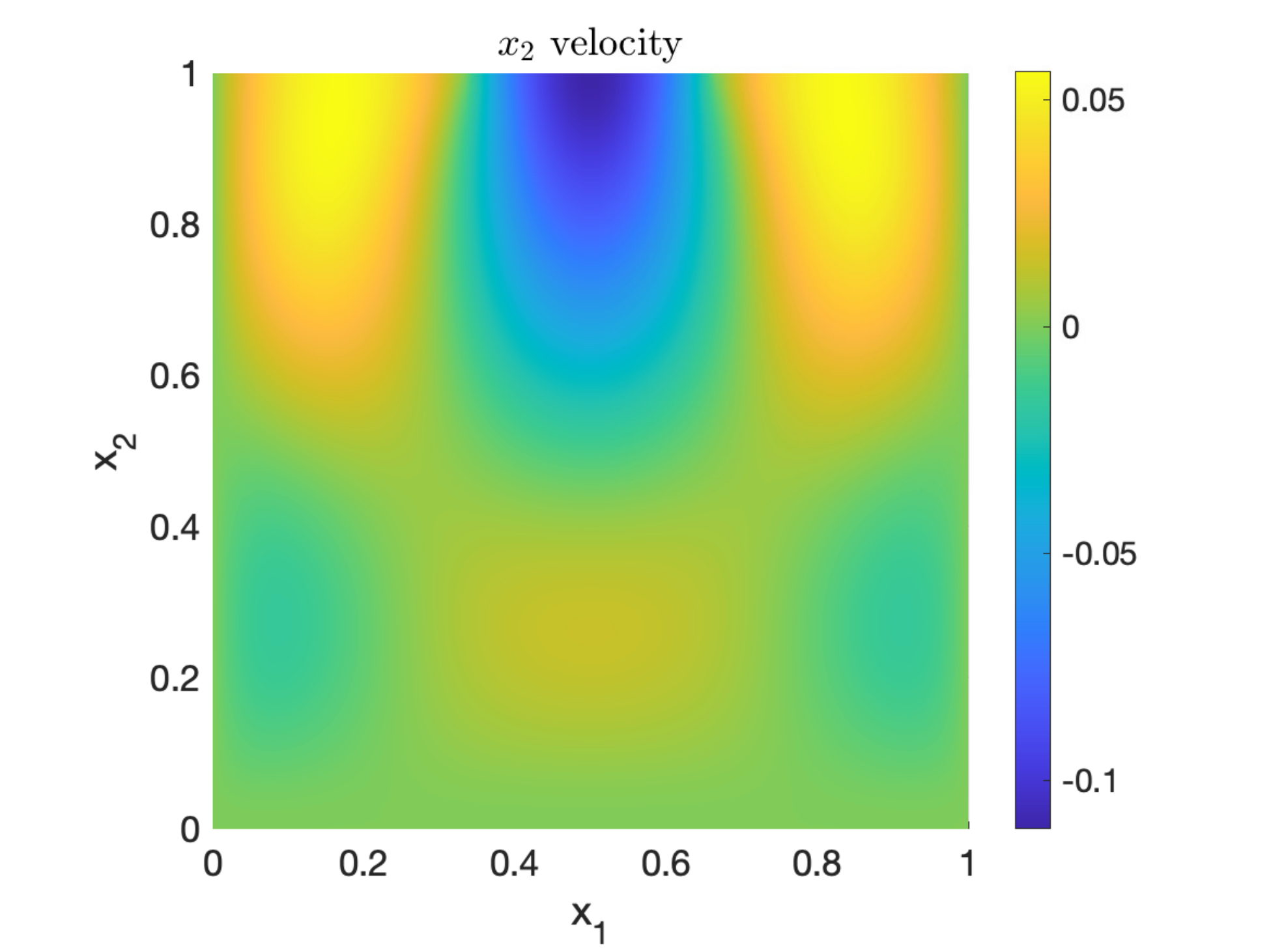}
      \includegraphics[width=0.24\textwidth]{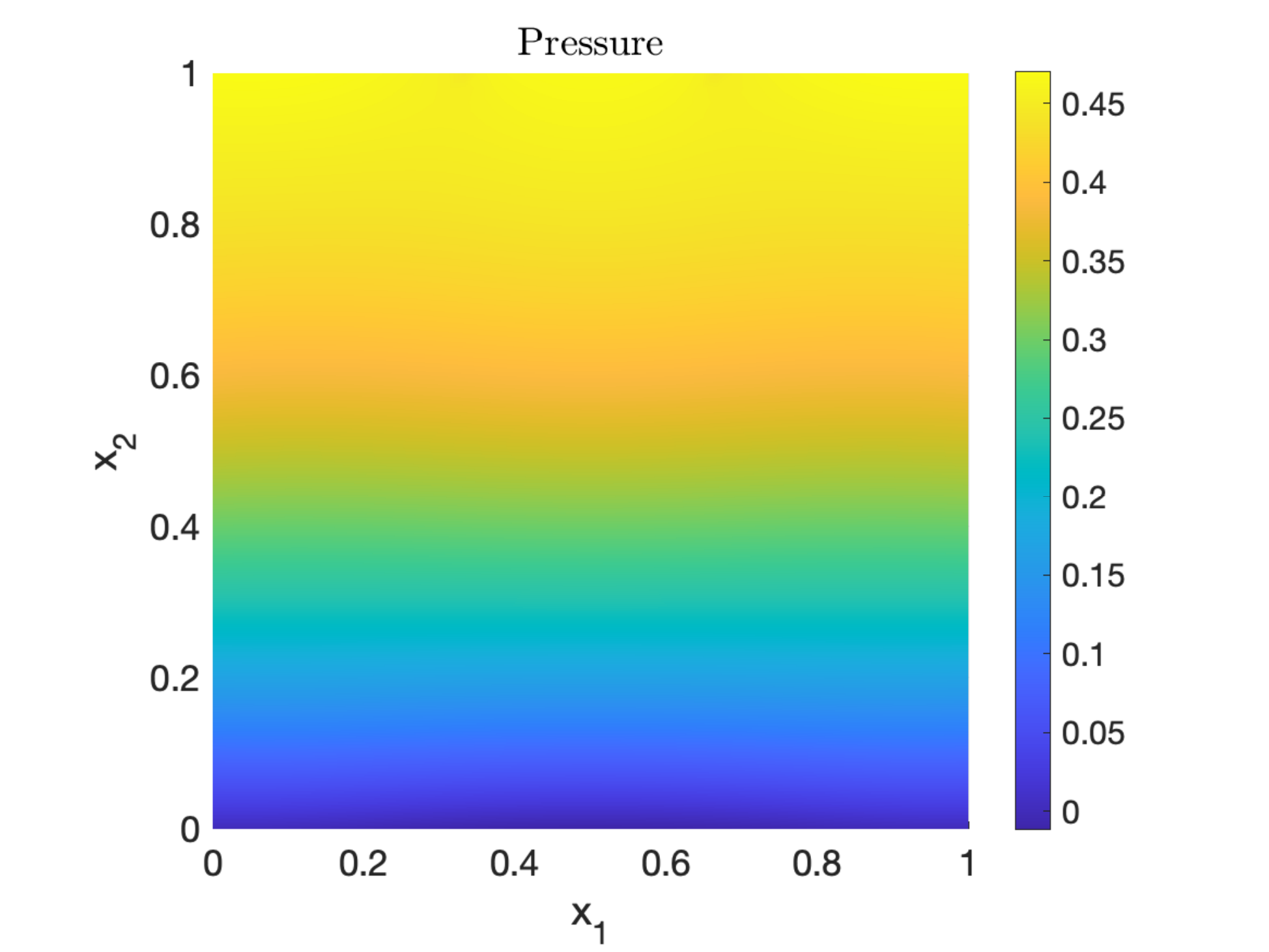}
        \includegraphics[width=0.24\textwidth]{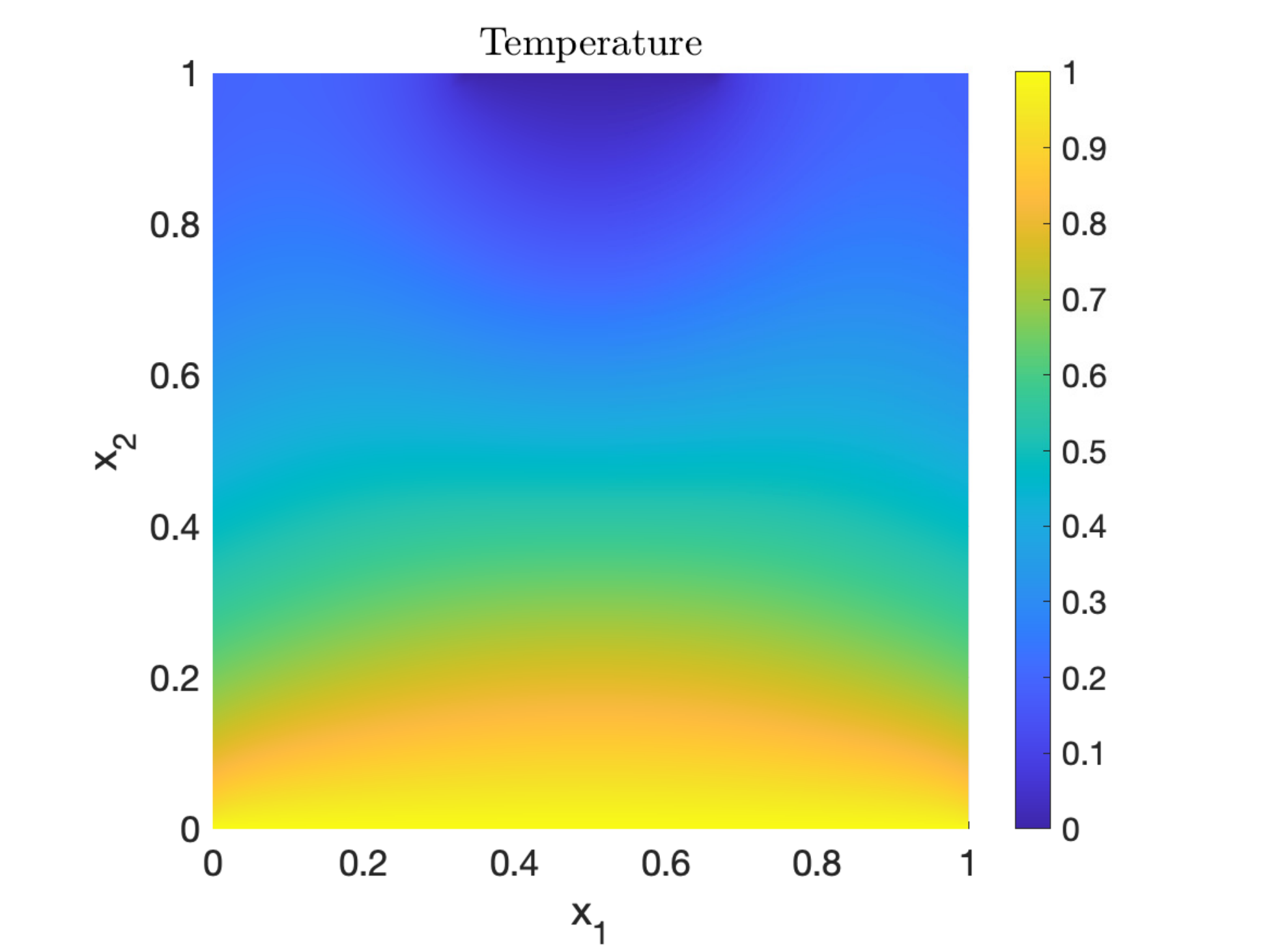}
  \caption{Optimal states solution for the thermal-fluid problem.}
  \label{fig:tf_states}
\end{figure}

\begin{figure}[h]
\centering
  \includegraphics[width=0.3\textwidth]{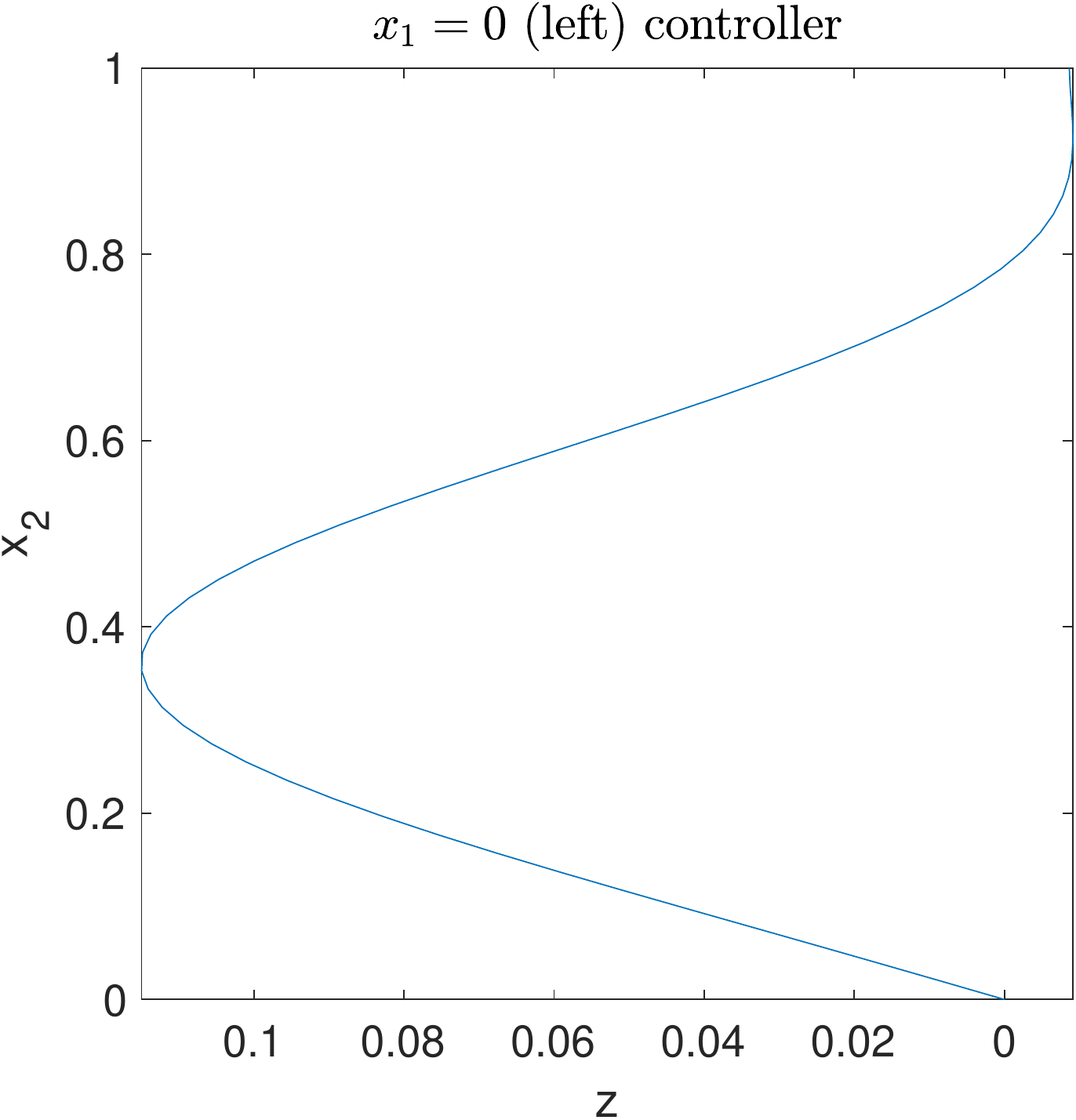}
    \includegraphics[width=0.3\textwidth]{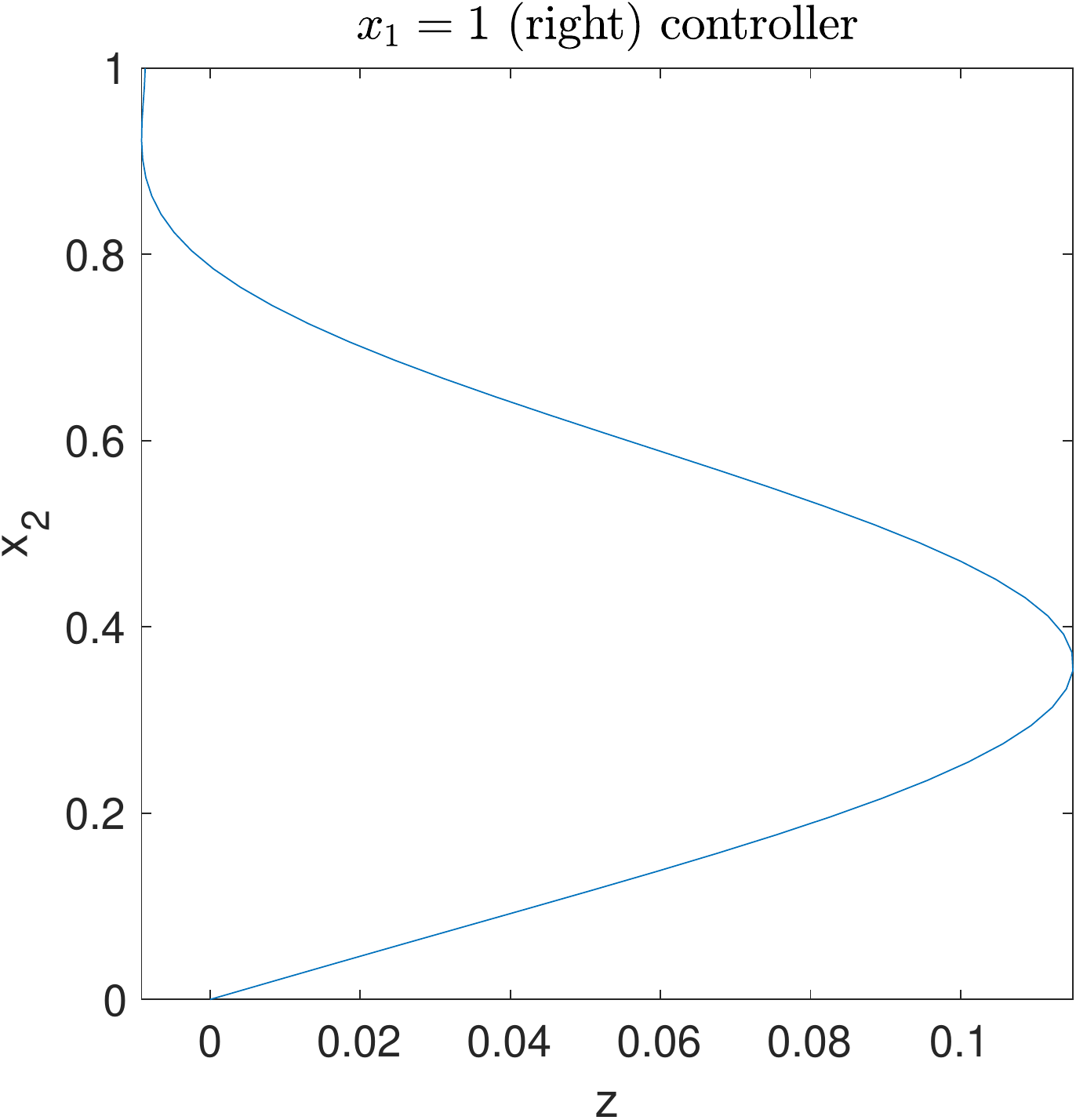}
  \caption{Optimal control solution for the thermal-fluid problem.}
  \label{fig:tf_control}
\end{figure}

\subsubsection*{HDSA with respect to model discrepancy}
It is common that model discrepancies arise from simplifying assumptions made to produce more management fluid flow simulations. Since the objective function only depends on the velocity field, errors in the pressure or thermal states will be implicitly represented in the velocity states. To facilitate coupling of the states we define  the state inner product weighting matrix $\L = \nabla_{\u} \tilde{\c}^T \vec{W}^T \M_u^{-1} \vec{W} \nabla_{\u} \tilde{\c}$, where $\nabla_{\u} \tilde{\c}$ is the Jacobian of the PDE residual evaluated at the optimal state, $\W$ is a weighting matrix to prescribe tolerances in deviation for each individual equation, and the inverse of the mass matrix, $\M_u^{-1}$, ensures mesh independence by weighting the inner product appropriately. We define $\vec{W}$ to weight the divergence equation by $10^2$ and the other equations by $1$. This soft penalty ensures that the model discrepancies for the velocity field will be divergence free up to a small tolerance. By penalizing $\delta$'s deviation from the original PDE system, we will identify how errors in the pressure and temperature propagate to errors in the velocity field. We define $\vec{\Gamma}$ as the inverse of an elliptic operator as in the convection-diffusion-reaction problem, with $\alpha=6$ and $\beta=10^{-6}$.

Figure~\ref{fig:tf_sing_vectors} displays $\sigma_N$, $\d(\overline{\z},\t_N)$ and $\sigma_N \w_N$ for the six leading singular vectors, $N=1,2,\dots,6$. We observe symmetries, anti-symmetries, and a progression in sinusoidal frequencies across the modes. In particular, the $\d$'s correspond to model discrepancy which creates additional vorticity and the different modes correspond to different orientations and frequencies of the vorticity. We observe symmetries in the left and right controller perturbations which either work together (moving heat in the same direction) or against one another (moving heat to or from the interior of the domain) as they correspond to the updated control strategy seeking to mitigate the vorticity generated by the $\d$'s. This gives valuable insights such as the realization that lower frequency discrepancies will have a greater influence on the optimal controller. In practice, information such as the frequency of the model discrepancy may be known from physical principles and/or experiments even if the form of the high-fidelity model is unknown. 

\begin{figure}
  \centering
  \begin{tabular}{M{.07\textwidth}M{.115\textwidth}M{.115\textwidth}M{.115\textwidth}M{.115\textwidth}M{.115\textwidth}M{.115\textwidth}}
$\sigma_N$ & $\delta_{v_1}(\overline{z},\theta_N)$ & $\delta_{v_2}(\overline{z},\theta_N)$ & $\delta_{p}(\overline{z},\theta_N)$ & $\delta_T(\overline{z},\theta_N)$ & Left $\sigma_N w_N$ & Right $\sigma_N w_N$ \\
 .0522  & 
  \includegraphics[width=0.115\textwidth]{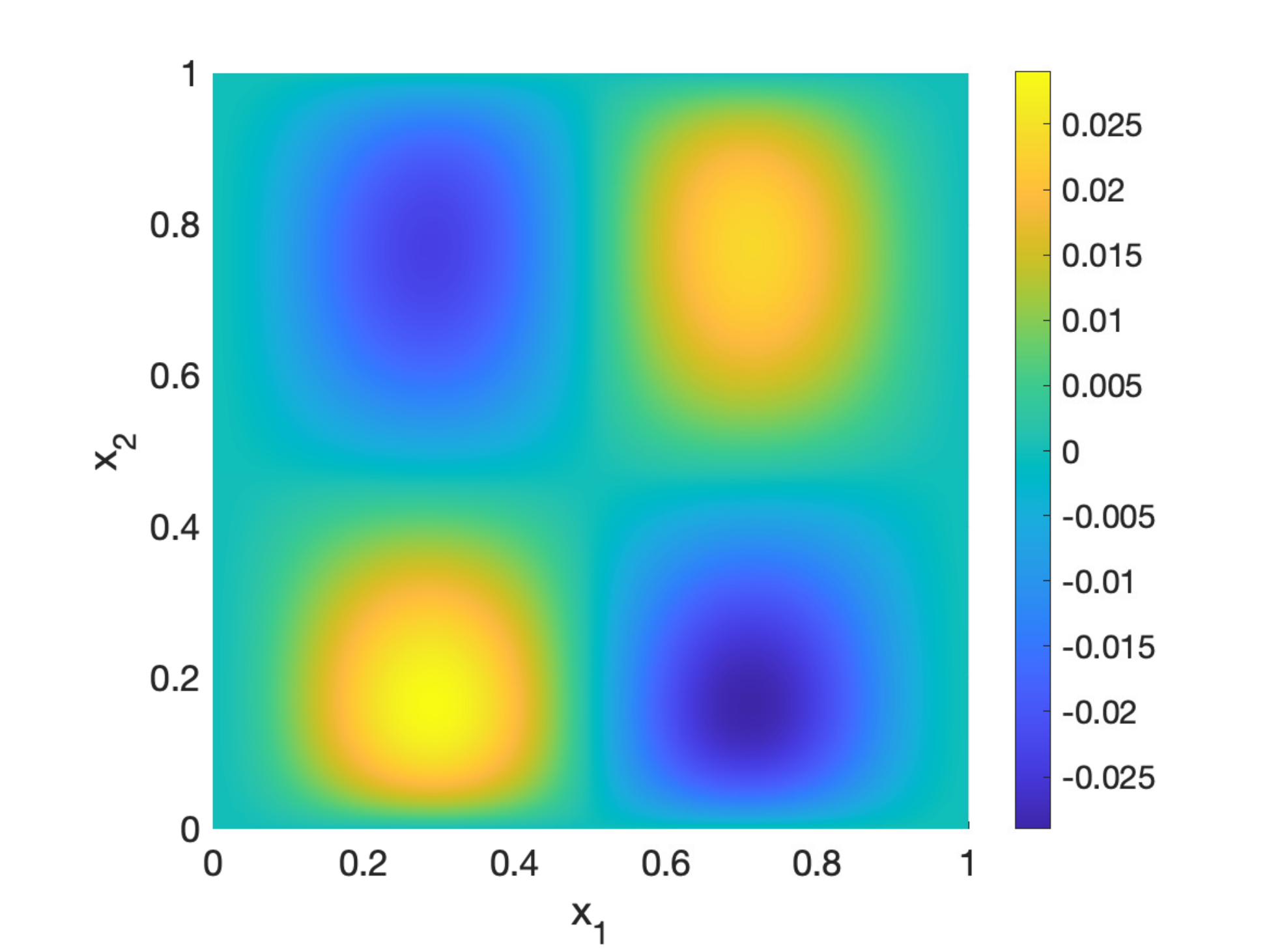} &
    \includegraphics[width=0.115\textwidth]{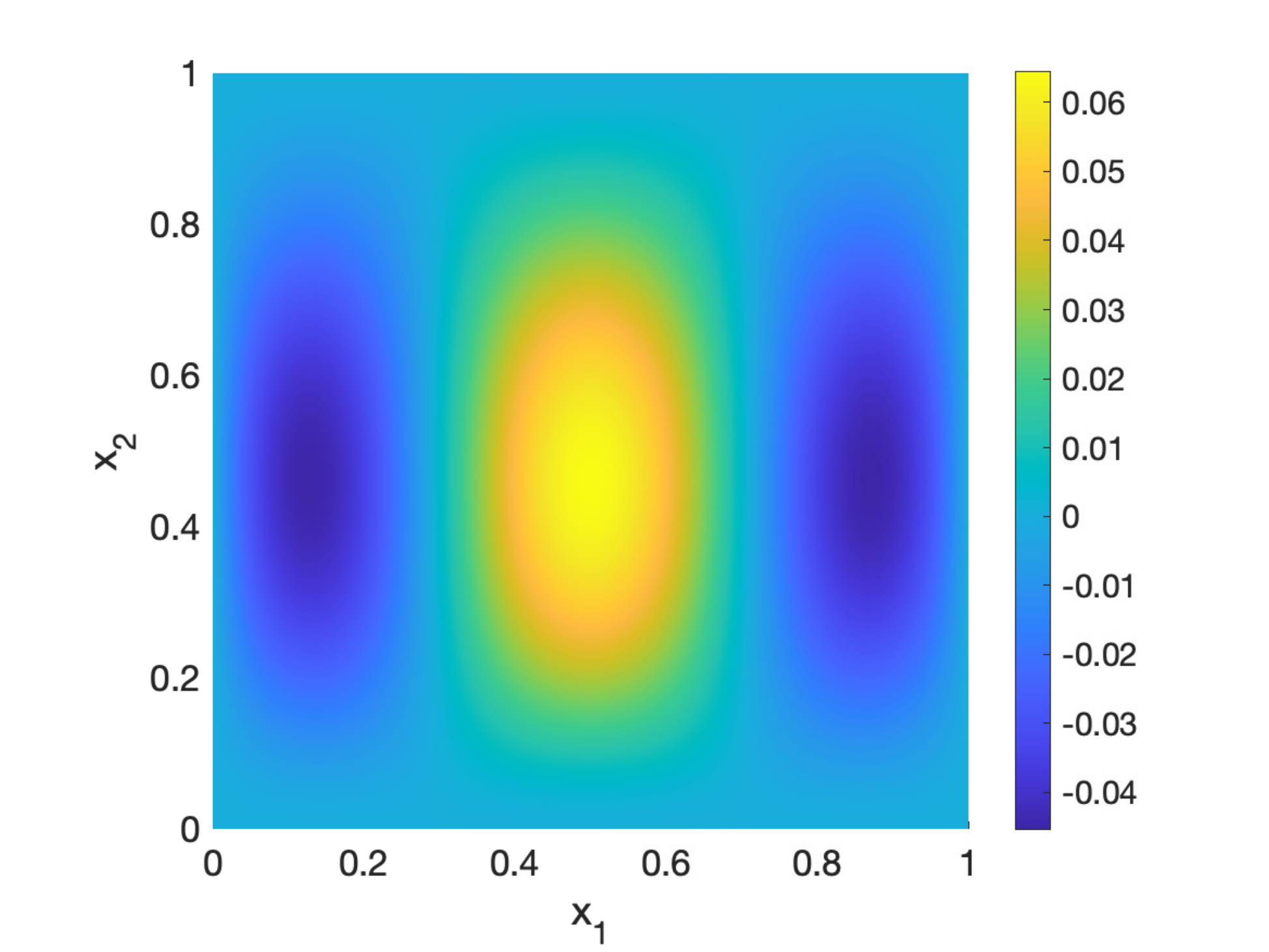} &
        \includegraphics[width=0.115\textwidth]{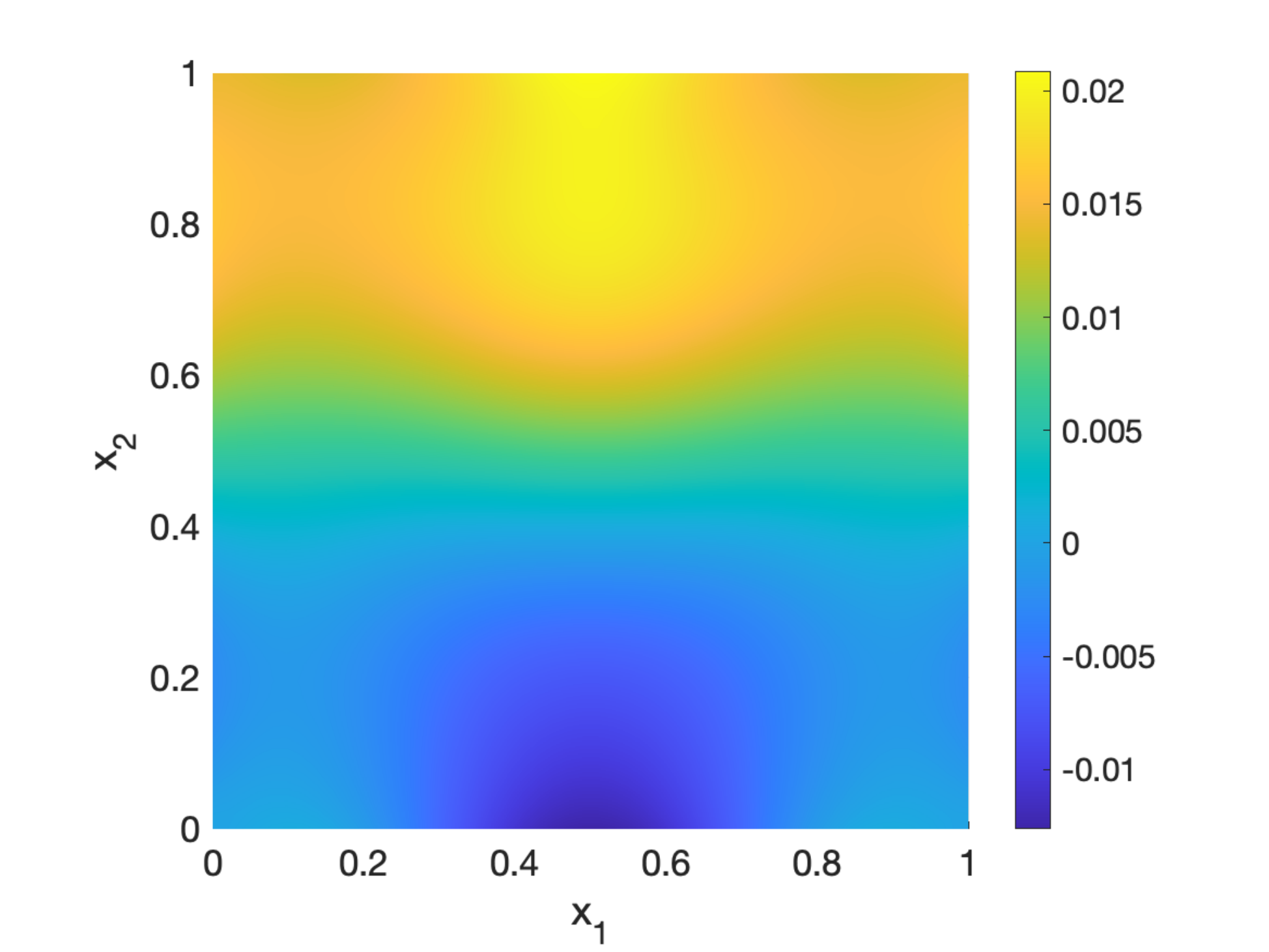} &
         \includegraphics[width=0.115\textwidth]{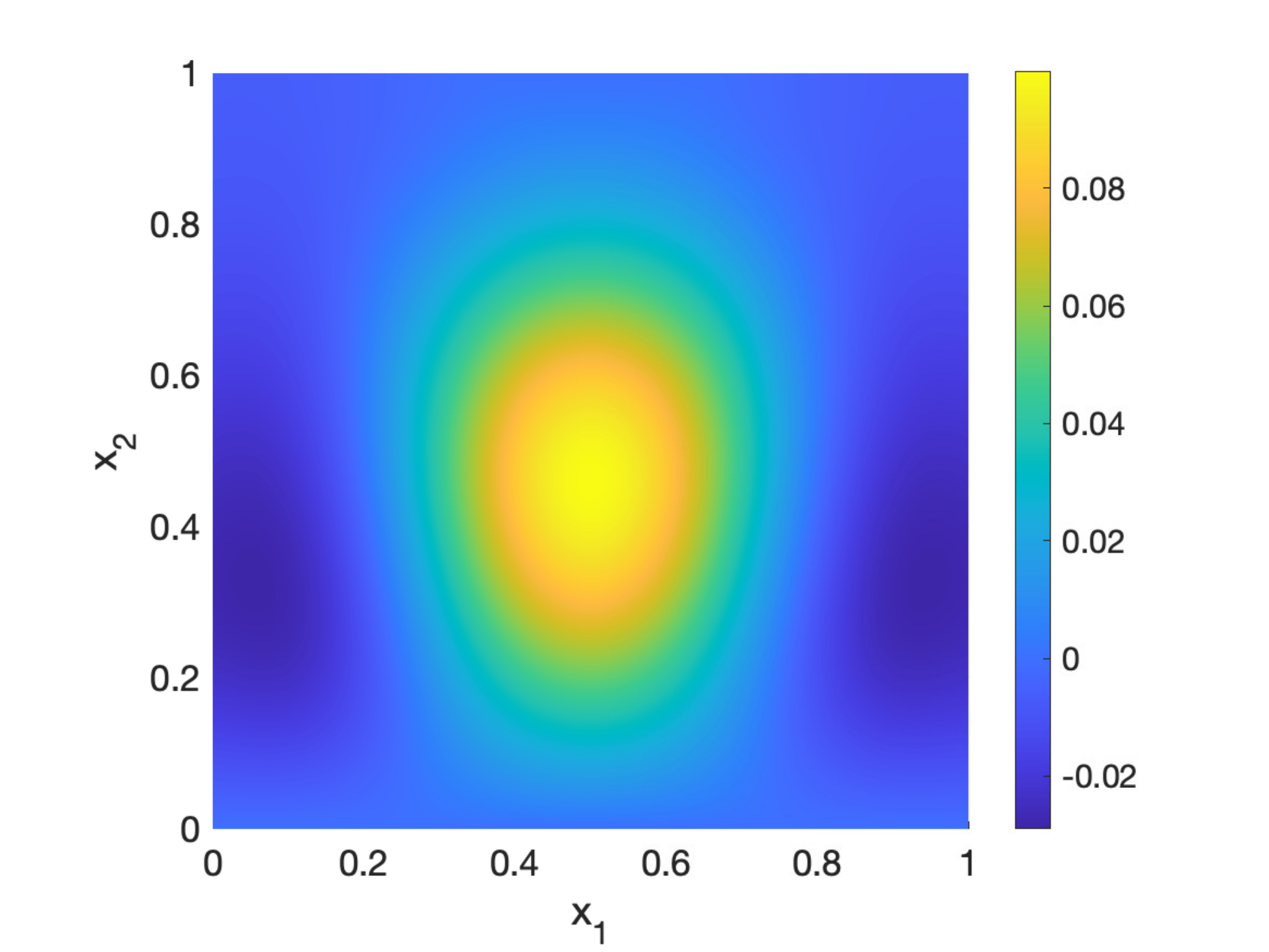} &
          \includegraphics[width=0.115\textwidth]{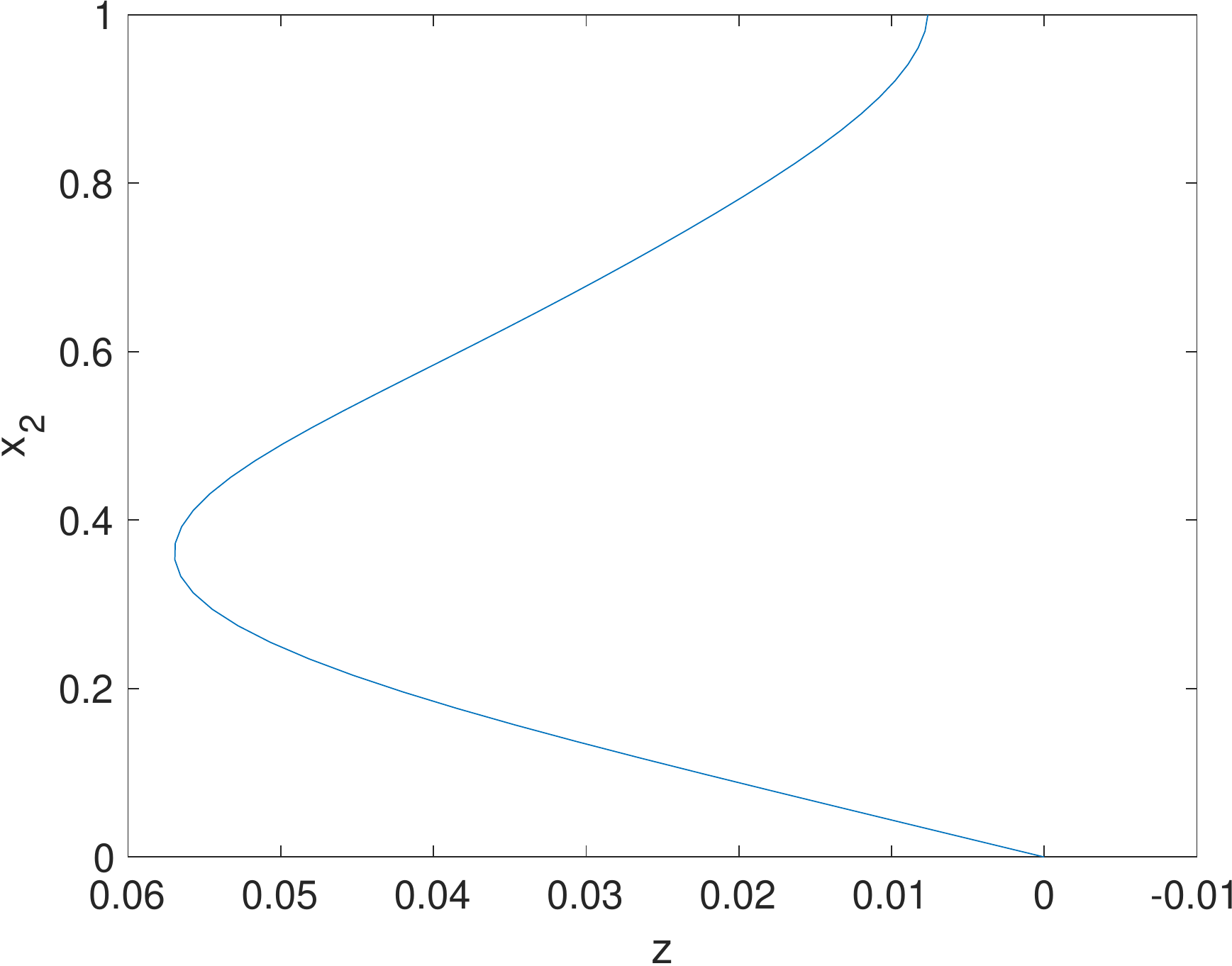} &
            \includegraphics[width=0.115\textwidth]{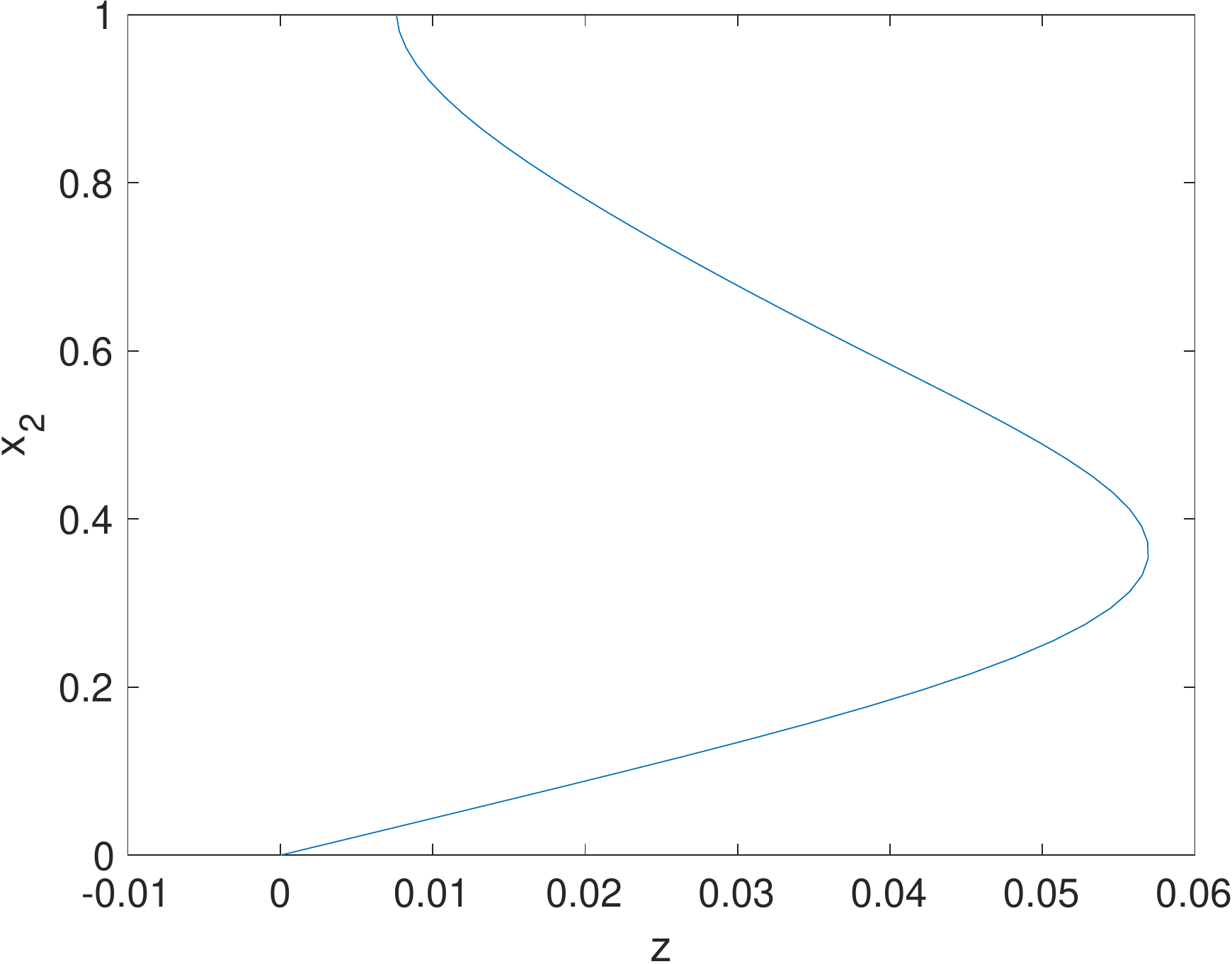}  \\
         .0102 &
  \includegraphics[width=0.115\textwidth]{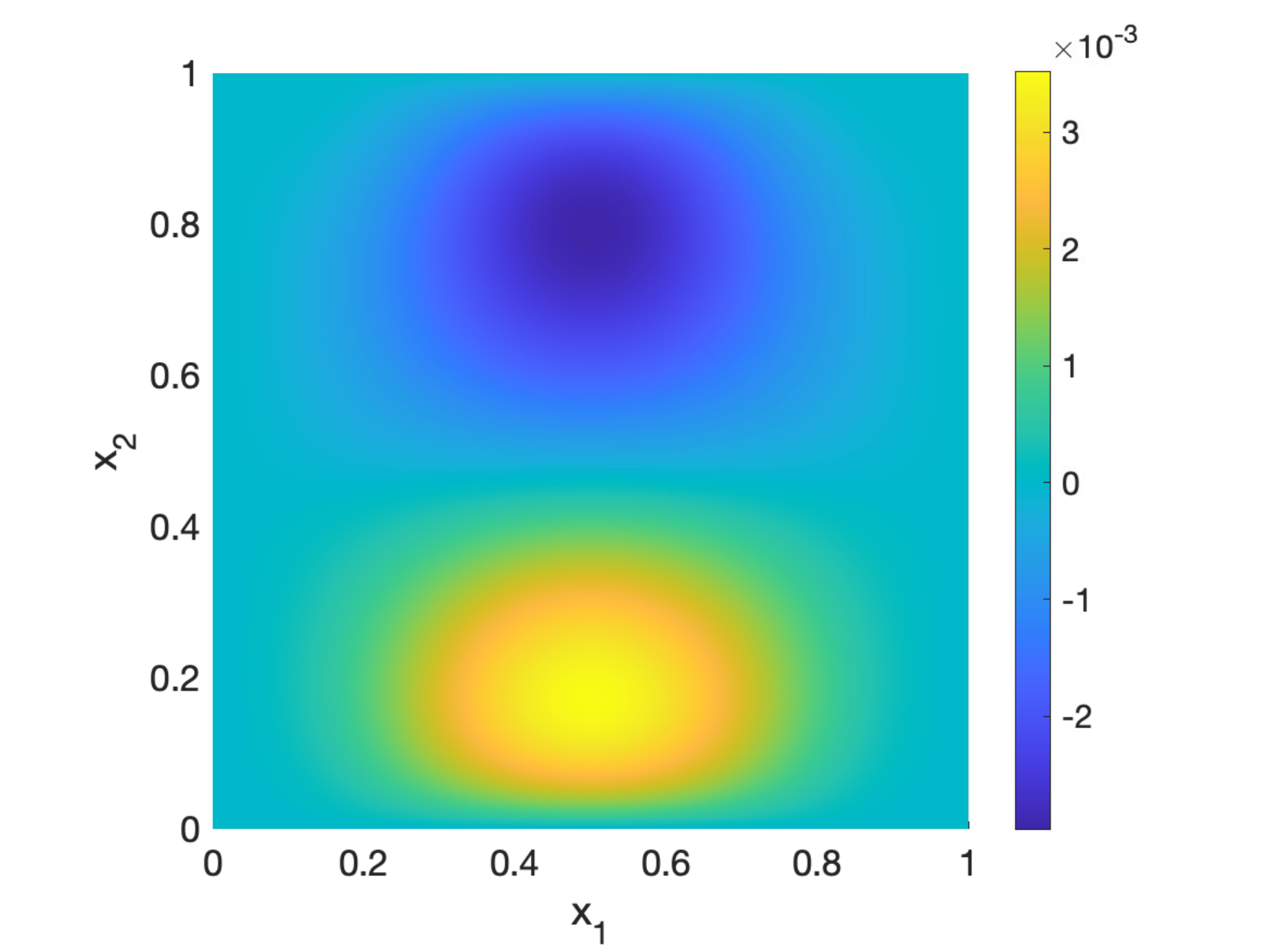} &
    \includegraphics[width=0.115\textwidth]{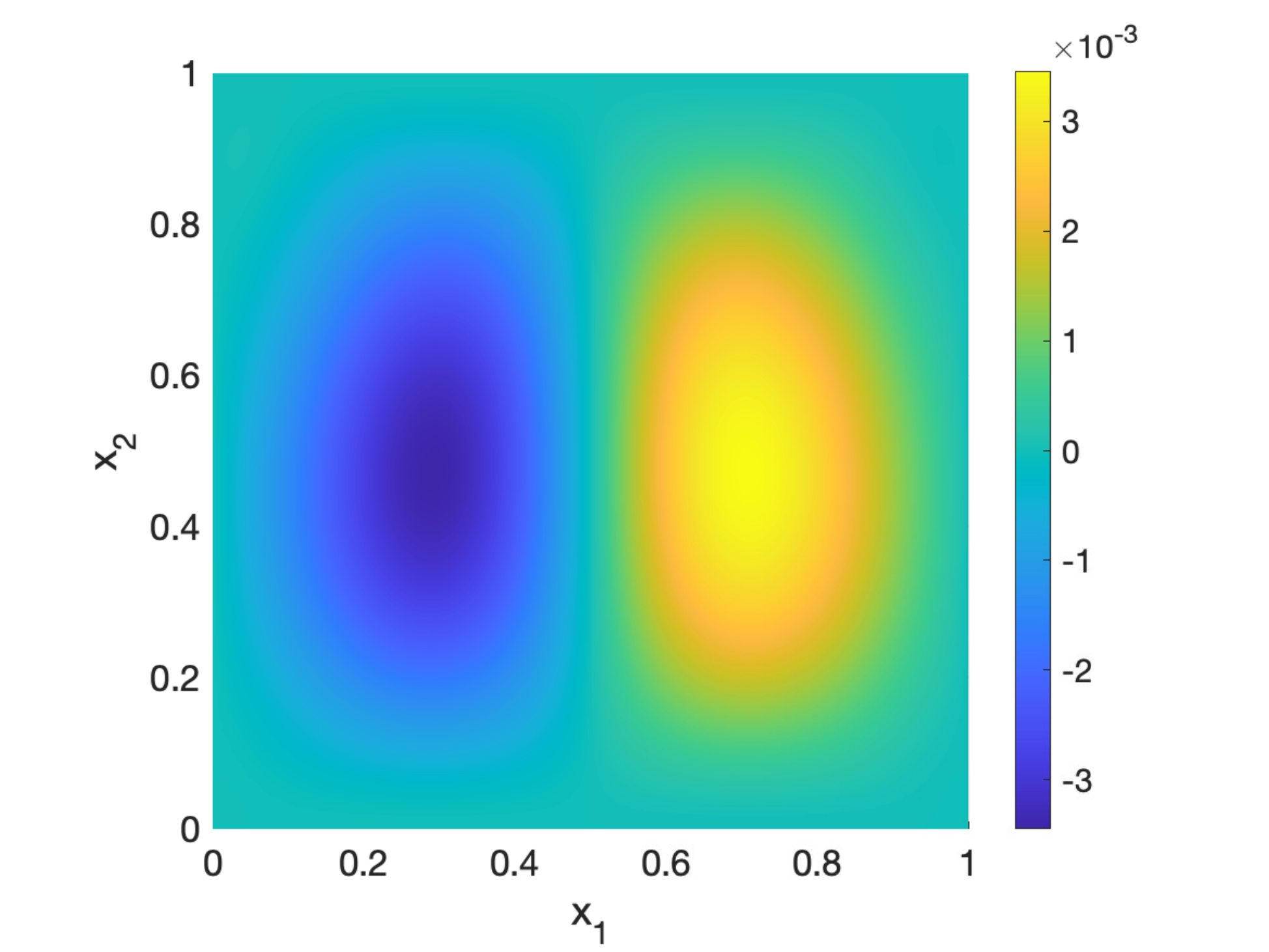} &
        \includegraphics[width=0.115\textwidth]{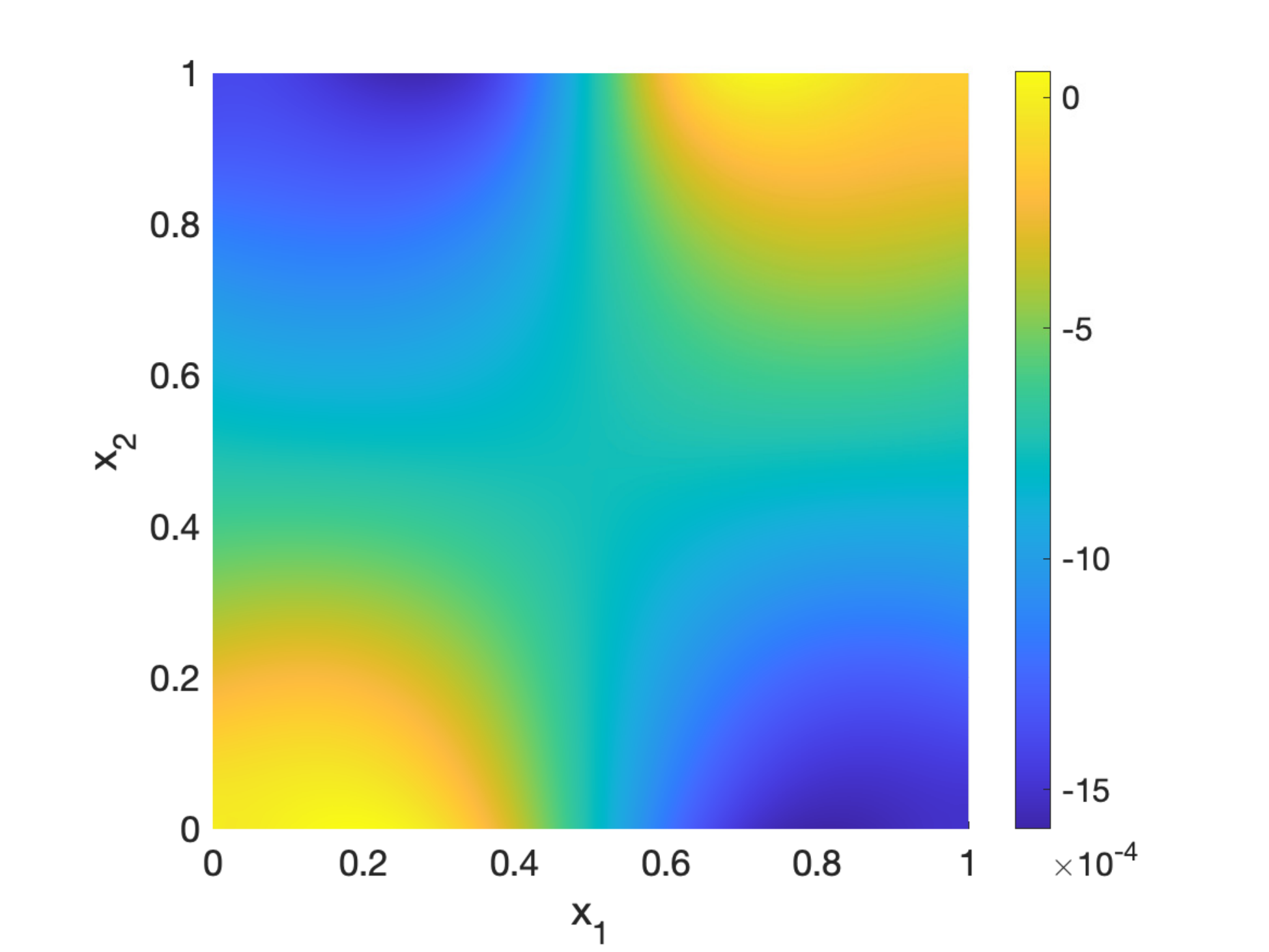} &
         \includegraphics[width=0.115\textwidth]{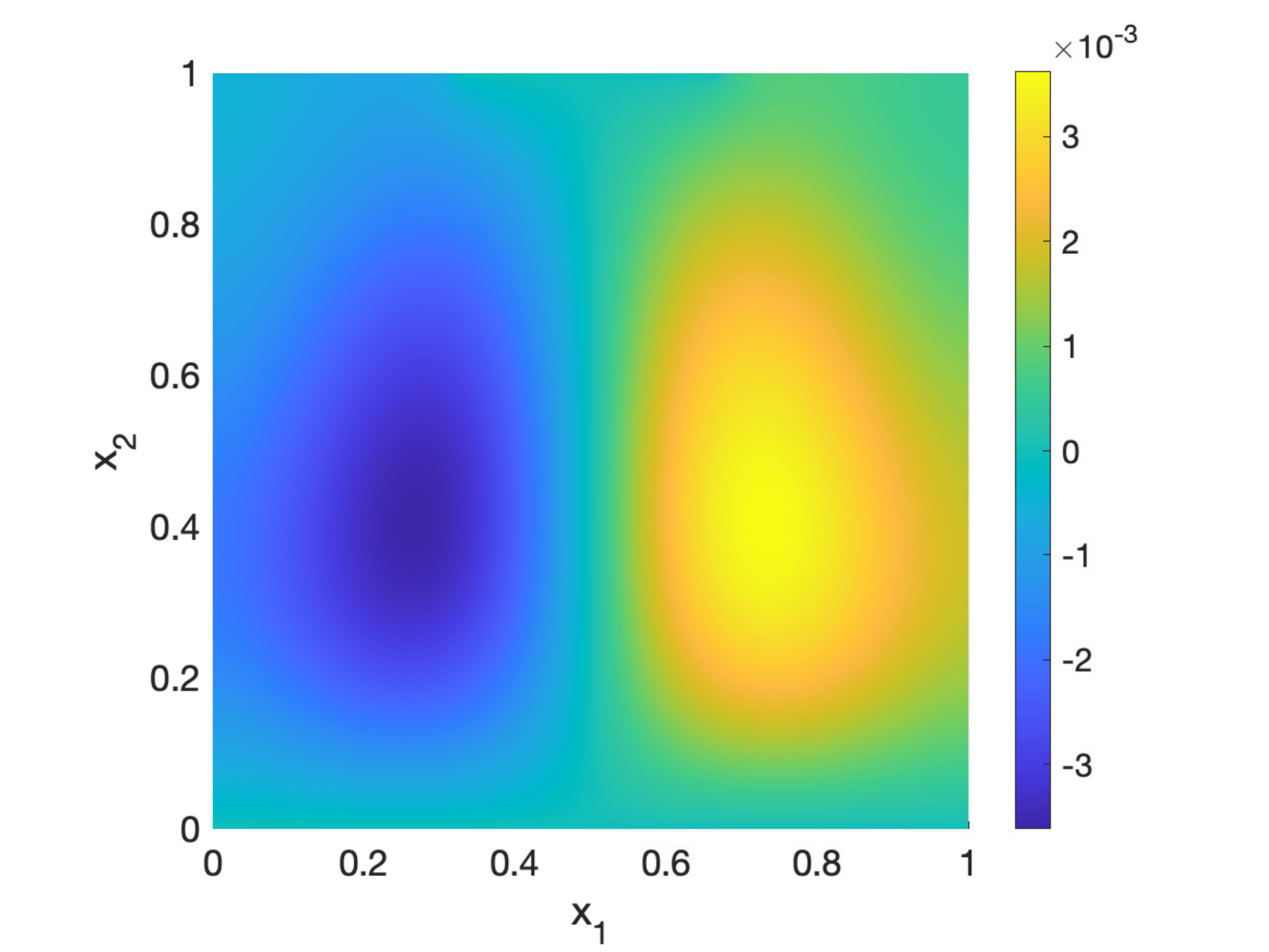} &
          \includegraphics[width=0.115\textwidth]{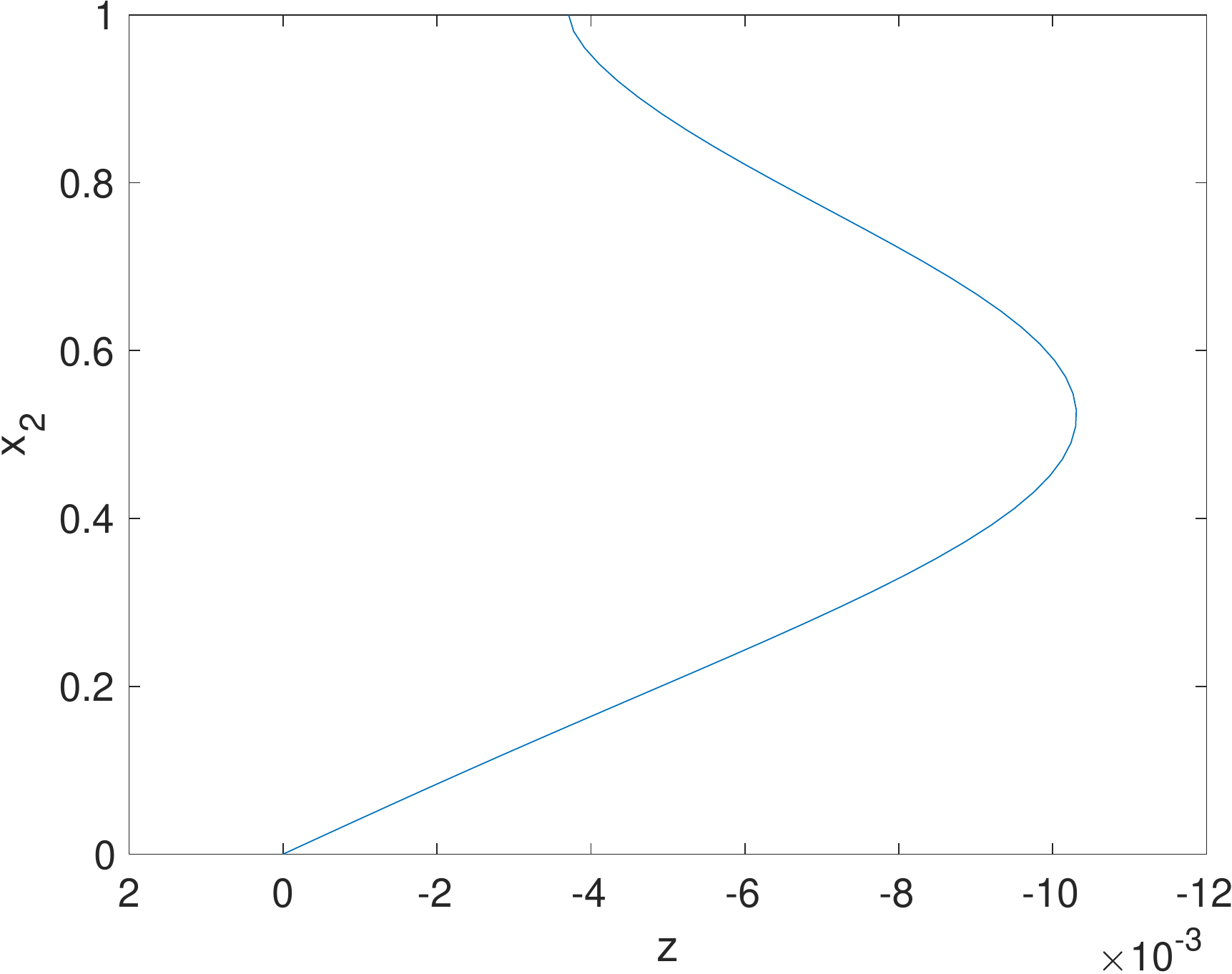} &
            \includegraphics[width=0.115\textwidth]{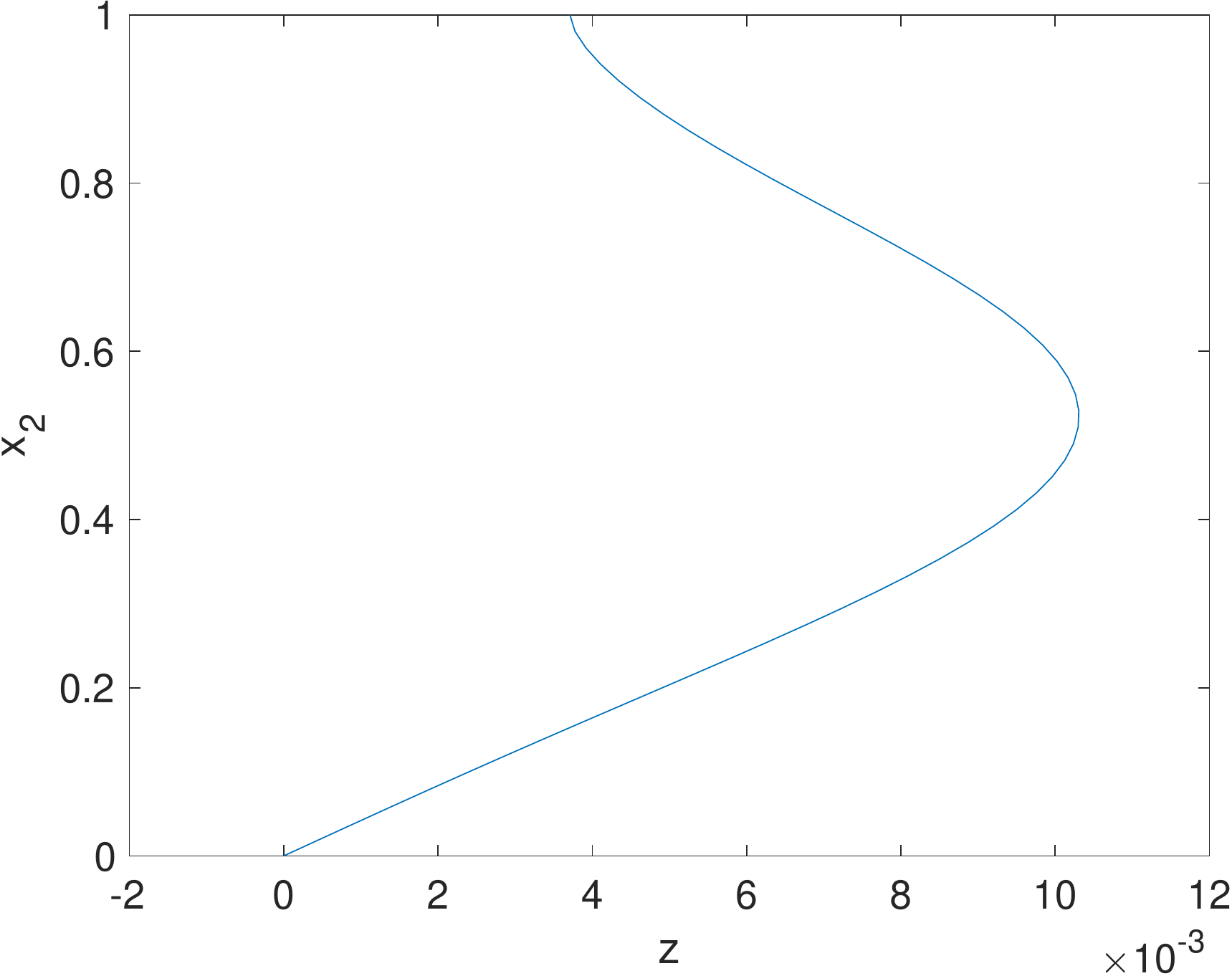}  \\
                 .0040  &
  \includegraphics[width=0.115\textwidth]{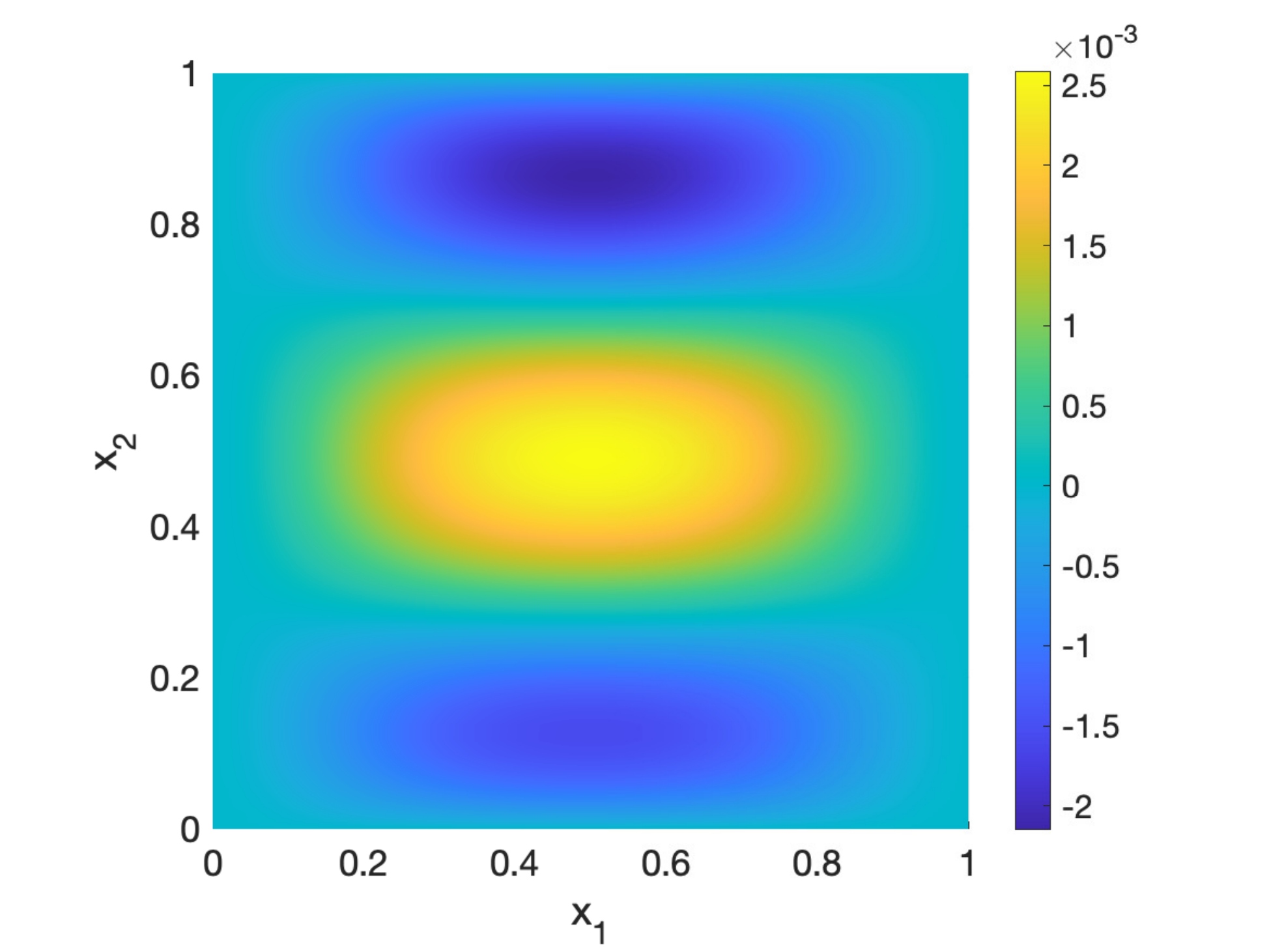} &
    \includegraphics[width=0.115\textwidth]{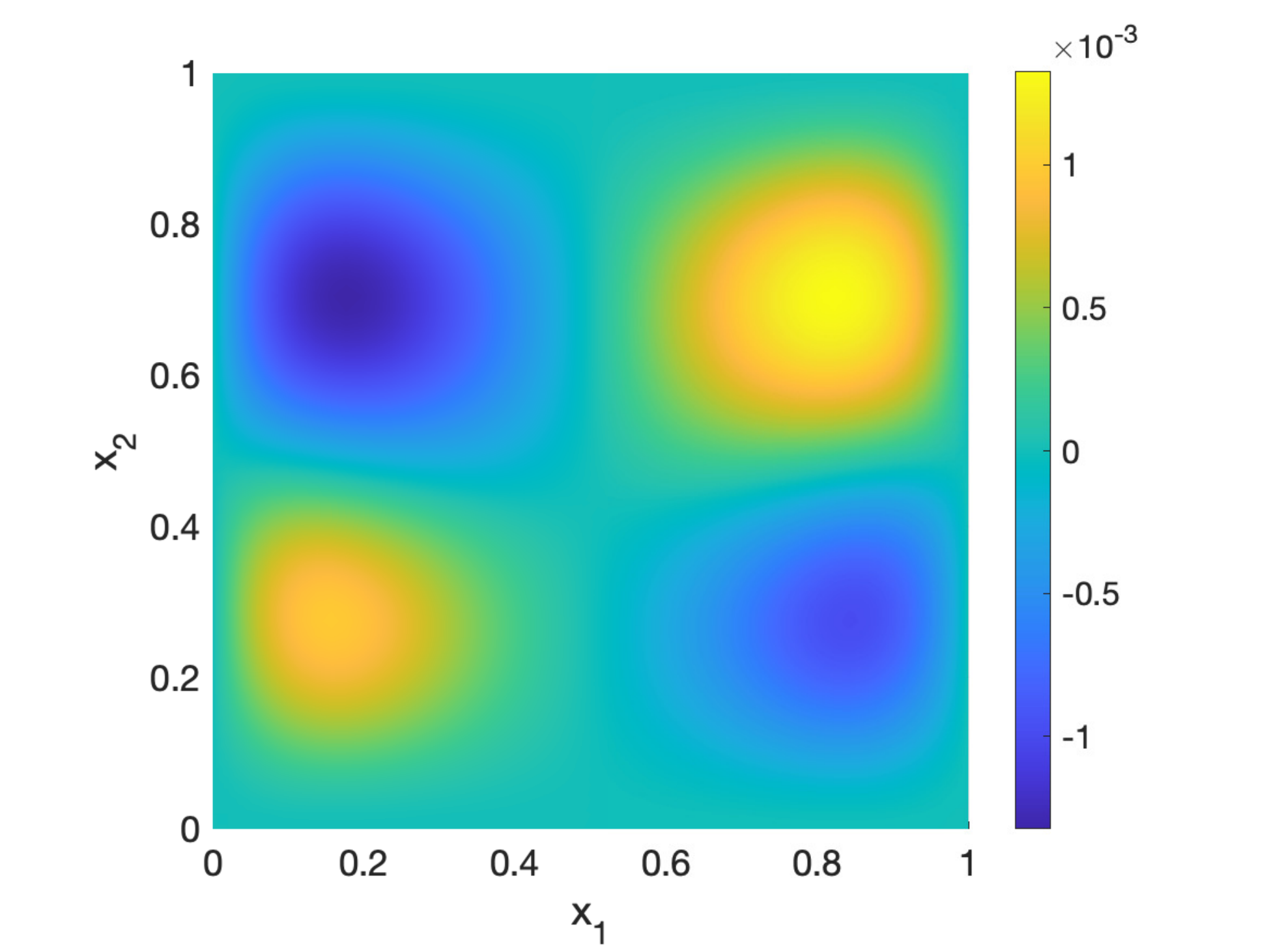} &
        \includegraphics[width=0.115\textwidth]{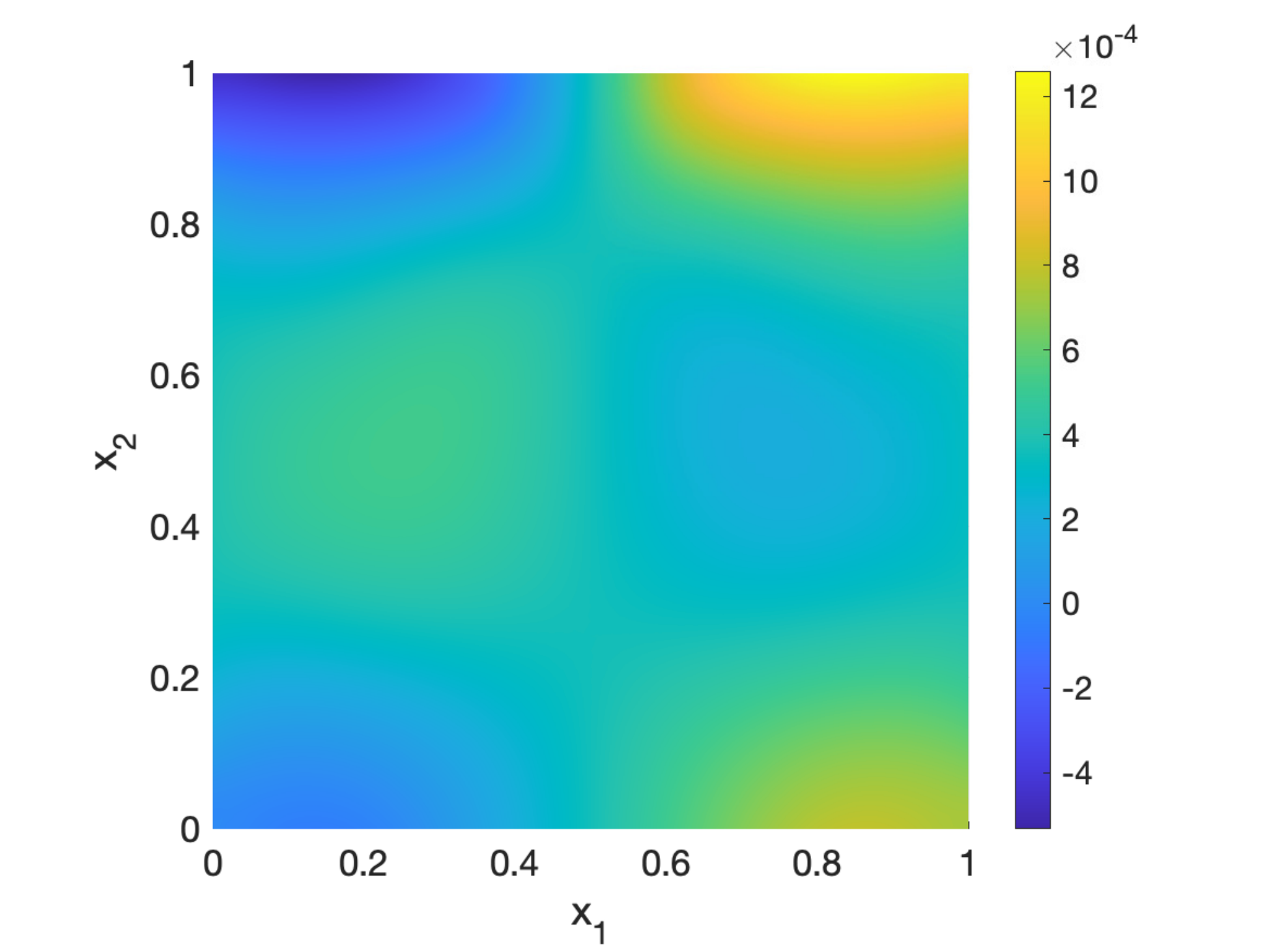} &
         \includegraphics[width=0.115\textwidth]{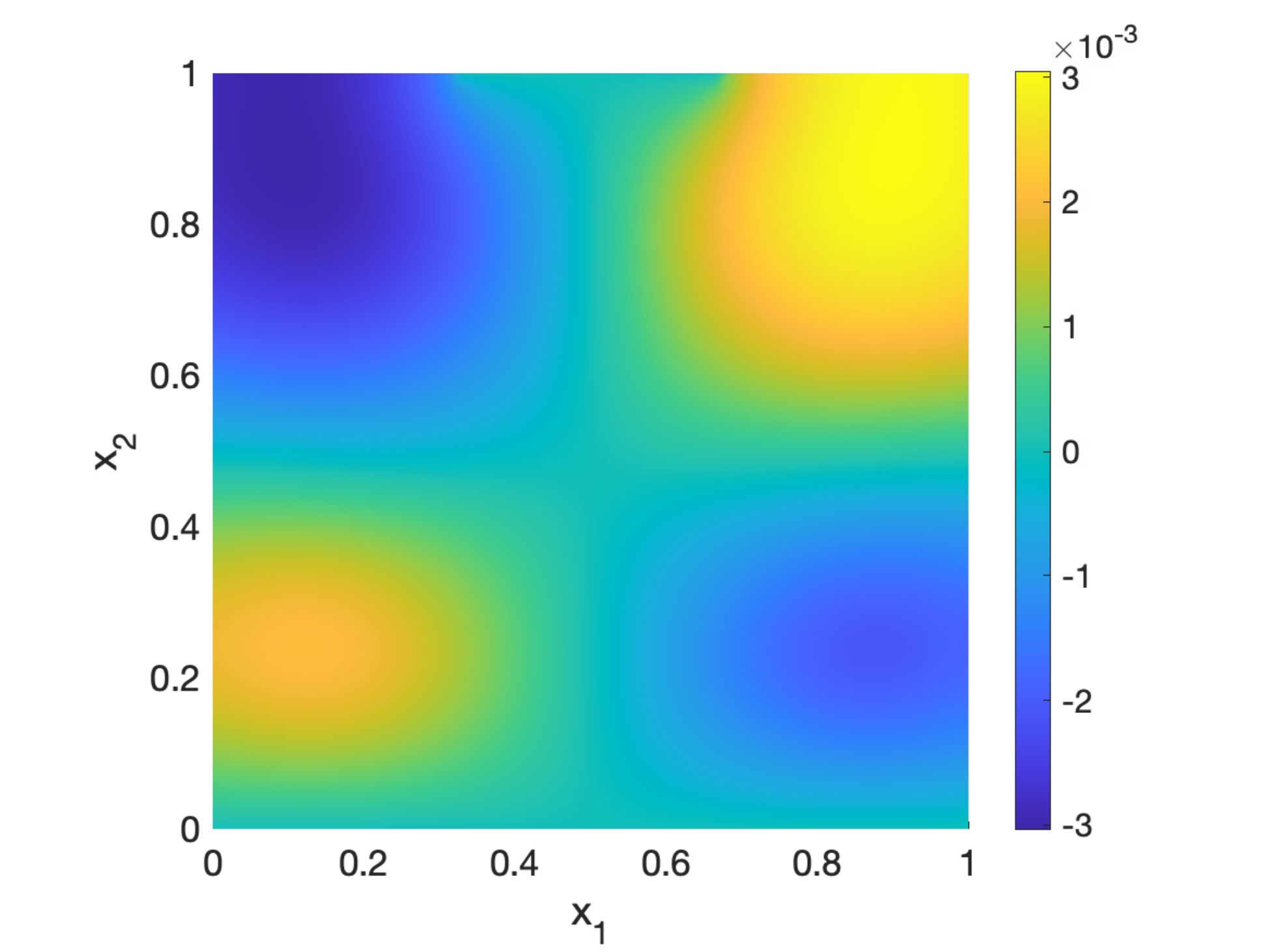} &
          \includegraphics[width=0.115\textwidth]{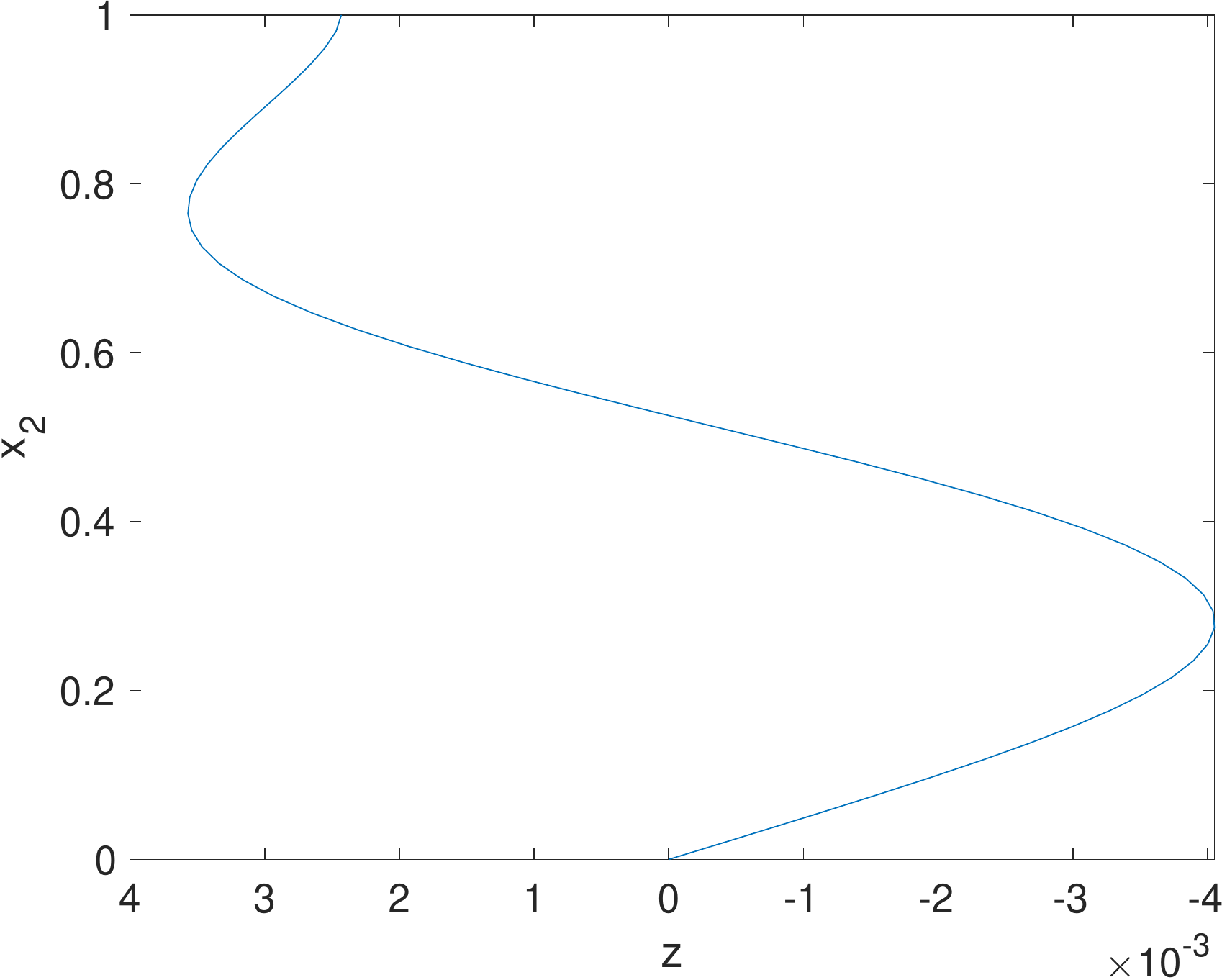} &
            \includegraphics[width=0.115\textwidth]{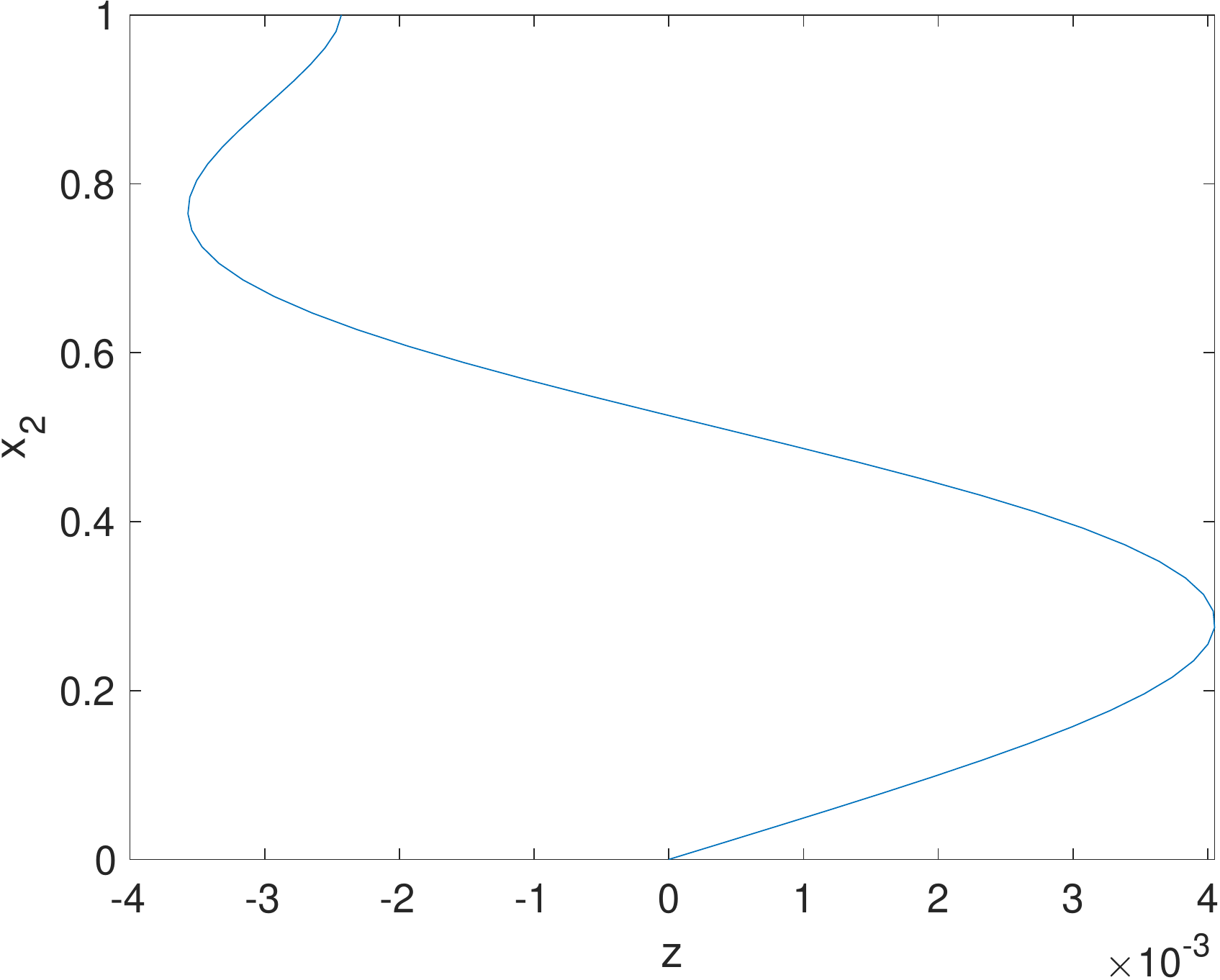}  \\
                         .0032 &
  \includegraphics[width=0.115\textwidth]{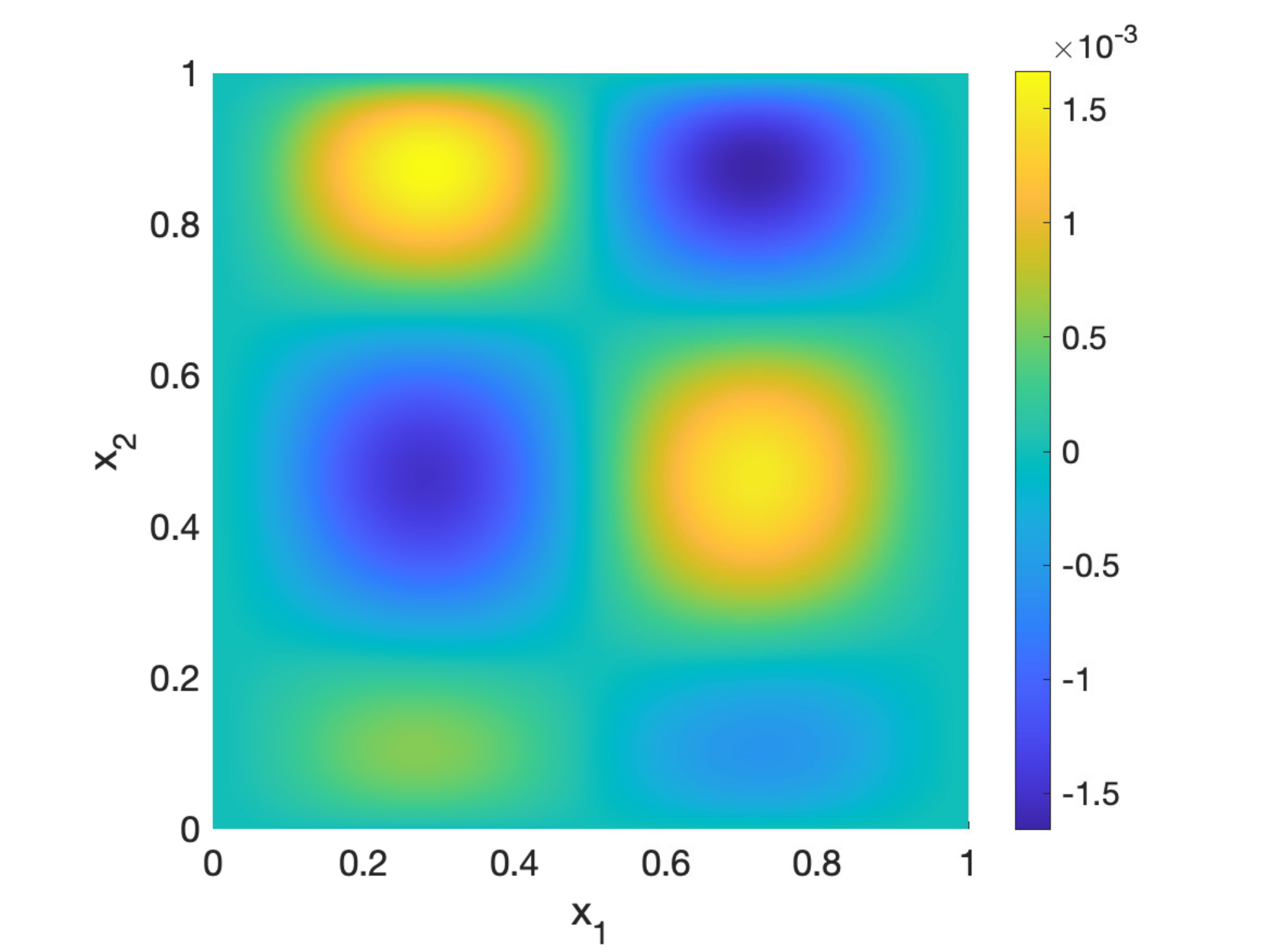} &
    \includegraphics[width=0.115\textwidth]{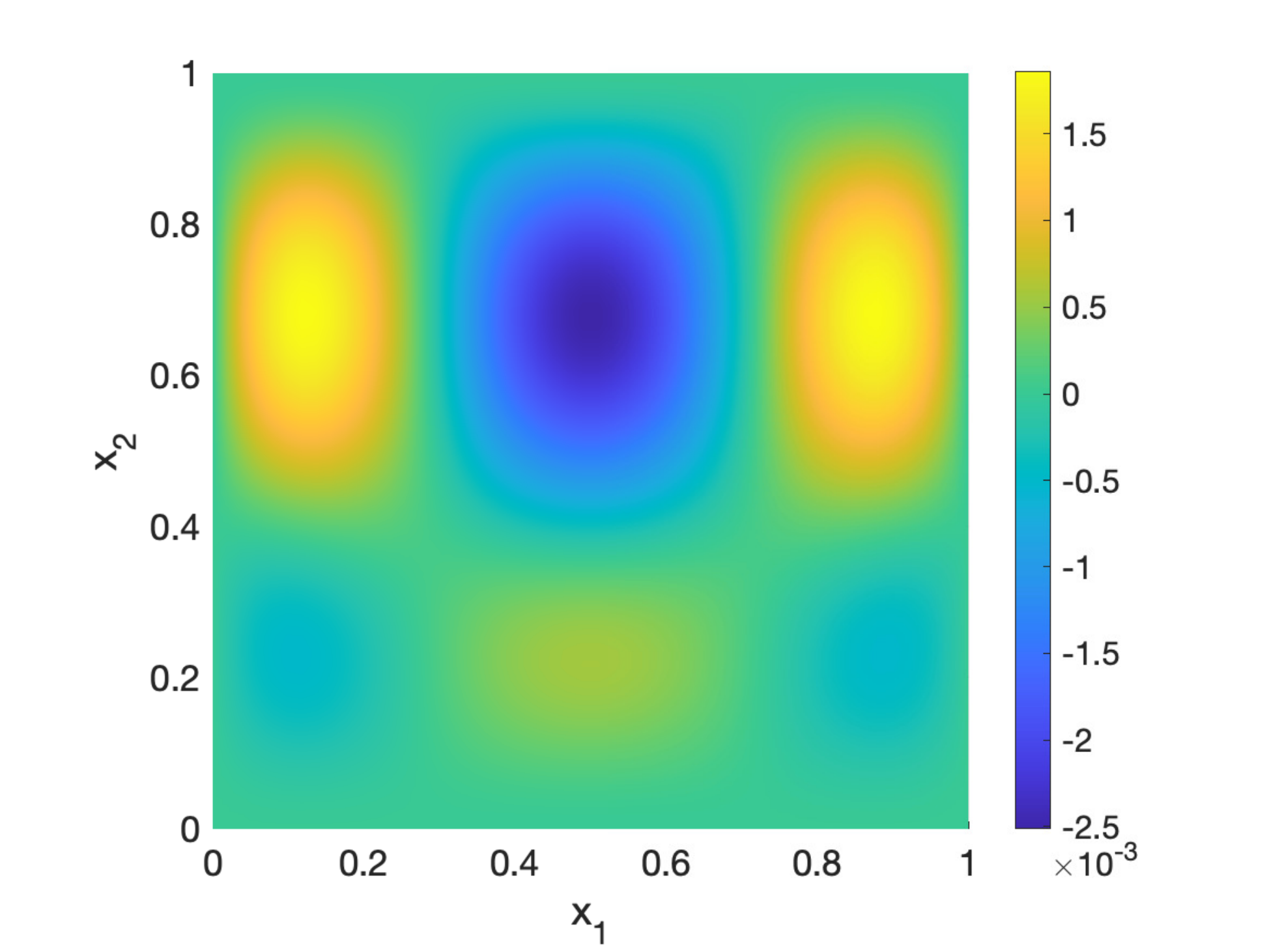} &
        \includegraphics[width=0.115\textwidth]{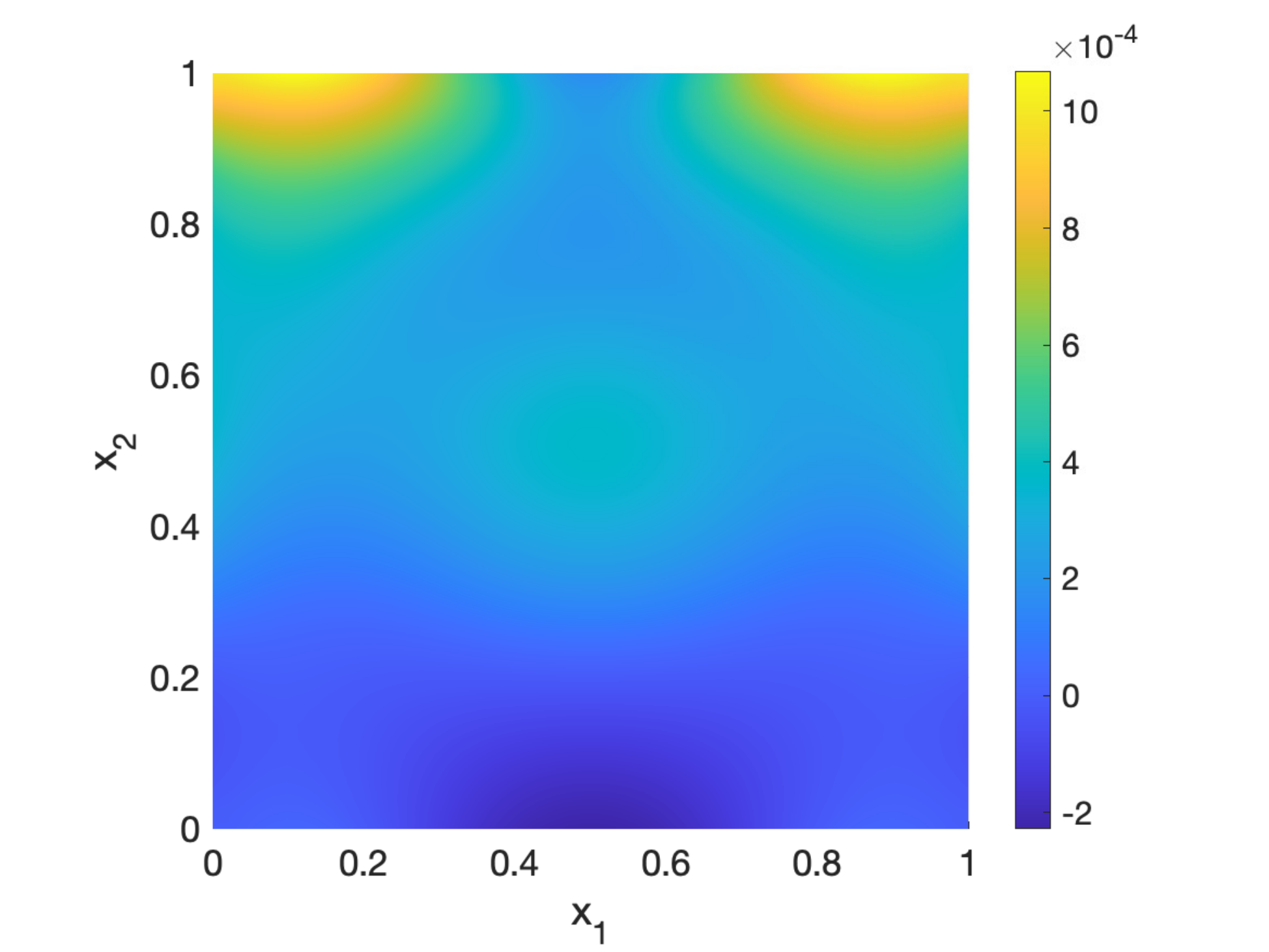} &
         \includegraphics[width=0.115\textwidth]{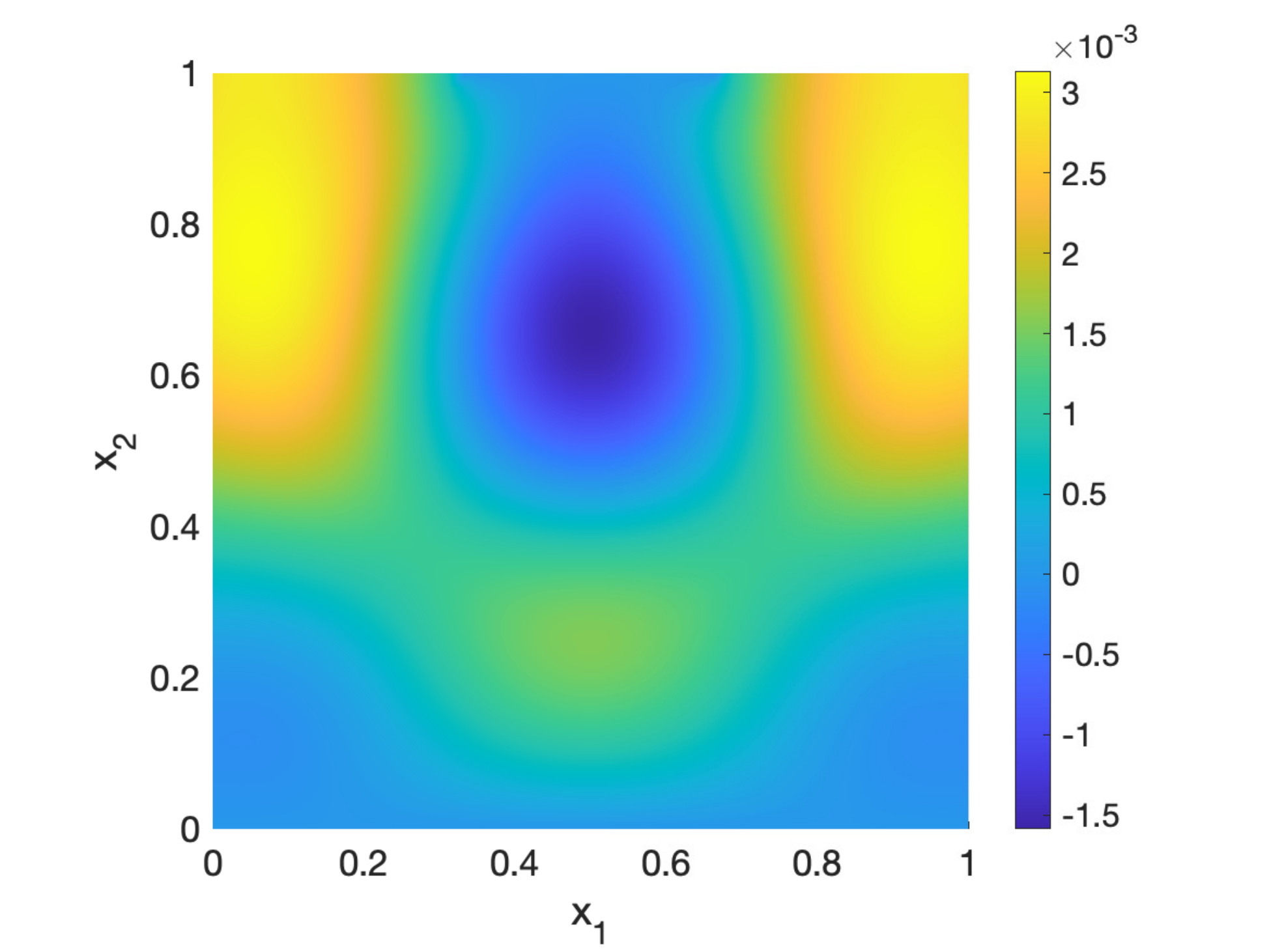} &
          \includegraphics[width=0.115\textwidth]{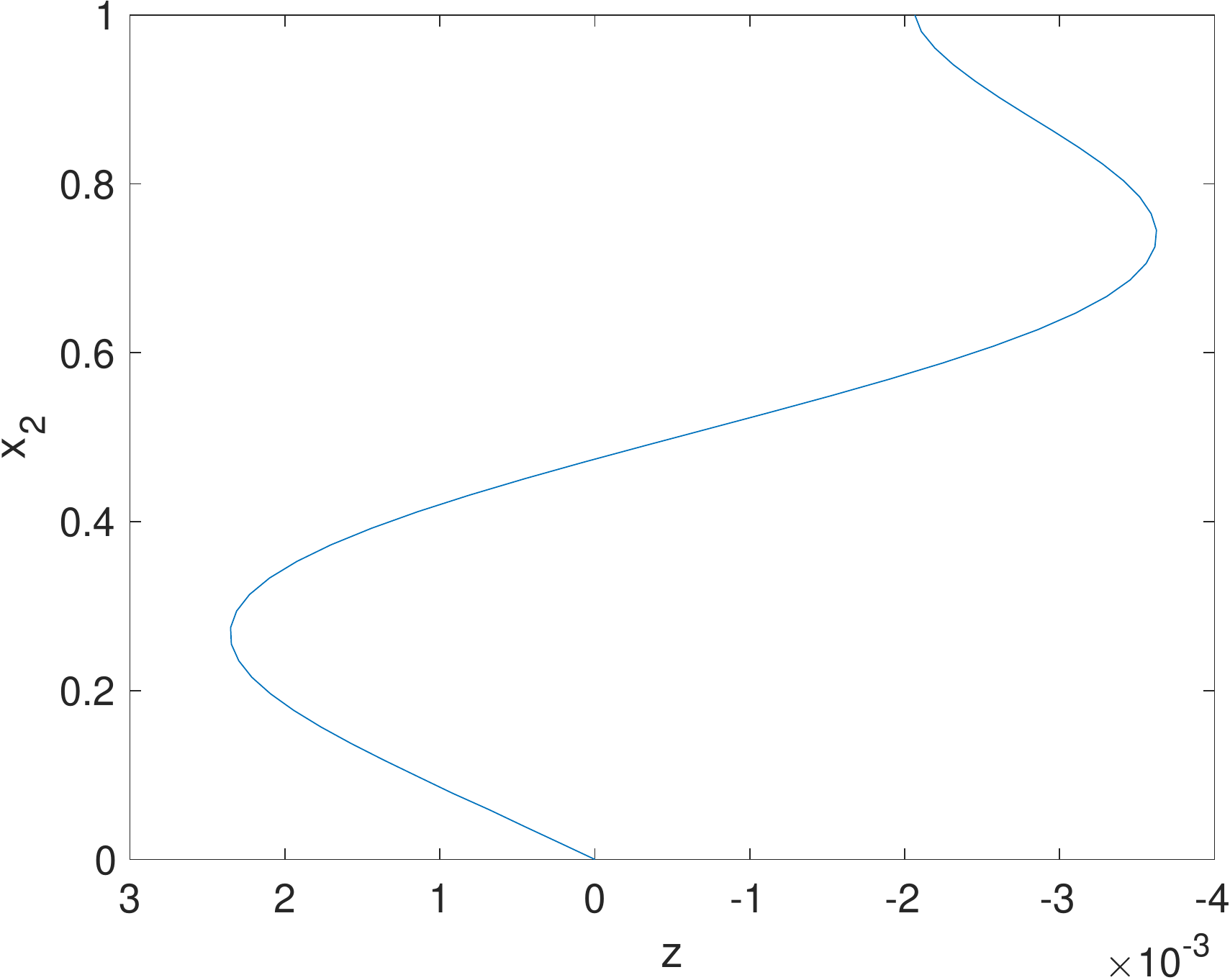} &
            \includegraphics[width=0.115\textwidth]{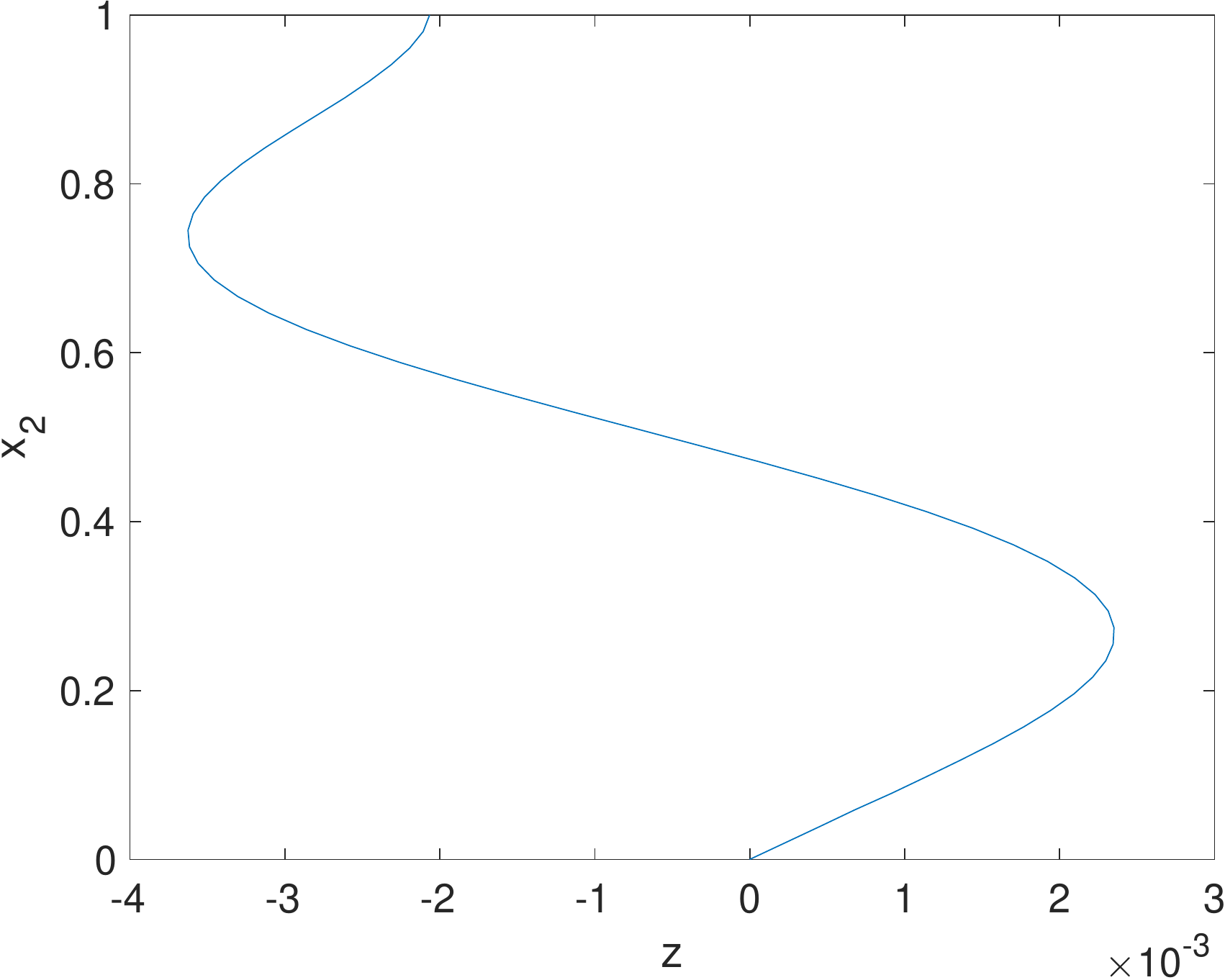}  \\
                                 .0011 &
  \includegraphics[width=0.115\textwidth]{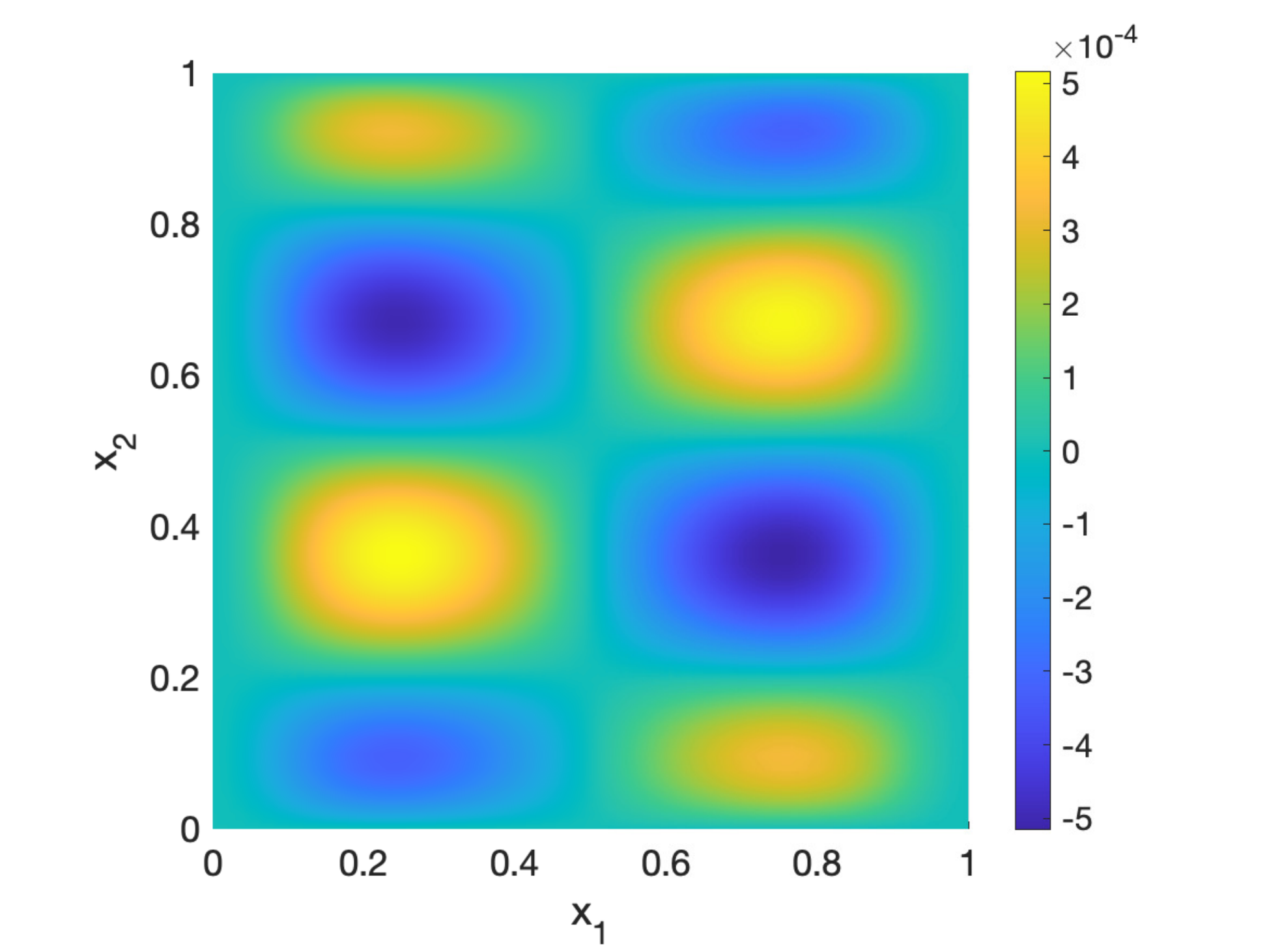} &
    \includegraphics[width=0.115\textwidth]{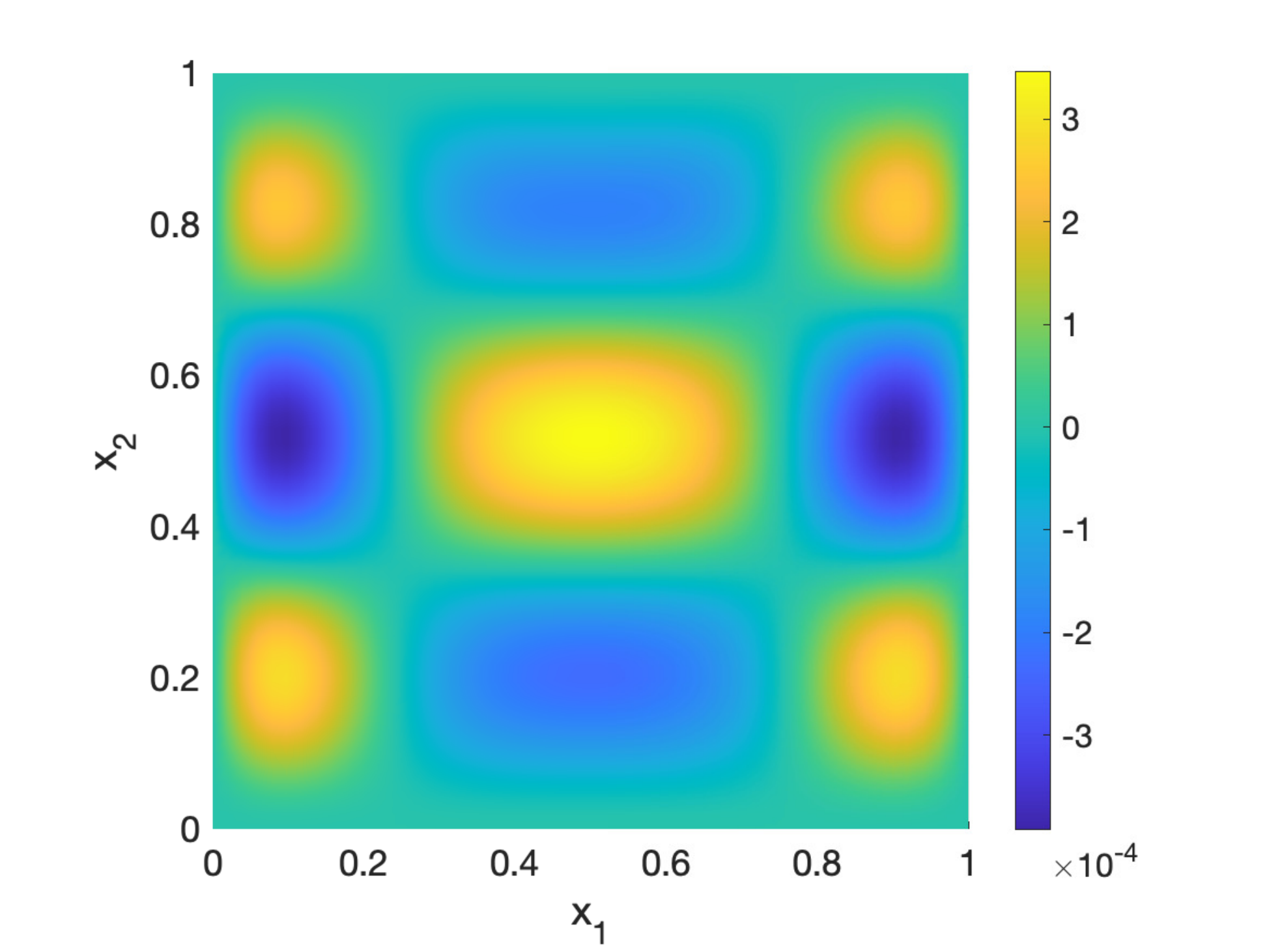} &
        \includegraphics[width=0.115\textwidth]{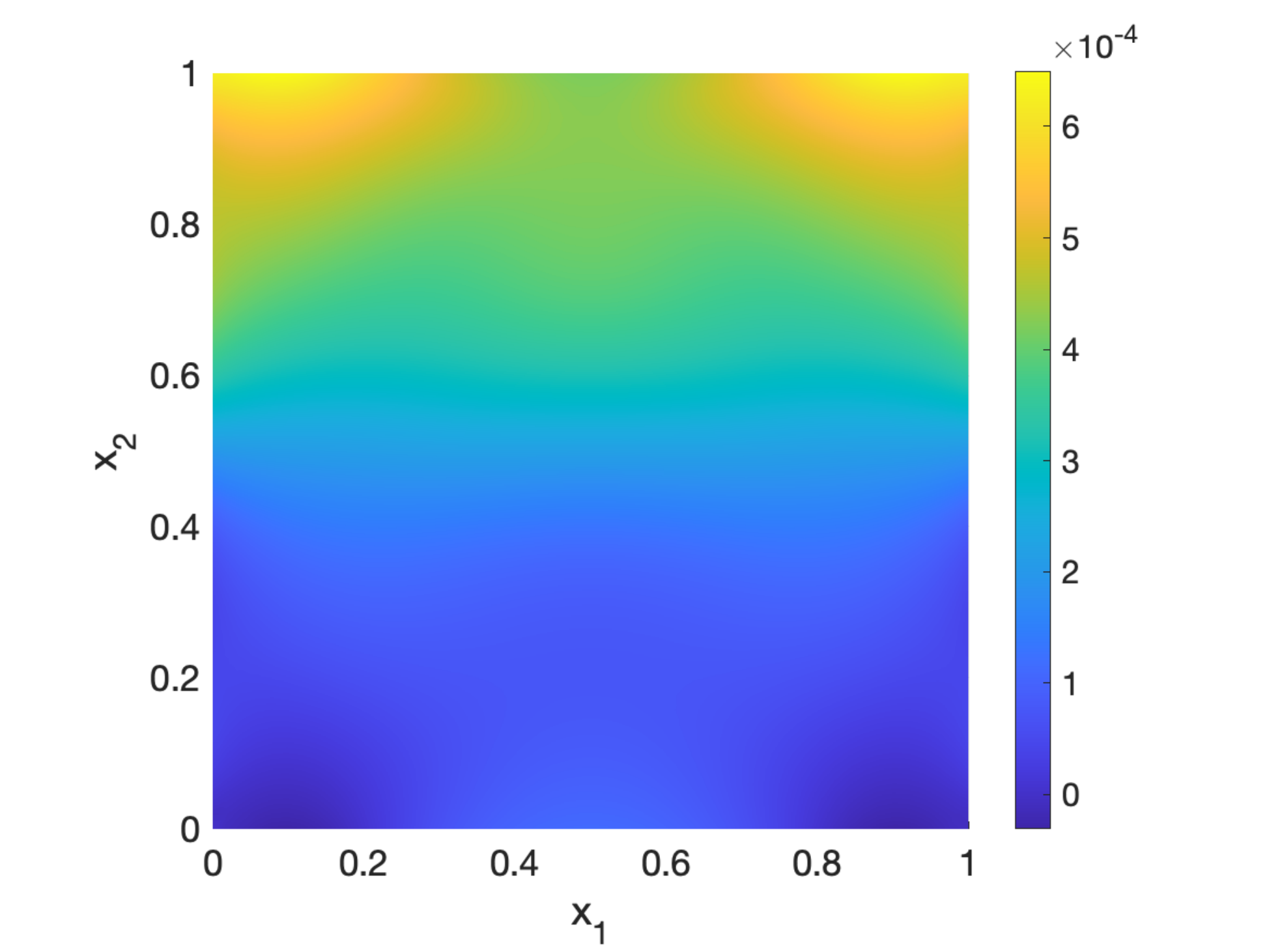} &
         \includegraphics[width=0.115\textwidth]{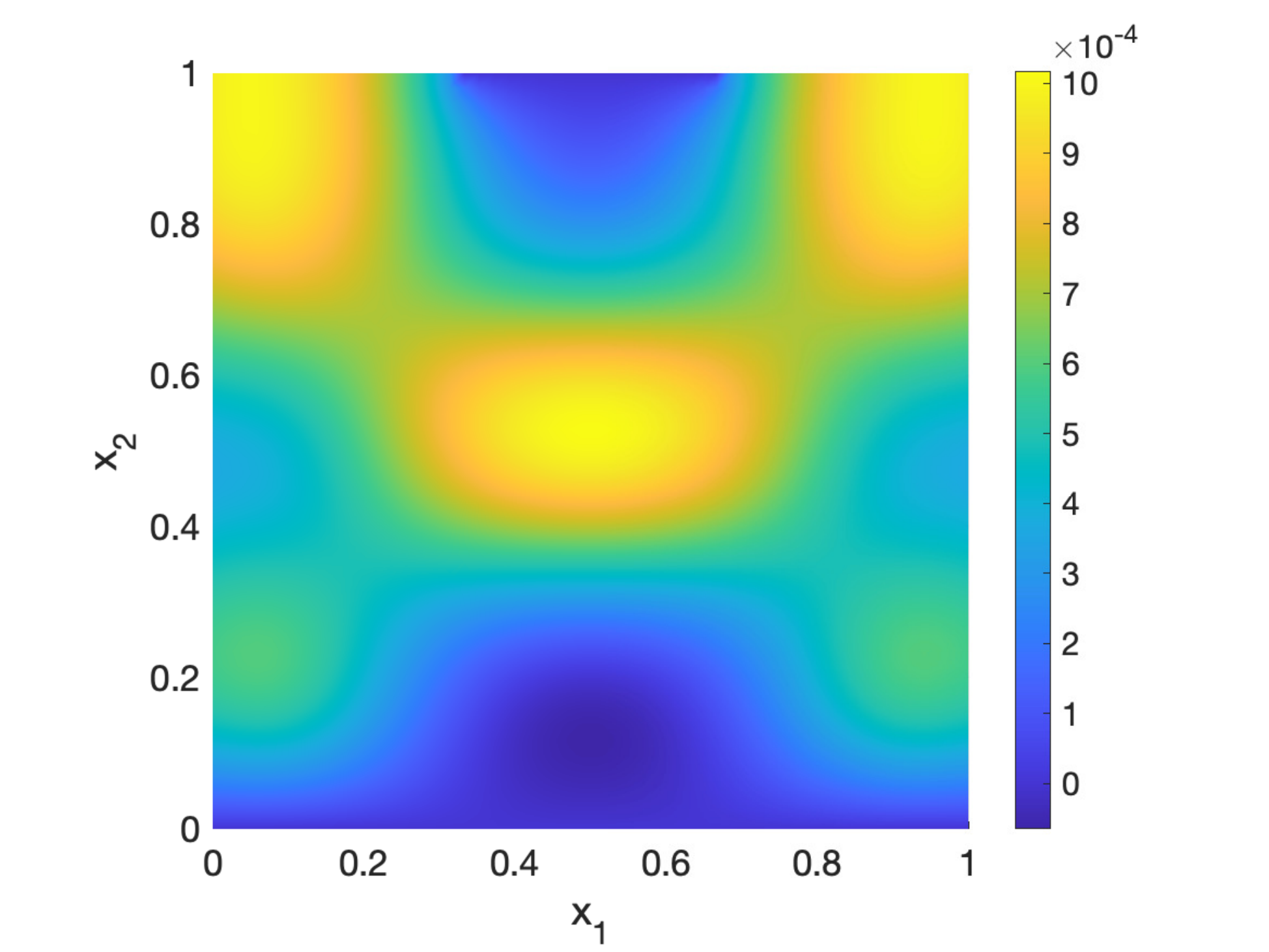} &
          \includegraphics[width=0.115\textwidth]{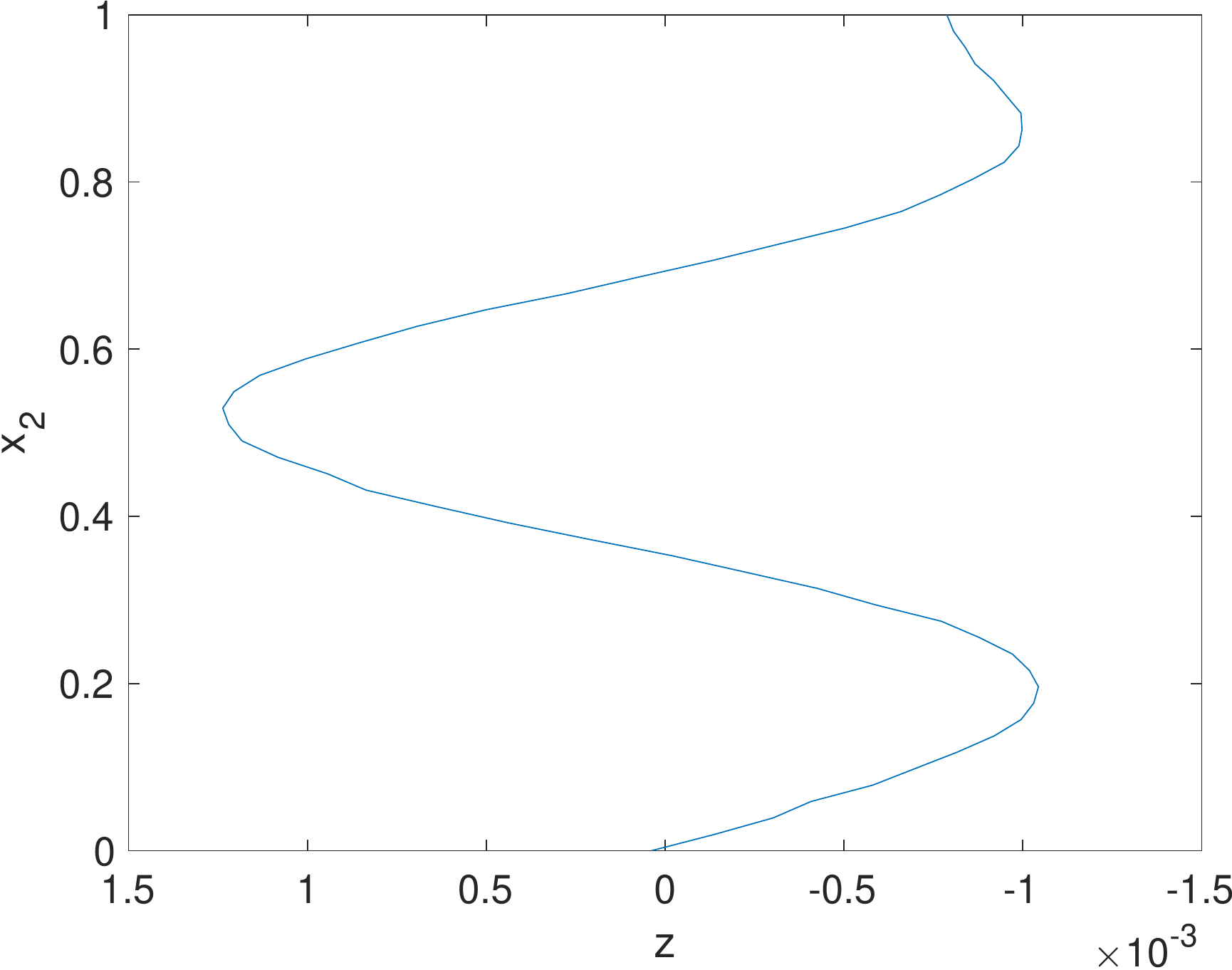} &
            \includegraphics[width=0.115\textwidth]{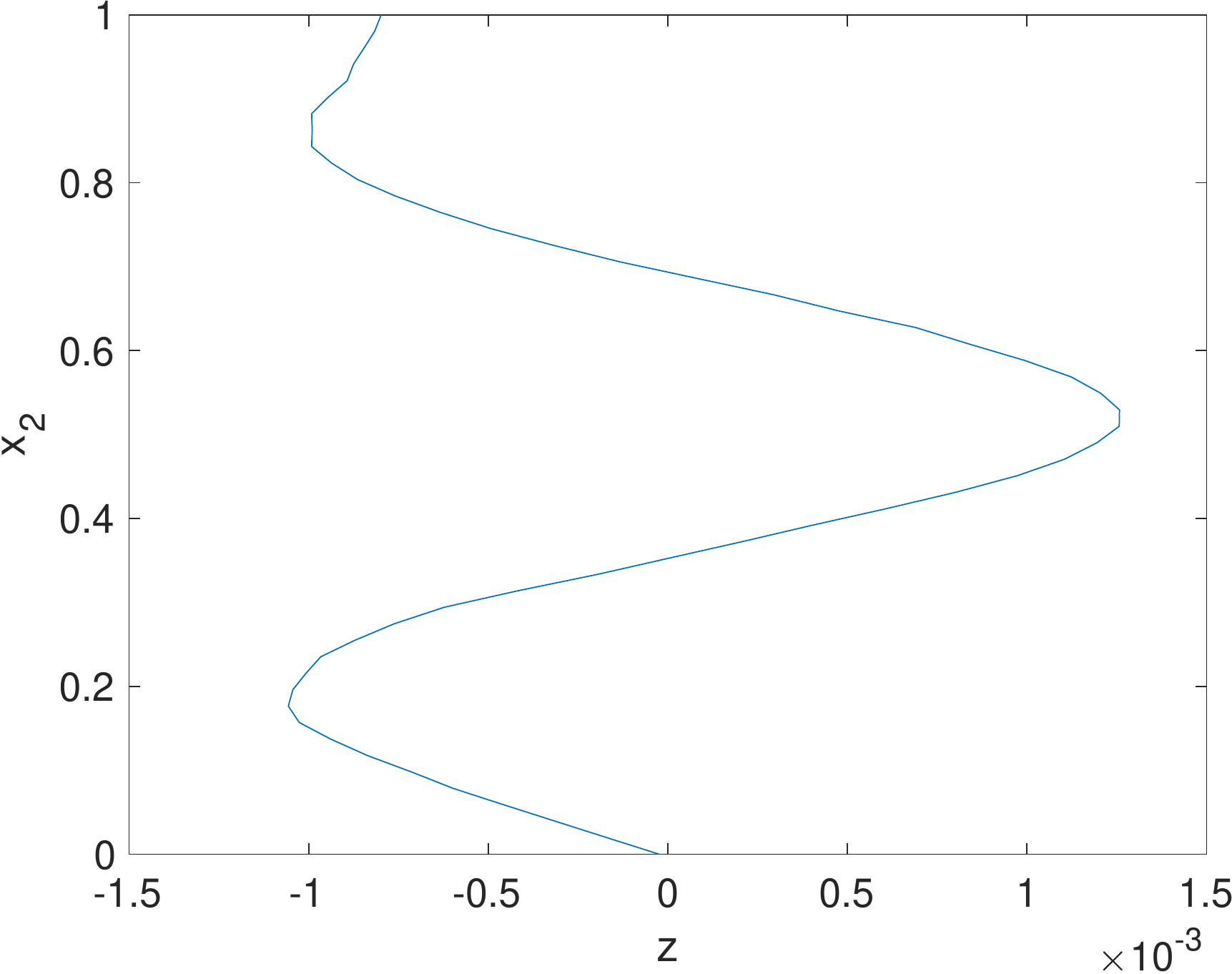}  \\
                                             .0011 &
  \includegraphics[width=0.115\textwidth]{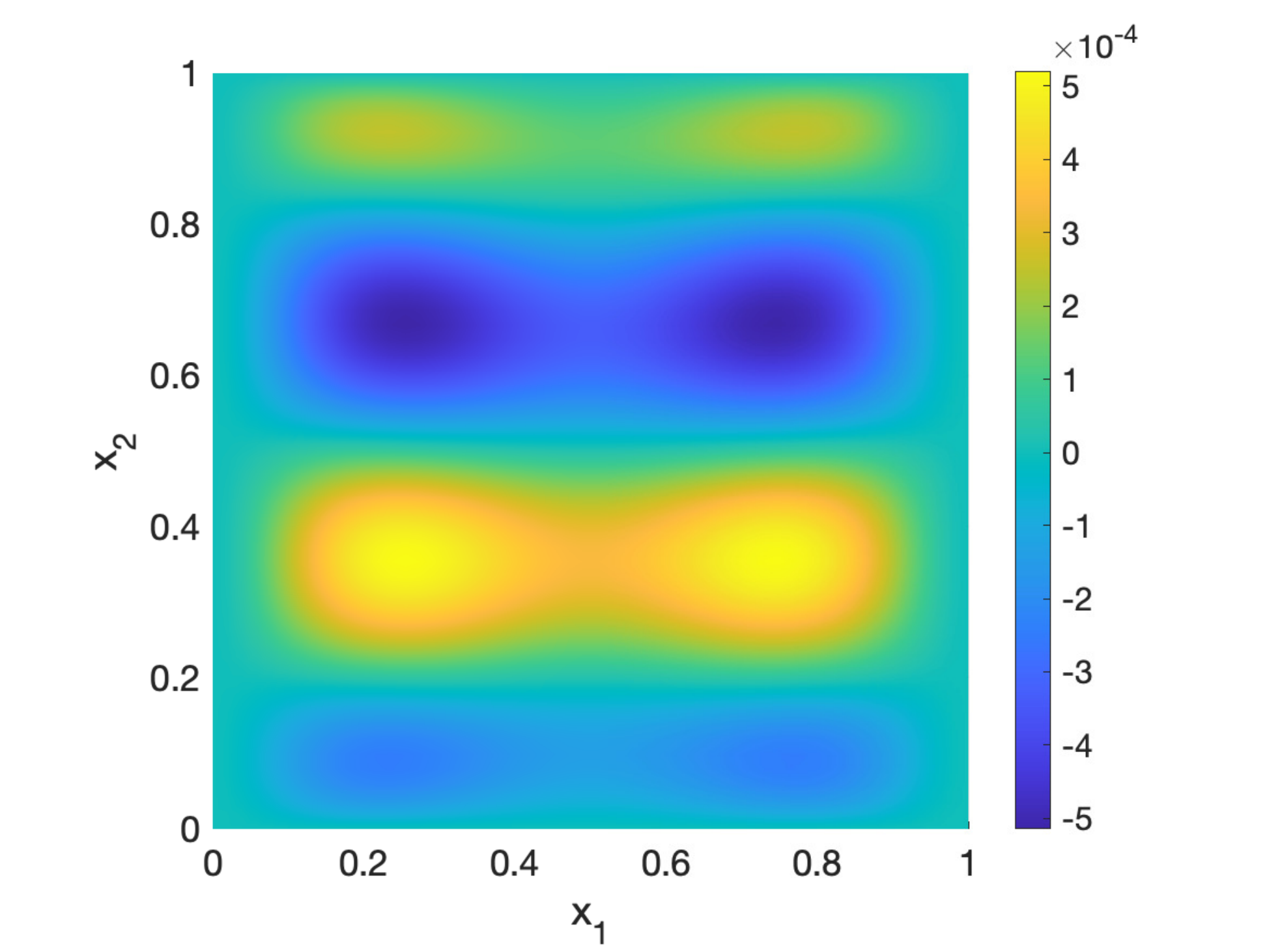} &
    \includegraphics[width=0.115\textwidth]{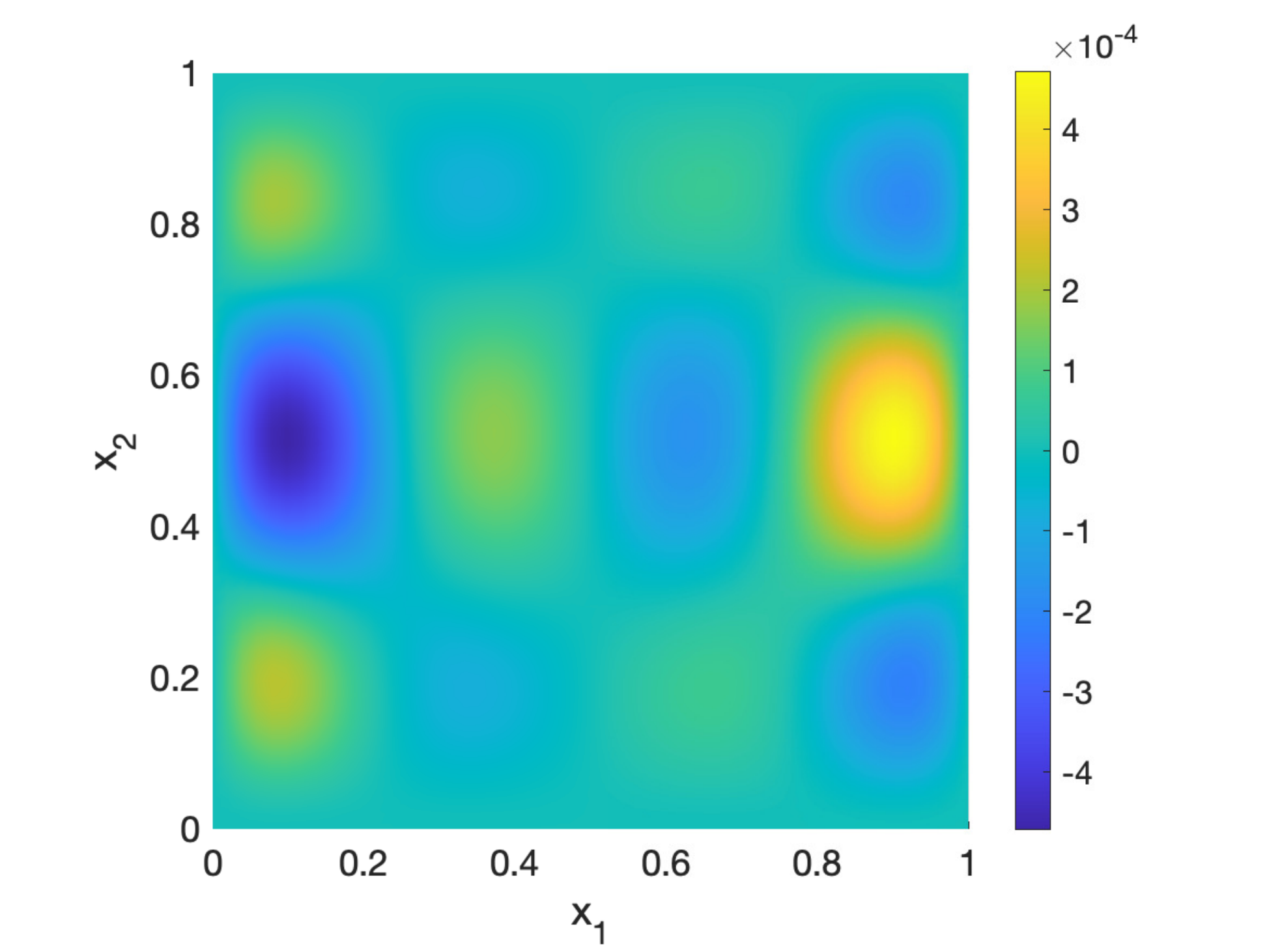} &
        \includegraphics[width=0.115\textwidth]{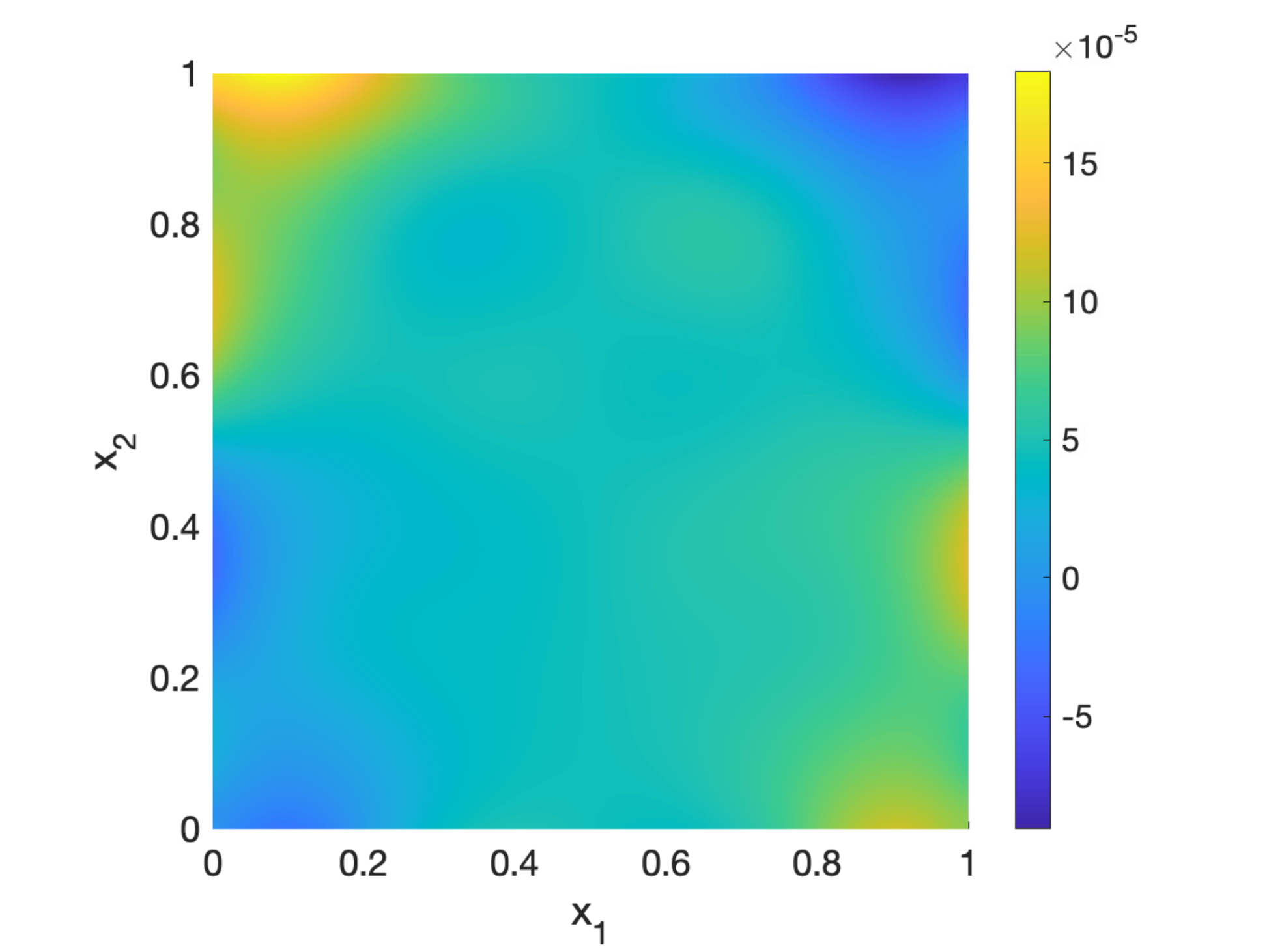} &
         \includegraphics[width=0.115\textwidth]{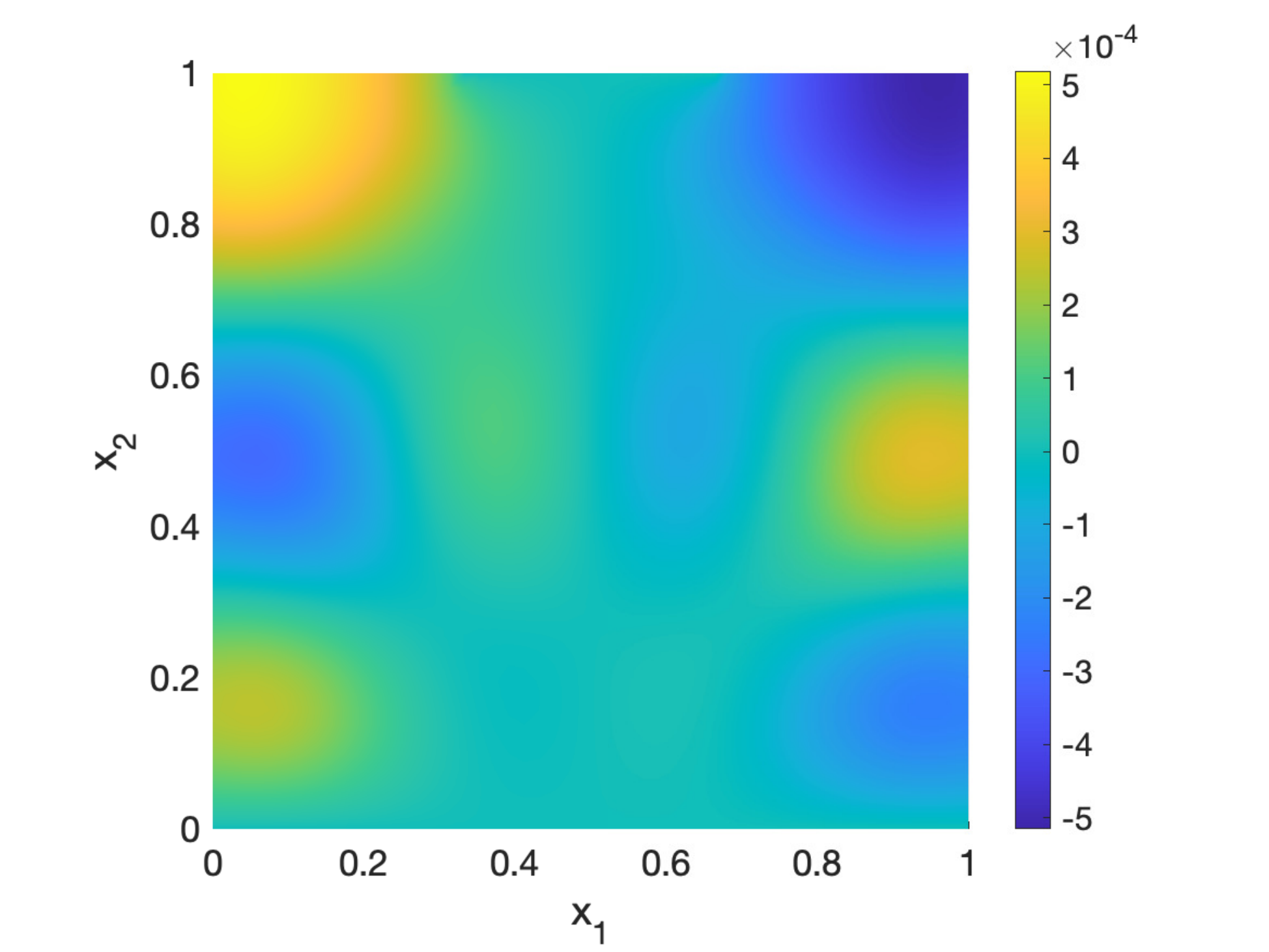} &
          \includegraphics[width=0.115\textwidth]{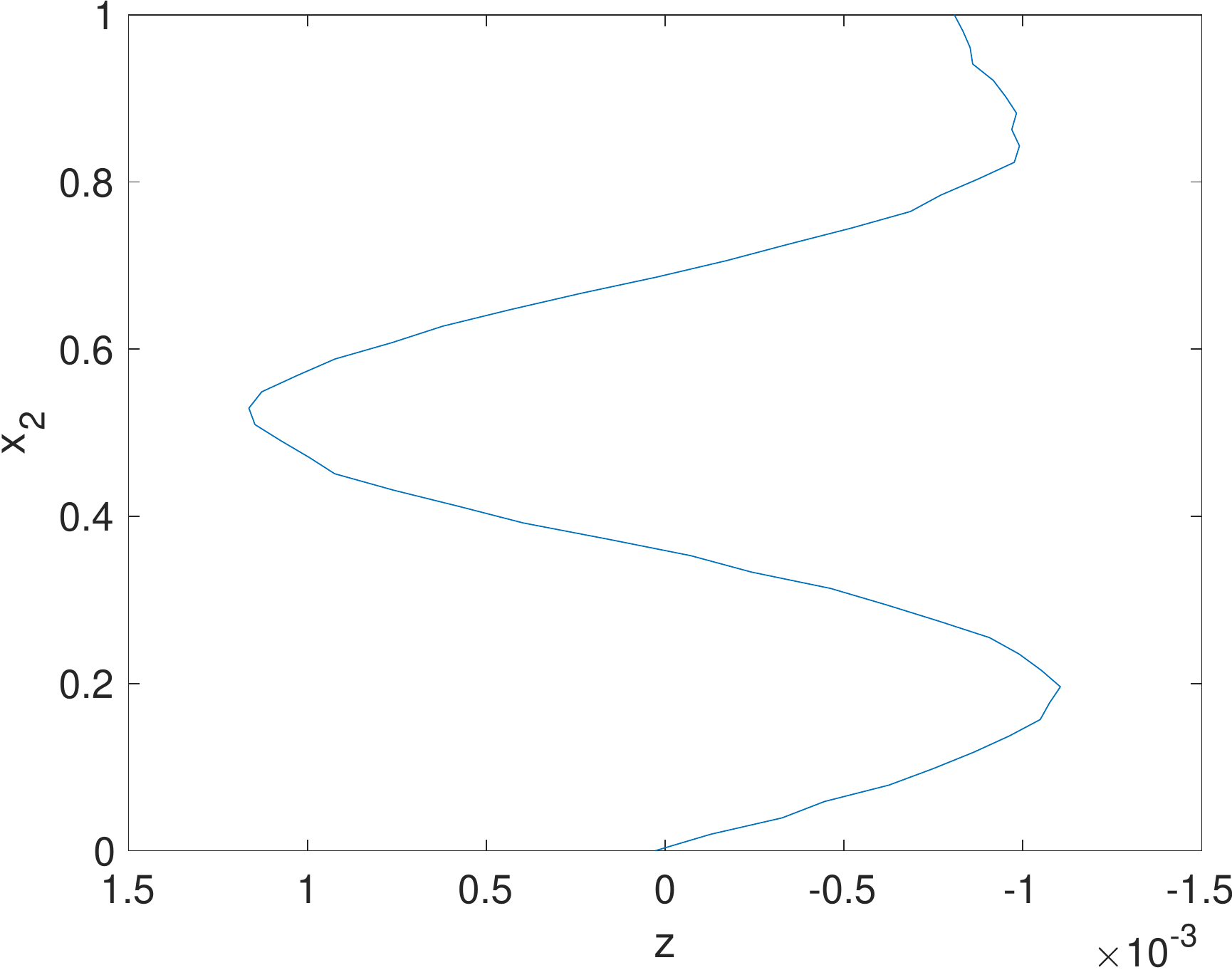} &
            \includegraphics[width=0.115\textwidth]{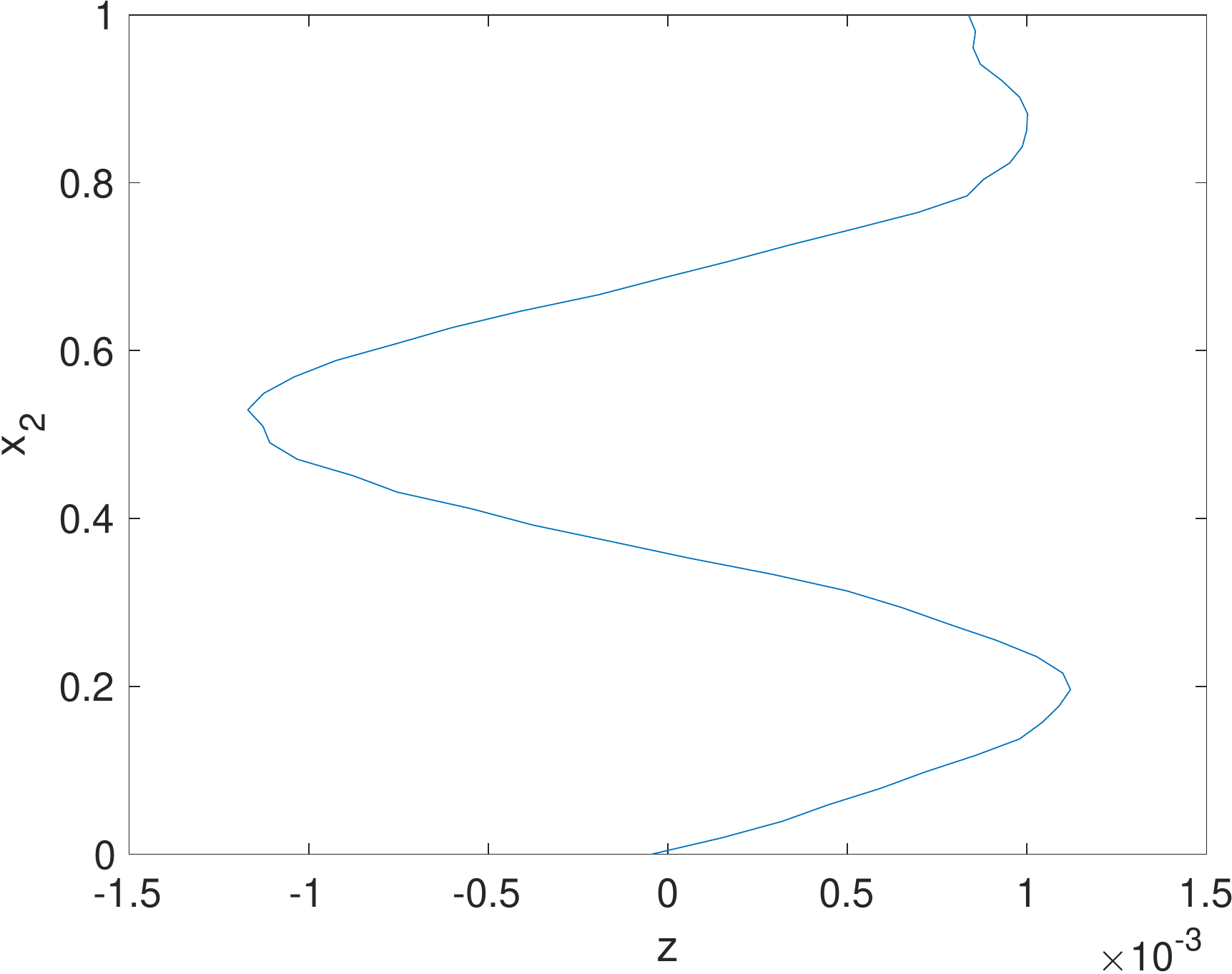}  \\
  \end{tabular}
    \caption{Leading model discrepancies, $\d(\overline{\z},\t_N)$, and the corresponding perturbations of the optimal solution, $\sigma_N \w_N$, for the thermal-fluid problem. From top to bottom, each row corresponds to a singular vector, $\t_N$, $N=1,2,\dots,6$. Across each row, the panels show the model discrepancies for the horizontal velocity, vertical velocity, pressure, and temperature, and the perturbations of the optimal controllers on the left and right boundaries.}
  \label{fig:tf_sing_vectors}
\end{figure}

\section{Conclusion} \label{sec:conclusion}
This article introduces a new approach to analyze the effect of model discrepancy for optimization problems constrained by PDEs. We leverage the PDE discretization and mathematical expressions for post-optimality sensitivities to define a general representation of model discrepancy. The resulting parameterization of the model discrepancy is high dimensional, scaling like the product of the discretized state dimension times the discretized controller dimension, which is reflective of the minimal assumptions we place on the form of the discrepancy. However, thanks to Kronecker product structure, the leading singular values/vectors of the post-optimality sensitivity operator are computed with a randomized algorithm whose computational complexity scales with the sum of the discretized state and control dimensions rather than the product. By combining the Kronecker product representation, a randomized generalized SVD algorithm, and adjoint-based derivative calculations, the result is an algorithm which is computationally scalable and hence practical for large-scale PDECO. 

As it commonly the case in uncertainty quantification, exploiting low rank structure in high dimensional parameter spaces is paramount to achieving computational efficiency. In our proposed framework, we do not impose physics specific structure on the model discrepancy representation and hence it is not necessarily low rank. However, physics constraints such as conservation properties, boundary conditions, or other physical properties are imposed through weighting matrices which define inner products. The coupling of these weighting matrices and the optimization objectives (data misfit, state target, regularization, etc.) constrain the analysis to a subspace informed by both prior knowledge and decision-making goals. Our proposed approach leverages existing investment in PDECO infrastructure (PDE discretization tools, derivative calculations, etc.) to enable rapid development while simultaneously enabling application specification in the form of weak constraints imposed through weighting matrices. Mathematical structure in the algorithm enables general purpose software development which is both computationally scalable and portable between applications. 

Our proposed approach has potential impact for a range of applications. In the context of inverse problems, the model discrepancy sensitivities provide a first order approximation for quantification of uncertainty due to model discrepancy. When proceeding through model development, the sensitivities guide the allocation of effort needed to achieve reliable optimal solutions. This has implications for both first principles physics developments (which physics simplifications are admissible) and reduced order model development (in a general sense including machine learning approaches trained on high-fidelity models). For applications with a real-time component, optimization is typically performed on simplified or reduced models to achieve fast computation. The model discrepancy sensitivities are positioned to complement this optimization. A companion article~\cite{model_discrepancy_2} demonstrates how high-fidelity data may be used to calibrate the model discrepancy in a Bayesian framework. By coupling high-fidelity data with post-optimality sensitivities we provide both an updating mechanism for real-time computation and uncertainty quantification in the optimal solution due to the model discrepancy.

\section*{Appendix}
Proof of Theorem~\eqref{thm:M_theta_inv}.
\begin{proof}
First observe the identities
\begin{enumerate}
\item $$\overline{\z}^T \M_z \x= \frac{\beta}{1+\beta},$$
\item $$\left( \vec{\Gamma}^{-1}-\G \right) \overline{\z} = \frac{1}{1+\beta} \vec{\Gamma}^{-1} \overline{\z} ,$$
\item $$ \E \x = \M_z \overline{\z}$$
\item $$\overline{\z} \hspace{.4 mm}  \overline{\z}^T \left( \vec{\Gamma}^{-1}-\G \right) = \frac{1}{1+\beta} \overline{\z} \hspace{.4 mm} \overline{\z}^T \vec{\Gamma}^{-1}.$$
\end{enumerate}
The block-wise matrix-matrix multiplication
\begin{align*}
& \left( \begin{array}{cc}
\L & \L \otimes \overline{\z}^T \M_z \\
\L \otimes \M_z \overline{\z} & \L \otimes \E
\end{array} \right) 
\left( \begin{array}{cc}
\L^{-1} & -\L^{-1} \otimes \x^T \\
- \L^{-1} \otimes \x & \L^{-1} \otimes \N
\end{array} \right) \\
\end{align*}
simplifies to:\\
(1,1) block:
\begin{align*}
& \L \L^{-1} + \left( \L \otimes \overline{\z}^T \M_z \right) \left( -\L^{-1} \otimes \x \right) \\
& = \I_m - \I_m \otimes \left( \x^T \M_z \overline{\z}  \right) \\
& = \I_m - \frac{\beta}{1+\beta} \I_m  \qquad \text{(identity 1)} \\
& = \frac{1}{1+\beta} \I_m
\end{align*}
(1,2) block:
\begin{align*}
& \L \left( -\L^{-1} \otimes  \x^T \right) + \left( \L \otimes \overline{\z}^T \M_z  \right) \left(  \L^{-1} \otimes \N \right)    \\
 &=  \I_m \otimes \left( \overline{\z}^T \M_z  \N -  \x^T \right) \\
 &= \I_m \otimes \left(  \left[ \overline{\z}^T \frac{1}{1+\beta} \vec{\Gamma}^{-1} - \overline{\z}^T \left( \vec{\Gamma}^{-1} - \G \right) \right] \M_z^{-1} \right) \\
 & = \I_m \otimes \left( \left[  \frac{1}{1+\beta} \vec{\Gamma}^{-1} \overline{\z}  - \frac{1}{1+\beta}  \vec{\Gamma}^{-1}\overline{\z} \right]^T \M_z^{-1} \right)  \qquad \text{(identity 2)}\\
 & =  \vec{0}
\end{align*}
(2,1) block:
\begin{align*}
& \left( \L \otimes \M_z \overline{\z} \right) \L^{-1} +\left( \L \otimes \E \right) \left(- \L^{-1} \otimes \x \right) \\
&= \I_m \otimes \M_z \overline{\z} - \I_m \otimes \M_z \overline{\z}  \qquad \text{(identity 3)} \\
& = \vec{0}
\end{align*}
(2,2) block:
\begin{align*}
& \left( \L \otimes \M_z \overline{\z}  \right) \left( -\L^{-1} \otimes \x^T \right) + \left( \L \otimes \E \right) \left( \L^{-1} \otimes \N \right) \\
& = \I_m \otimes \left( \E \N - \M_z \overline{\z} \hspace{.4 mm} \x^T \right) \\
& = \I_m \otimes \left( \M_z \left[ (\vec{\Gamma} + \overline{\z} \hspace{.4 mm} \overline{\z}^T )\frac{1}{1+\beta} \vec{\Gamma}^{-1}- \overline{\z} \hspace{.4 mm} \overline{\z}^T ( \vec{\Gamma}^{-1}-\G) \right] \M_z^{-1} \right) \\
& = \I_m \otimes \left( \M_z \left[  \frac{1}{1+\beta}  \left( \I_n + \overline{\z} \hspace{.4 mm}  \overline{\z}^T  \vec{\Gamma}^{-1} \right)- \frac{1}{1+\beta} \overline{\z} \hspace{.4 mm} \overline{\z}^T \vec{\Gamma}^{-1}\right] \M_z^{-1} \right) \qquad \text{(identity 4)} \\
& = \frac{1}{1+\beta} \I_m \otimes \I_n .
\end{align*}
Multiplication by a constant completes the proof.
\end{proof}

\begin{table}[!ht]
\centering
\begin{tabular}{|c|c|}
\hline
\textbf{Illustrative Example} & \hspace{1 mm} \\
PDE domain & $\Omega = (0,1)$ \\
True source & $z^\star(x) = e^{(-50 (x-0.5)^2)}$ \\
\hline
\textbf{Convection-diffusion reaction} & \hspace{1 mm} \\
PDE domain & $\Omega = (0,1)^2$ \\
Dirichlet boundary & $\Gamma_d = (0,1) \times \{ 0 \}$ \\
Neumann Boundary & $\Gamma_n=\partial \Omega \setminus \Gamma_d$ \\
Target state & $T(x_1,x_2)=(3x_1^2-3x_1^3)(2x_2-x_2^2)$ \\
Diffusion coefficient & $\nu=1$ \\
Velocity field & $\v(x_1,x_2)=(\cos(2 \pi x_1),1+\cos^2(2\pi x_2))$ \\
Reaction function & $R(u)=u^3+u$ \\
Regularization coefficients & $\beta_1=\beta_2=10^{-6}$ \\
Discretization & Linear finite elements \\
Discretization dimension & $\u,\z \in \R^{10000}$ \\
\hline
\textbf{Thermal-fluid} & \hspace{1 mm} \\
PDE domain & $\Omega = (0,1)^2$ \\
Inflow boundary & $\Gamma_i = [1/3,2/3] \times \{1\}$ \\
Outflow boundary & $\Gamma_o = [0,1/3] \times \{1\} \cup [2/3,1] \times \{1\}$ \\
Bottom boundary & $\Gamma_b = [0,1] \times \{0\}$ \\
Control boundary & $\Gamma_c = \{0,1\} \times [0,1]$ \\
Inflow boundary condition & $v_i(x_1) = -4(x_1-\frac{1}{3})(\frac{2}{3}-x_1)$ \\
Outflow boundary condition on $[0,\frac{1}{3}]$ & $v_o(x_1) =  2( \frac{1}{3}-x_1)x_1$ \\
Outflow boundary condition on $[\frac{2}{3},1]$ &  $v_o(x_1) = 2(x_1-\frac{2}{3})(1-x_1) $ \\
Reynolds number & $Re=10^2$ \\
Grashof number & $Ge=10^4$ \\
Prandtl number & $Pr=0.72$ \\
Coefficient & $\kappa = \frac{1}{RePr}$ \\
Coefficient & $\eta = \frac{Ge}{Re^2}$ \\
Control penalty & $\gamma = 0.5$ \\
Discretization & Q2-Q1 Taylor-Hood finite elements \\
Discretization dimension & $\u \in \R^{34531}$, $\z \in \R^{206}$ \\
\hline
\end{tabular}
\caption{Parameters defining the optimization problems presented in Section~\ref{sec:numerical_results}. In all examples $\vec{n}$ denotes the outward facing normal vector to the boundary.}
\label{tab:numerical_results_problem_params}
\end{table}

\section*{Acknowledgements}
This paper describes objective technical
results and analysis. Any subjective views or opinions that might be
expressed in the paper do not necessarily represent the views of the
U.S. Department of Energy or the United States Government. Sandia
National Laboratories is a multimission laboratory managed and
operated by National Technology and Engineering Solutions of Sandia
LLC, a wholly owned subsidiary of Honeywell International, Inc., for
the U.S. Department of Energy's National Nuclear Security
Administration under contract DE-NA-0003525. SAND2022-14143 O.

\bibliography{dasco}


\begin{thebibliography}{39}
\ifx \bisbn   \undefined \def \bisbn  #1{ISBN #1}\fi
\ifx \binits  \undefined \def \binits#1{#1}\fi
\ifx \bauthor  \undefined \def \bauthor#1{#1}\fi
\ifx \batitle  \undefined \def \batitle#1{#1}\fi
\ifx \bjtitle  \undefined \def \bjtitle#1{#1}\fi
\ifx \bvolume  \undefined \def \bvolume#1{\textbf{#1}}\fi
\ifx \byear  \undefined \def \byear#1{#1}\fi
\ifx \bissue  \undefined \def \bissue#1{#1}\fi
\ifx \bfpage  \undefined \def \bfpage#1{#1}\fi
\ifx \blpage  \undefined \def \blpage #1{#1}\fi
\ifx \burl  \undefined \def \burl#1{\textsf{#1}}\fi
\ifx \doiurl  \undefined \def \doiurl#1{\url{https://doi.org/#1}}\fi
\ifx \betal  \undefined \def \betal{\textit{et al.}}\fi
\ifx \binstitute  \undefined \def \binstitute#1{#1}\fi
\ifx \binstitutionaled  \undefined \def \binstitutionaled#1{#1}\fi
\ifx \bctitle  \undefined \def \bctitle#1{#1}\fi
\ifx \beditor  \undefined \def \beditor#1{#1}\fi
\ifx \bpublisher  \undefined \def \bpublisher#1{#1}\fi
\ifx \bbtitle  \undefined \def \bbtitle#1{#1}\fi
\ifx \bedition  \undefined \def \bedition#1{#1}\fi
\ifx \bseriesno  \undefined \def \bseriesno#1{#1}\fi
\ifx \blocation  \undefined \def \blocation#1{#1}\fi
\ifx \bsertitle  \undefined \def \bsertitle#1{#1}\fi
\ifx \bsnm \undefined \def \bsnm#1{#1}\fi
\ifx \bsuffix \undefined \def \bsuffix#1{#1}\fi
\ifx \bparticle \undefined \def \bparticle#1{#1}\fi
\ifx \barticle \undefined \def \barticle#1{#1}\fi
\bibcommenthead
\ifx \bconfdate \undefined \def \bconfdate #1{#1}\fi
\ifx \botherref \undefined \def \botherref #1{#1}\fi
\ifx \url \undefined \def \url#1{\textsf{#1}}\fi
\ifx \bchapter \undefined \def \bchapter#1{#1}\fi
\ifx \bbook \undefined \def \bbook#1{#1}\fi
\ifx \bcomment \undefined \def \bcomment#1{#1}\fi
\ifx \oauthor \undefined \def \oauthor#1{#1}\fi
\ifx \citeauthoryear \undefined \def \citeauthoryear#1{#1}\fi
\ifx \endbibitem  \undefined \def \endbibitem {}\fi
\ifx \bconflocation  \undefined \def \bconflocation#1{#1}\fi
\ifx \arxivurl  \undefined \def \arxivurl#1{\textsf{#1}}\fi
\csname PreBibitemsHook\endcsname

\bibitem{box_1979}
\begin{botherref}
\oauthor{\bsnm{Box}, \binits{G.E.P.}}:
Robustness in the strategy of scientific model building.
Robustness in Statistics,
201--236
(1979)
\end{botherref}
\endbibitem

\bibitem{uq_handbook}
\begin{bbook}
\beditor{\bsnm{Ghanem}, \binits{R.}},
\beditor{\bsnm{Higdon}, \binits{D.}},
\beditor{\bsnm{Owhadi}, \binits{H.}} (eds.):
\bbtitle{Handbook of Uncertainty Quantification}.
\bpublisher{Springer},
\blocation{Cham}
(\byear{2016})
\end{bbook}
\endbibitem

\bibitem{ohagan2001}
\begin{barticle}
\bauthor{\bsnm{Kennedy}, \binits{M.C.}},
\bauthor{\bsnm{O'Hagan}, \binits{A.}}:
\batitle{Bayesian calibration of computer models}.
\bjtitle{Journal of the Royal Statistical Society}
\bvolume{63}(\bissue{3}),
\bfpage{425}--\blpage{464}
(\byear{2001})
\end{barticle}
\endbibitem

\bibitem{Ling_2014}
\begin{barticle}
\bauthor{\bsnm{Ling}, \binits{Y.}},
\bauthor{\bsnm{Mullins}, \binits{J.}},
\bauthor{\bsnm{Mahadevan}, \binits{S.}}:
\batitle{Selection of model discrepancy priors in {B}ayesian calibration}.
\bjtitle{Journal of Computational Physics}
\bvolume{276},
\bfpage{665}--\blpage{680}
(\byear{2014})
\end{barticle}
\endbibitem

\bibitem{Maupin}
\begin{botherref}
\oauthor{\bsnm{Maupin}, \binits{K.A.}},
\oauthor{\bsnm{Swiler}, \binits{L.P.}}:
Model discrepancy calibration across experimental settings.
Reliability Engineering \& System Safety
\textbf{200}
(2020)
\end{botherref}
\endbibitem

\bibitem{Arendt_2012}
\begin{botherref}
\oauthor{\bsnm{Arendt}, \binits{P.D.}},
\oauthor{\bsnm{Apley}, \binits{D.W.}},
\oauthor{\bsnm{Chen}, \binits{W.}}:
Quantification of model uncertainty: Calibration, model discrepancy, and
  identifiability.
Journal of Mechanical Design
\textbf{134}
(2012)
\end{botherref}
\endbibitem

\bibitem{Higdon_2008}
\begin{barticle}
\bauthor{\bsnm{Higdon}, \binits{D.}},
\bauthor{\bsnm{Gattiker}, \binits{J.}},
\bauthor{\bsnm{Williams}, \binits{B.}},
\bauthor{\bsnm{Rightley}, \binits{M.}}:
\batitle{Computer model calibration using high-dimensional output}.
\bjtitle{Journal of the American Statistical Association}
\bvolume{103}(\bissue{482}),
\bfpage{570}--\blpage{583}
(\byear{2008})
\end{barticle}
\endbibitem

\bibitem{bayes_approx_error_petra}
\begin{botherref}
\oauthor{\bsnm{Nicholson}, \binits{R.}},
\oauthor{\bsnm{Petra}, \binits{N.}},
\oauthor{\bsnm{Kaipio}, \binits{J.P.}}:
Estimation of the {R}obin coefficient field in a {P}oisson problem with
  uncertain conductivity field.
Inverse Problems
\textbf{34}
(2018)
\end{botherref}
\endbibitem

\bibitem{kopke_2018}
\begin{barticle}
\bauthor{\bsnm{K{\"o}pke}, \binits{C.}},
\bauthor{\bsnm{Irving}, \binits{J.}},
\bauthor{\bsnm{Elsheikh}, \binits{A.H.}}:
\batitle{Accounting for model error in {B}ayesian solutions to hydrogeophysical
  inverse problems using a local basis approach}.
\bjtitle{Advances in Water Resources}
\bvolume{116},
\bfpage{195}--\blpage{207}
(\byear{2018})
\end{barticle}
\endbibitem

\bibitem{Kaipio_2008}
\begin{botherref}
\oauthor{\bsnm{Nissinen}, \binits{A.}},
\oauthor{\bsnm{Heikkinen}, \binits{L.M.}},
\oauthor{\bsnm{Kaipio}, \binits{J.P.}}:
The {B}ayesian approximation error approach for electrical impedance
  tomography---experimental results.
Measurement Science and Technology
\textbf{19}
(208)
\end{botherref}
\endbibitem

\bibitem{Sargsyan_2019}
\begin{barticle}
\bauthor{\bsnm{Sargsyan}, \binits{K.}},
\bauthor{\bsnm{Huan}, \binits{X.}},
\bauthor{\bsnm{Najm.}, \binits{H.}}:
\batitle{Embedded model error representation for {B}ayesian model calibration}.
\bjtitle{International Journal for Uncertainty Quantification}
\bvolume{9}(\bissue{4}),
\bfpage{365}--\blpage{394}
(\byear{2019})
\end{barticle}
\endbibitem

\bibitem{sargsyan_2018}
\begin{barticle}
\bauthor{\bsnm{Huan}, \binits{X.}},
\bauthor{\bsnm{Safta}, \binits{C.}},
\bauthor{\bsnm{Sargsyan}, \binits{K.}},
\bauthor{\bsnm{Geraci}, \binits{G.}},
\bauthor{\bsnm{Eldred}, \binits{M.S.}},
\bauthor{\bsnm{Vane}, \binits{Z.P.}},
\bauthor{\bsnm{Lacaze}, \binits{G.}},
\bauthor{\bsnm{Oefelein}, \binits{J.C.}},
\bauthor{\bsnm{Najm}, \binits{H.N.}}:
\batitle{Global sensitivity analysis and estimation of model error, toward
  uncertainty quantification in scramjet computations}.
\bjtitle{AIAA Journal}
\bvolume{56}(\bissue{3}),
\bfpage{1170}--\blpage{1184}
(\byear{2018})
\end{barticle}
\endbibitem

\bibitem{Sargsyan_2015}
\begin{barticle}
\bauthor{\bsnm{Sargsyan}, \binits{K.}},
\bauthor{\bsnm{Najm}, \binits{H.N.}},
\bauthor{\bsnm{Ghanem.}, \binits{R.}}:
\batitle{On the statistical calibration of physical models}.
\bjtitle{International Journal of Chemical Kinetics}
\bvolume{47}(\bissue{4}),
\bfpage{246}--\blpage{276}
(\byear{2015})
\end{barticle}
\endbibitem

\bibitem{morrison_2018}
\begin{barticle}
\bauthor{\bsnm{Morrison}, \binits{R.E.}},
\bauthor{\bsnm{Oliver}, \binits{T.A.}},
\bauthor{\bsnm{Moser}, \binits{R.D.}}:
\batitle{Representing model inadequacy: A stochastic operator approach}.
\bjtitle{SIAM/ASA J. Uncertain. Quantif.}
\bvolume{6}(\bissue{2}),
\bfpage{457}--\blpage{496}
(\byear{2018})
\end{barticle}
\endbibitem

\bibitem{portone}
\begin{botherref}
\oauthor{\bsnm{Portone}, \binits{T.}}:
Representing model-form uncertainty from missing microstructural information.
PhD thesis,
UT Austin
(2019)
\end{botherref}
\endbibitem

\bibitem{post_opt_tutorial}
\begin{bchapter}
\bauthor{\bsnm{Fiacco}, \binits{A.V.}},
\bauthor{\bsnm{Kyparisis}, \binits{J.}}:
\bctitle{A tutorial on parametric nonlinear programming sensitivity and
  stability analysis}.
\bbtitle{Systems and management science by extremal methods}.
\bpublisher{Springer},
\blocation{Boston, MA.}
(\byear{1992})
\end{bchapter}
\endbibitem

\bibitem{shapiro_SIAM_review}
\begin{barticle}
\bauthor{\bsnm{Bonnans}, \binits{J.F.}},
\bauthor{\bsnm{Shapiro}, \binits{A.}}:
\batitle{Optimization problems with perturbations: A guided tour}.
\bjtitle{SIAM Review}
\bvolume{40}(\bissue{2}),
\bfpage{228}--\blpage{264}
(\byear{1998})
\end{barticle}
\endbibitem

\bibitem{Griesse_part_1}
\begin{barticle}
\bauthor{\bsnm{Griesse}, \binits{R.}}:
\batitle{Parametric sensitivity analysis in optimal control of a reaction
  diffusion system. {P}art--{I}. solution differentiability}.
\bjtitle{Numerical Functional Analysis and Optimization}
\bvolume{25}(\bissue{1-2}),
\bfpage{93}--\blpage{117}
(\byear{2004})
\end{barticle}
\endbibitem

\bibitem{Griesse_part_2}
\begin{barticle}
\bauthor{\bsnm{Griesse}, \binits{R.}}:
\batitle{Parametric sensitivity analysis in optimal control of a
  reaction-diffusion system -- part {II}: Practical methods and examples}.
\bjtitle{Optimization Methods and Software}
\bvolume{19}(\bissue{2}),
\bfpage{217}--\blpage{242}
(\byear{2004})
\end{barticle}
\endbibitem

\bibitem{griesse2}
\begin{barticle}
\bauthor{\bsnm{Brandes}, \binits{K.}},
\bauthor{\bsnm{Griesse}, \binits{R.}}:
\batitle{Quantitative stability analysis of optimal solutions in
  {PDE}-constrained optimization}.
\bjtitle{Journal of Computational and Applied Mathematics}
\bvolume{206},
\bfpage{908}--\blpage{926}
(\byear{2007})
\end{barticle}
\endbibitem

\bibitem{HDSA}
\begin{barticle}
\bauthor{\bsnm{Hart}, \binits{J.}},
\bauthor{\bsnm{{van Bloemen Waanders}}, \binits{B.}},
\bauthor{\bsnm{Hertzog}, \binits{R.}}:
\batitle{Hyper-differential sensitivity analysis of uncertain parameters in
  {PDE}-constrained optimization}.
\bjtitle{International Journal for Uncertainty Quantification}
\bvolume{10}(\bissue{3}),
\bfpage{225}--\blpage{248}
(\byear{2020})
\end{barticle}
\endbibitem

\bibitem{sunseri_hdsa}
\begin{botherref}
\oauthor{\bsnm{Sunseri}, \binits{I.}},
\oauthor{\bsnm{Hart}, \binits{J.}},
\oauthor{\bsnm{{van Bloemen Waanders}}, \binits{B.}},
\oauthor{\bsnm{Alexanderian}, \binits{A.}}:
Hyper-differential sensitivity analysis for inverse problems constrained by
  partial differential equations.
Inverse Problems
\textbf{36}(12)
(2020)
\end{botherref}
\endbibitem

\bibitem{saibaba_gsvd}
\begin{botherref}
\oauthor{\bsnm{Saibaba}, \binits{A.K.}},
\oauthor{\bsnm{Hart}, \binits{J.}},
\oauthor{\bsnm{{van Bloemen Waanders}}, \binits{B.}}:
Randomized algorithms for generalized singular value decomposition with
  application to sensitivity analysis.
Numerical Linear Algebra with Applications
(2021)
\end{botherref}
\endbibitem

\bibitem{hart_2021_bayes}
\begin{botherref}
\oauthor{\bsnm{Hart}, \binits{J.}},
\oauthor{\bsnm{{van Bloemen Waanders}}, \binits{B.}}:
Enabling hyper-differential sensitivity analysis for ill-posed inverse
  problems.
arXiv:2106.11813
(2022)
\end{botherref}
\endbibitem

\bibitem{model_discrepancy_2}
\begin{botherref}
\oauthor{\bsnm{Hart}, \binits{J.}},
\oauthor{\bsnm{{van Bloemen Waanders}}, \binits{B.}}:
Hyper-differential sensitivity analysis with respect to model discrepancy:
  Calibration and optimal solution updating.
arXiv:2210.09044
(2022)
\end{botherref}
\endbibitem

\bibitem{Vogel_99}
\begin{botherref}
\oauthor{\bsnm{Vogel}, \binits{C.R.}}:
Sparse matrix computations arising in distributed parameter identification.
SIAM J. Matrix Anal. Appl.,
1027--1037
(1999)
\end{botherref}
\endbibitem

\bibitem{Archer_01}
\begin{barticle}
\bauthor{\bsnm{Ascher}, \binits{U.M.}},
\bauthor{\bsnm{Haber}, \binits{E.}}:
\batitle{Grid refinement and scaling for distributed parameter estimation
  problems}.
\bjtitle{Inverse Problems}
\bvolume{17},
\bfpage{571}--\blpage{590}
(\byear{2001})
\end{barticle}
\endbibitem

\bibitem{Haber_01}
\begin{barticle}
\bauthor{\bsnm{Haber}, \binits{E.}},
\bauthor{\bsnm{Ascher}, \binits{U.M.}}:
\batitle{Preconditioned all-at-once methods for large, sparse parameter
  estimation problems}.
\bjtitle{Inverse Problems}
\bvolume{17},
\bfpage{1847}--\blpage{1864}
(\byear{2001})
\end{barticle}
\endbibitem

\bibitem{Vogel_02}
\begin{bbook}
\bauthor{\bsnm{Vogel}, \binits{C.R.}}:
\bbtitle{Computational Methods for Inverse Problems}.
\bpublisher{SIAM},
\blocation{Frontiers in Applied Mathematics Series}
(\byear{2002})
\end{bbook}
\endbibitem

\bibitem{Biegler_03}
\begin{bbook}
\beditor{\bsnm{L.~T.~Biegler}, \binits{M.H.} \bsuffix{O.~Ghattas}},
\beditor{\bparticle{van} \bsnm{Bloemen~Waanders}, \binits{B.}} (eds.):
\bbtitle{Large-scale {PDE}-constrained Optimization}
vol. \bseriesno{30}.
\bpublisher{Springer},
\blocation{Berlin, Heidelberg}
(\byear{2003})
\end{bbook}
\endbibitem

\bibitem{Biros_05}
\begin{barticle}
\bauthor{\bsnm{Biros}, \binits{G.}},
\bauthor{\bsnm{Ghattas}, \binits{O.}}:
\batitle{Parallel {L}agrange-{N}ewton-{K}rylov-{S}chur methods for
  {PDE}-constrained optimization. {Parts I-II}}.
\bjtitle{SIAM J. Sci. Comput.}
\bvolume{27},
\bfpage{687}--\blpage{738}
(\byear{2005})
\end{barticle}
\endbibitem

\bibitem{Laird_05}
\begin{botherref}
\oauthor{\bsnm{Laird}, \binits{C.D.}},
\oauthor{\bsnm{Biegler}, \binits{L.T.}},
\oauthor{\bparticle{van} \bsnm{Bloemen~Waanders}, \binits{B.}},
\oauthor{\bsnm{Bartlett}, \binits{R.A.}}:
Time dependent contaminant source determination for municipal water networks
  using large scale optimization.
ASCE J. Water Res. Mgt. Plan.,
125--134
(2005)
\end{botherref}
\endbibitem

\bibitem{Hintermuller_05}
\begin{barticle}
\bauthor{\bsnm{Hintermuller}, \binits{M.}},
\bauthor{\bsnm{Vicente}, \binits{L.N.}}:
\batitle{Space mapping for optimal control of partial differential equations}.
\bjtitle{SIAM J. Opt.}
\bvolume{15},
\bfpage{1002}--\blpage{1025}
(\byear{2005})
\end{barticle}
\endbibitem

\bibitem{Hazra_06}
\begin{barticle}
\bauthor{\bsnm{Hazra}, \binits{S.B.}},
\bauthor{\bsnm{Schulz}, \binits{V.}}:
\batitle{Simultaneous pseudo-timestepping for aerodynamic shape optimization
  problems with state constraints}.
\bjtitle{SIAM J. Sci. Comput.}
\bvolume{28},
\bfpage{1078}--\blpage{1099}
(\byear{2006})
\end{barticle}
\endbibitem

\bibitem{Biegler_07}
\begin{bbook}
\beditor{\bsnm{Biegler}, \binits{L.T.}},
\beditor{\bsnm{Ghattas}, \binits{O.}},
\beditor{\bsnm{Heinkenschloss}, \binits{M.}},
\beditor{\bsnm{Keyes}, \binits{D.}},
\beditor{\bparticle{van} \bsnm{Bloemen~Waanders}, \binits{B.}} (eds.):
\bbtitle{Real-time {PDE}-constrained Optimization}
vol. \bseriesno{3}.
\bpublisher{SIAM},
\blocation{Computational Science and Engineering}
(\byear{2007})
\end{bbook}
\endbibitem

\bibitem{Borzi_07}
\begin{barticle}
\bauthor{\bsnm{Borzi}, \binits{A.}}:
\batitle{High-order discretization and multigrid solution of elliptic nonlinear
  constrained optimal control problems}.
\bjtitle{J. Comp. Applied Math}
\bvolume{200},
\bfpage{67}--\blpage{85}
(\byear{2007})
\end{barticle}
\endbibitem

\bibitem{Hinze_09}
\begin{bbook}
\bauthor{\bsnm{Hinze}, \binits{M.}},
\bauthor{\bsnm{Pinnau}, \binits{R.}},
\bauthor{\bsnm{Ulbrich}, \binits{M.}},
\bauthor{\bsnm{Ulbrich}, \binits{S.}}:
\bbtitle{Optimization with PDE Constraints}.
\bpublisher{Springer},
\blocation{Dordrecht}
(\byear{2009})
\end{bbook}
\endbibitem

\bibitem{Biegler_11}
\begin{bbook}
\beditor{\bsnm{Biegler}, \binits{L.}},
\beditor{\bsnm{Biros}, \binits{G.}},
\beditor{\bsnm{Ghattas}, \binits{O.}},
\beditor{\bsnm{Heinkenschloss}, \binits{M.}},
\beditor{\bsnm{Keyes}, \binits{D.}},
\beditor{\bsnm{Mallick}, \binits{B.}},
\beditor{\bsnm{Marzouk}, \binits{Y.}},
\beditor{\bsnm{Tenorio}, \binits{L.}},
\beditor{\bparticle{van} \bsnm{Bloemen~Waanders}, \binits{B.}},
\beditor{\bsnm{Willcox}, \binits{K.}} (eds.):
\bbtitle{Large-scale Inverse Problems and Quantification of Uncertainty}.
\bpublisher{John Wiley and Sons},
\blocation{Computational Statistics}
(\byear{2011})
\end{bbook}
\endbibitem

\bibitem{frontier_in_pdeco}
\begin{bbook}
\beditor{\bsnm{Antil}, \binits{H.}},
\beditor{\bsnm{Kouri}, \binits{D.P.}},
\beditor{\bsnm{Lacasse}, \binits{M.}},
\beditor{\bsnm{Ridzal}, \binits{D.}} (eds.):
\bbtitle{Frontiers in {PDE}-constrained Optimization}.
\bpublisher{Springer},
\blocation{New York, NY}
(\byear{2018})
\end{bbook}
\endbibitem

\end{thebibliography}

\end{document}